\documentclass[leqno,10pt]{amsart}
\usepackage{amsmath,amssymb,amsthm,amsfonts}
\usepackage{mathrsfs}

\usepackage{color}

\newtheorem{lemma}{Lemma}[section]
\newtheorem{theorem}{Theorem}[section]

\newtheorem{remark}{Remark}[section]

\numberwithin{equation}{section}

\newcommand{\dis}{\displaystyle}

\newcommand{\R}{\mathbb{R}}

\newcommand{\CE}{\mathcal{E}}

\begin{document}

\title[The Vlasov-Poisson-Boltzmann system for the whole range of cutoff soft potentials]{The Vlasov-Poisson-Boltzmann System for the Whole Range of Cutoff Soft Potentials}

\author[Qinghua Xiao]{Qinghua Xiao}
\address[Qinghua Xiao]{\newline School of Mathematics and Statistics, Wuhan University, Wuhan 430072, China}
\email{pdexqh@hotmail.com}

\author[Linjie Xiong]{Linjie Xiong}
\address[Linjie Xiong]{\newline School of Mathematics and Statistics, Wuhan University, Wuhan 430072, China}
\email{xlj@whu.edu.cn}

\author[Huijiang Zhao]{Huijiang Zhao}
\address[Huijiang Zhao]{\newline School of Mathematics and Statistics, Wuhan University, Wuhan 430072, China}
\email{hhjjzhao@hotmail.com}
\date{\today}

%\thanks{}

\begin{abstract}
The dynamics of dilute electrons can be modeled by the fundamental one-species Vlasov-Poisson-Boltzmann system which describes mutual interactions of the electrons through collisions in the self-consistent electrostatic field. For cutoff intermolecular interactions, although there are some progress on the construction of global smooth solutions to its Cauchy problem near Maxwellians recently, the problem for the case of very soft potentials remains unsolved. By introducing a new time-velocity weighted energy method and based on some new optimal temporal decay estimates on the solution itself and some of its derivatives with respect to both the spatial and the velocity variables, it is shown in this manuscript that the Cauchy problem of the one-species Vlasov-Poisson-Boltzmann system for all cutoff soft potentials does exist a unique global smooth solution for general initial perturbation which is unnecessary to satisfy the neutral condition imposed in \cite{Duan-Yang-Zhao-M3AS-2013} for the case of cutoff moderately soft potentials but is assumed to be small in certain weighted Sobolev spaces. Our approach applies also to the case of cutoff hard potentials and thus provides a satisfactory global well-posedness theory to the one-species Vlasov-Poisson-Boltzmann system near Maxwellians for the whole range of cutoff intermolecular interactions in the perturbative framework.
\end{abstract}

\maketitle
\thispagestyle{empty}

\tableofcontents

\section{Introduction}
The dynamics of dilute electrons can be modeled by the fundamental one-species Vlasov-Poisson-Boltzmann system (called VPB system in the sequel for simplicity) which describes mutual interactions of the electrons through collisions in the self-consistent electrostatic field
\begin{equation} \label{VPB}
 \partial_tf+\xi \cdot\nabla_xf+\nabla_{x}\phi\cdot\nabla_{\xi}f=Q(f,f),
  \end{equation}
\begin{equation} \label{P}
 \Delta_x\phi(t,x)=\displaystyle\int_{\R^3}f(t,x,\xi) d\xi-n_b(x),\quad \lim\limits_{|x|\to+\infty}\phi(t,x)=0.
  \end{equation}
Here the unknown $f=f(t,x,\xi)\geq0$ is the density distribution function of electrons located at $x =(x_1,x_2,x_3)\in \R^3$ with velocity $\xi= (\xi_1, \xi_2, \xi_3)\in\R^3$ at time $t\geq 0$. The potential function $\phi=\phi(t,x)$ generating the self-consistent electrostatic field $\nabla_x\phi$ in \eqref{VPB} is coupled with $f(t,x,\xi)$ through the Poisson equation (\ref{P}) where $n_b(x)>0$ is the background charge which is assumed to be a positive constant in the rest of this manuscript denoting that the background charge is spatially homogeneous and in such a case, we can set $n_b(x)=1$ without loss of generality. The bilinear collision operator $Q(f,g)$ is defined by, cf. \cite{Cercignani-Illner-Pulvirenti-1994}, \cite{Glassey-1996}, \cite{Grad-1963}
\begin{equation}\label{Q}
Q(f,g)=\displaystyle\int_{\R^3}\int_{\mathbb{S}^2}|\xi-\xi_{\ast}|^{\gamma}q_0(\vartheta)
\big\{f(\xi_{\ast}')g(\xi')-f(\xi_{\ast})g(\xi)\big\} d\xi_{\ast}d\omega,
\end{equation}
where $(\xi,\xi_{\ast})$ and $(\xi',\xi_{\ast}')$, denoting velocities of two particles before and after their
collisions respectively, satisfy
\begin{equation*}
\xi'=\xi-[(\xi-\xi_{\ast})\cdot\omega]\omega,\quad
\xi'_{\ast}=\xi_{\ast}+[(\xi-\xi_{\ast})\cdot\omega]\omega,\quad \omega\in\mathbb{S}^2,
\end{equation*}
which follows from the conservation of momentum and kinetic energy
\begin{equation*}
 \xi+\xi_{\ast} =\xi'+\xi'_{\ast},\quad|\xi|^2+|\xi_{\ast}|^2 =|\xi'|^2+|\xi'_{\ast}|^2.
\end{equation*}
Consequently, the identity $ |\xi-\xi_{\ast}|=|\xi'-\xi'_{\ast}|$ holds.

The non-negative cross-section ${\mathcal{B}}(\xi-\xi_\ast,\vartheta)=|\xi-\xi_{\ast}|^{\gamma}q_0(\vartheta)$ in (\ref{Q}) depends only on the relative velocity $ |\xi  -\xi_{\ast}|$ and on the deviation angle $\vartheta$ given by $\cos\vartheta=(\xi-\xi_{\ast})\cdot\omega/|\xi-\xi_{\ast}|$. Throughout this manuscript, such a cross-section is assumed to satisfy Grad's angular cutoff assumption, cf. \cite{Grad-1963}
$$
0 \leq q_0(\vartheta)\leq C|\cos\vartheta|
$$
with $C>0$ being some positive constant. The exponent $\gamma\in (-3,1]$ is determined by the potential of intermolecular forces, which is classified into the soft potential case for $-3<\gamma<0$, the Maxwell molecular case for $\gamma=0$, and the hard potential case for $0<\gamma\leq1$ which includes the hard sphere model with $\gamma=1$ and $q_0(\vartheta)=C|\cos\vartheta|$. For the soft potentials, the case $-2\leq\gamma<0$ is called the moderately soft potentials while $-3<\gamma<-2$ is called the very soft potentials, cf. \cite{Villani-2002}.

The one-species VPB system \eqref{VPB}, \eqref{P} can be thought as a reduced model of the following two-species VPB system which describes the dynamics of two-species charged dilute particles (e.g., electrons and ions) under the influence of the interactions with themselves through collisions and their self-consistent electrostatic field
\begin{equation}\label{2-VPB}
\left\{
\begin{array}{l}
\partial_tf_++ \xi \cdot\nabla_x f_+-\nabla_x\phi\cdot\nabla_\xi f_+=Q(f_+,f_+)+Q(f_+,f_-),\\[2mm]
\partial_tf_-+ \xi \cdot\nabla_x f_-+\nabla_x\phi\cdot\nabla_{\xi}f_-=Q(f_-,f_+)+Q(f_-,f_-),
\end{array}
\right.
\end{equation}
\begin{equation}\label{2-P}
-\Delta\phi(t,x)=\int_{{\mathbb{R}^3}}(f_+(t,x,\xi)-f_-(t,x,\xi))d\xi,\quad \lim\limits_{|x|\to+\infty}\phi(t,x)=0.
\end{equation}
Here $f_\pm(t, x, \xi)\geq 0$ are the density distribution functions for the ions $(+)$ and electrons $(-)$ respectively, at time $t\geq 0$, position $x=(x_1,x_2,x_3)\in \mathbb{R}^3$, and  velocity $\xi=(\xi_1, \xi_2, \xi_3)\in \mathbb{R}^3$.
Here, in both \eqref{VPB}-\eqref{P} and \eqref{2-VPB}-\eqref{2-P}, all the physical parameters, such as the particle masses $m_\pm$, their charges $e_\pm$, together with some other involving constants such as the universal constant $4\pi$, etc., have been chosen to be unit for simplicity of presentation and also without loss of generality.

In physical situations the ion mass is usually much larger than the electron mass so that the electrons
move much faster than the ions. Thus, the ions are often described by a fixed ion background $n_b(x)$ and only the electrons move. For such a case, the two-species VPB system \eqref{2-VPB}-\eqref{2-P} can be reduced to the one-species VPB system \eqref{VPB}-\eqref{P}.

What we are interested in this paper is on the construction of global smooth solutions to the Cauchy problem of the VPB system \eqref{VPB}-\eqref{P} and \eqref{2-VPB}-\eqref{2-P} for cutoff intermolecular interactions. Notice that, as shown in \cite{Wang-JDE-2013} for the the two-species VPB system \eqref{2-VPB}-\eqref{2-P} for the hard sphere model, cf. also the corresponding result obtained in \cite{Wang-SIMA-2012} for the two-species Vlasov-Poisson-Landau system, the electrostatic field $\nabla_x\phi$ enjoys much better temporal decay estimates than the one-species VPB system \eqref{VPB}-\eqref{P} which is due to the cancelation of the different species of charged particles and since, as will be explained later, the temporal decay property of the electrostatic filed $\nabla_x\phi$ plays an essential role in establishing the global well-posedness theory of the VPB systems \eqref{VPB}-\eqref{P} and \eqref{2-VPB}-\eqref{2-P} in the perturbative framework, thus the problem on the global solvability of the one-species VPB system \eqref{VPB}-\eqref{P} is much harder then the two-species VPB system \eqref{2-VPB}-\eqref{2-P} and the main purpose of our present paper is concerned with the Cauchy problem of the one-species VPB system \eqref{VPB}, \eqref{P} with prescribed initial data
\begin{equation} \label{initial}
 f(0,x,\xi)=f_0(x,\xi)
\end{equation}
around the following normalized global Maxwellian
\begin{equation*}
  {\bf M}(\xi)=(2\pi)^{-\frac{3}{2}}\exp{\left(-\frac{|\xi|^2}{2}\right)}.
\end{equation*}
For this purpose, as in \cite{Grad-1963, Ukai-1986}, if we define the perturbation $u=u(t,x,\xi)$ by
\begin{equation*}
  f(t,x,\xi)={\bf M}+{\bf M}^{\frac{1}{2}}u(t,x,\xi),
\end{equation*}
then the Cauchy problem \eqref{VPB}, \eqref{P}, \eqref{initial} is reformulated as
\begin{equation}\label{u}
\begin{cases}
   \partial_tu+\xi\cdot\nabla_xu+\nabla_x\phi\cdot\nabla_{\xi}u-\frac{1}{2}\xi\cdot\nabla_x\phi u- \nabla_x\phi\cdot\xi {\bf M}^{\frac{1}{2}}-{\bf L}u=\Gamma(u,u),\\
   \triangle_x\phi(t,x)=\dis\int_{\R^3}{\bf M}^{\frac{1}{2}}(\xi)u(t,x,\xi)d\xi,\quad \lim\limits_{|x|\to+\infty}\phi(t,x)=0,\\
   u(0,x,\xi)=u_0(x,\xi)={\bf M}^{-\frac{1}{2}}(f_0-{\bf M}).
\end{cases}
\end{equation}
Here the linearized collision operator ${\bf L}$ and the nonlinear collision term $\Gamma$ are
defined by
\begin{equation*}
{\bf L}u={\bf M}^{-\frac{1}{2}}\left[{Q\left({\bf M},{\bf M}^{\frac{1}{2}}u\right)+ Q\left({\bf M}^{\frac{1}{2}}u,{\bf M}\right)}\right],
\end{equation*}
and
\begin{equation*}
\Gamma(u,u)={\bf M}^{-\frac{1}{2}}Q\left({\bf M}^{\frac{1}{2}}u,{\bf M}^{\frac{1}{2}}u\right),
\end{equation*}
respectively. It is well known, cf. \cite{Cercignani-Illner-Pulvirenti-1994, Glassey-1996, Grad-1963, Guo-ARMA-2003, Ukai-1986} that the linearized collision operator ${\bf L}$ is non-negative with its null space $\mathcal{N}$ being given by
\begin{equation*}
  {\mathcal{ N}}={\textrm{Span}}\left\{{\bf M}^{\frac{1}{2}}, ~\xi_i{\bf M}^{\frac{1}{2}}(1\leq i\leq3),\left(|\xi|^2-3\right){\bf M}^{\frac{1}{2}}\right\}.
\end{equation*}
Moreover, ${\bf L}$ can be decomposed as ${\bf L}= -\nu +K$ with
\begin{equation*}
\nu(\xi)=\iint_{\R^3\times{\mathbb{S}^2}}|\xi-\xi_{*}|^{\gamma}q_0(\vartheta){\bf M}(\xi_{*})d\omega d\xi_{*}\sim(1+|\xi|)^{\gamma}
\end{equation*}
and
\begin{equation*}
\begin{split}
Ku(\xi)=&\iint_{\R^3\times{\mathbb{S}^2}}|\xi-\xi_{*}|^{\gamma}q_0(\vartheta){\bf M}^{\frac12}(\xi_{*})
\left\{{\bf M}^{\frac{1}{2}}(\xi'_{*})u(\xi')+{\bf M}^{\frac12}(\xi')u(\xi'_{*})
-{\bf M}^{\frac12}(\xi)u(\xi_{*})\right\}d \omega d\xi_{*}\\
=&\int_{\R^3}K(\xi,\xi_{*})u(\xi_{*})d\xi_{*}.
\end{split}
\end{equation*}
Here and in the rest of this paper, $A\sim B$ means that there exists some generic positive constant $C>0$ such that $C^{-1}B\leq A\leq CB$.

If we define ${\bf P}$ as the orthogonal projection from $L^2\left(\R^3_\xi\right)$ to $\mathcal{N}$, then for any given function $u(t,x,\xi)\in L^2\left(\R^3_\xi\right)$, one has
\begin{eqnarray*}
{\bf P}u &=&{a(t,x){\bf M}^{\frac{1}{2}} + b(t,x)\cdot\xi{\bf M}^{\frac{1}{2}}+c(t,x)\left(|\xi|^2-3\right)}{\bf M}^{\frac{1}{2}},\\
a&=&\dis\int_{\R^3}{\bf M}^{\frac{1}{2}}ud\xi=\dis\int_{\R^3}{\bf M}^{\frac{1}{2}}{\bf P}u d\xi,\\
b_i&=&\dis\int_{\R^3}\xi_i{\bf M}^{\frac{1}{2}}u d\xi =\dis\int_{\R^3}\xi_i{\bf M}^{\frac{1}{2}}{\bf P}ud\xi,
\quad  i =1,2,3,\\
c&=&\dis\frac{1}{6}\dis\int_{\R^3}\left(|\xi|^2-3\right){\bf M}^{\frac{1}{2}}u d\xi
=\frac{1}{6}\dis\int_{\R^3}\left(|\xi|^2-3\right){\bf M}^{\frac{1}{2}}{\bf P}ud\xi.
\end{eqnarray*}
Therefore, we have the following macro-micro decomposition with respect to a given global Maxwellian, cf. \cite{Guo-IUMJ-2004}:
\begin{equation*}
 u(t,x,\xi)={\bf P}u(t,x,\xi)+\{{\bf I}-{\bf P}\}u(t,x,\xi)\equiv u_1(t,x,\xi)+u_2(t,x,\xi),
 \end{equation*}
where ${\bf I}$ denotes the identity operator, ${\bf P}u\equiv u_1$ and $\{{\bf I}-{\bf P}\}u\equiv u_2$ are called the macroscopic and the microscopic component of $u$, respectively.

Under the above micro-macroscopic decomposition, ${\bf L}$ is locally coercive, cf. \cite{Cercignani-Illner-Pulvirenti-1994, Glassey-1996, Grad-1963, Guo-ARMA-2003}, in the sense that
\begin{equation}\label{coercive}
-\langle u,{\bf L}u\rangle \gtrsim|\{{\bf I}-{\bf P}\}u|_{\nu}^2=\left\|\sqrt{\nu} u_2\right\|_{L^2\left(\R^3_\xi\right)},\quad
\nu(\xi)\sim(1+|\xi|)^{\gamma}.
\end{equation}
Here $\langle\cdot,\cdot\rangle$ denotes the inner product in $L^2(\R^3_{\xi})$, $A\gtrsim B$ means that there is a generic positive constant $C>0$ such that $A\geq CB$ and $A\lesssim B$ can be defined similarly.

The problem on the global solvability of the VPB system \eqref{VPB}-\eqref{P} and/or \eqref{2-VPB}-\eqref{2-P} near Maxwellians have been studied by many authors and to explain the main difficulties encountered and the main problem we want to study, we will outline the main ideas developed recently on the construction of global smooth solutions to some complex kinetic equations and sketch some former results closely related to the theme of this manuscript. In fact in the perturbative context, there have been extensive investigations recently on the construction of global solutions to some complex kinetic equations, such as the Vlasov-Poisson-Landau system \cite{Duan-Yang-Zhao-VPL, Guo-JAMS-2012, Lei-Xiong-Zhao-1VPL, Strain-Zhu-ARMA-2013, Wang-SIMA-2012}, the Vlasov-Poisson-Boltzmann system \cite{Duan-Liu-CMP-2013, Duan-Strain-ARMA-2011, Duan-Yang-SIMA-2010, Duan-Yamg-Zhao-JDE-2012, Duan-Yang-Zhao-M3AS-2013, Guo-CPAM-2002, Guo-Jang-CMP-2010, Wang-JDE-2013, Xiao-Xiong-Zhao-JDE-2013, Xiao-Xiong-Zhao-Science China-2014, Yang-Yu-CMP-2011, Yang-Yu-Zhao-ARMA-2006, Yang-Zhao-CMP-2006, Zhang-JDE-2009},  the Vlasov-Maxwell-Landau system \cite{Duan-AIHP-AN-2014, Lei-Zao-2VML}, and the Vlasov-Maxwell-Boltzmann system \cite{Duan-SIMA-2011, Duan-Liu-Yang-Zhao-KEM-2013, Duan-Strain-CPAM-2011, Guo-Invent-Math-2003, Guo-Strain-CMP-2012, Jang-ARMA-2009, Lei-Zhao-2013-1VMB, Strain-CMP-2006}, etc., based on the energy method introduced in \cite{Guo-IUMJ-2004, Liu-Yang-Yu-PhD-2004, Liu-Yu-CMP-2004} for the Boltzmann equation. The main difficulty involved in dealing with such a type of problem lies in how to control the possible growth of the solutions induced by nonlinearity of the equations under consideration, especially
\begin{itemize}
\item[$\bullet$]the degeneration of the dissipation \eqref{coercive} at large velocity $\xi$ for the linearized Boltzmann collision operator ${\bf L}$ for soft potentials $-3<\gamma<0$ or the degeneration of the corresponding dissipative estimate at large velocity $\xi$ for the linearized Landau collision operator corresponding to the Coulomb potential, cf. \cite{Degond-Lemou-ARMA-1997, Guo-CMP-2002, Guo-JAMS-2012, Lions-1994, Strain-Guo-ARMA-2008, Strain-Zhu-ARMA-2013, Wang-SIMA-2012};
\item[$\bullet$] the velocity-growth of the nonlinear term related to the Lorenz or the electrostatic force with the velocity-growth rate $|\xi|$, such as the term $\xi\cdot \nabla_x\phi u$ in \eqref{u}$_1$.
\end{itemize}
An important progress in this direction is due to Y. Guo's work on the two-species Vlasov-Poisson-Landau system \cite{Guo-JAMS-2012} in a periodic box for initial data with small weighted $H^2-$norms. The main ideas developed in \cite{Guo-JAMS-2012} are the following;
\begin{itemize}
\item[$\bullet$] a new exponential weight of electric potential $e^{\pm\phi}$ is introduced to cancel the growth of the velocity in the nonlinear term $\mp\xi\cdot\nabla_x\phi u_\pm$,
\item[$\bullet$] a new time and velocity weighted energy method is designed which is based on a new velocity weight
$$
\overline{w}_{\ell-|\alpha|-|\beta|}(\xi)=\langle \xi\rangle^{-(\gamma+1)(\ell-|\alpha|-|\beta|)},\quad \langle \xi\rangle=\sqrt{1+|\xi|^2},\quad \ell\geq |\alpha|+|\beta|
$$
to capture the weak velocity diffusion in the linearized Landau kernel for the case of $-3\leq \gamma<-2$ and a decay of the electrostatic field $\nabla_x\phi$ to close the energy estimate.
\end{itemize}
Such a result is extended recently by R. Strain and K.-Y. Zhu in \cite{Strain-Zhu-ARMA-2013} and Y.-J. Wang in \cite{Wang-SIMA-2012} respectively to the case of the whole space by different approaches. The analysis in \cite{Strain-Zhu-ARMA-2013} is to combine the energy estimates with the linear decay analysis which requires some smallness assumption on the $Z_1$-norm of the initial data, while the analysis in \cite{Wang-SIMA-2012} removed such an assumption by an interesting observation that what one needed is just the temporal decay rate of the electrostatic field $\nabla_x\phi$ rather the whole solutions. Based on such an observation, Y.-J. Wang decoupled the two-species Vlasov-Poisson-Landau system into two independent subsystems (one is the Landau system and the other one is a system almost like one-species Vlasov-Poisson-Landau system) to yield better temporal decay estimates on the electrostatic potential $\phi(t,x)$, which is due to the cancelation effect between different species of charged particles, cf. \cite{Wang-SIMA-2012} for details. In all these three manuscripts, the fact that
\begin{equation}\label{difference}
\left\|\left(\partial_t\phi(t),\nabla_x\phi(t)\right)\right\|_{L^\infty\left({\mathbb{R}}^3_x\right)}\in L^1\left({\mathbb{R}}^+\right)
\end{equation}
plays an essential role in their analysis.

We note, however, that the above argument can not be adopted directly to deal with the one-species VPB system \eqref{VPB}-\eqref{P} since the temporal decay estimates on the corresponding linearized solution operator performed in \cite{Duan-Yamg-Zhao-JDE-2012, Duan-Yang-Zhao-M3AS-2013} tells us that, even under the neutral condition on the initial perturbation $u_0(x,\xi)$
\begin{equation}\label{neutral}
\int_{{\mathbb{R}}_x^3}\int_{{\mathbb{R}}_\xi^3}{\bf M}^{\frac 12}(\xi)u_0(x,\xi)d\xi dx=0,
\end{equation}
one can only deduce that $\partial_t\phi(t,x)$ decays at most like $\left\|\partial_t\phi(t)\right\|_{L^\infty\left({\mathbb{R}}^3_x\right)}\leq O(1)(1+t)^{-1}$. Consequently one can not hope that the estimate \eqref{difference} holds and the arguments developed in \cite{Guo-JAMS-2012, Strain-Zhu-ARMA-2013, Wang-SIMA-2012}, which have been proved to be effective in dealing with the two-species Vlasov-Poisson-Landau system, can not be used any longer to treat the one-species VPB system \eqref{VPB}-\eqref{P}.

Even so, for the case of cutoff hard potentials, although the arguments developed in \cite{Guo-JAMS-2012, Strain-Zhu-ARMA-2013, Wang-SIMA-2012} can not be used directly, one can use the smallness of $\|\partial_t\phi\|_{L^{\infty}}$ and the stronger dissipation of linearized Boltzmann operator ${\bf L}$ for hard potential case, {\it i.e.} the inequality
\begin{equation*}
\mathcal{E}_{N,l}(t)\lesssim \mathcal{D}_{N,l}(t)+\left\|\nabla_x\phi\right\|^2+\|(a,b,c)\|^2
\end{equation*}
with the energy functional $\mathcal{E}_{N,l}(t)$ satisfying
$$
\mathcal{E}_{N,l}(t)\sim\displaystyle\sum_{|\alpha|\leq N<l}\left\|
\partial^{\alpha}\nabla\phi(t)\right\|^2+\displaystyle\sum_{|\alpha|+|\beta|\leq  N<l}\left\|\langle\xi\rangle^{l-|\alpha|-|\beta|}
\partial^{\alpha}_{\beta}u(t)\right\|^2
$$
and the corresponding energy dissipation rate functional $\mathcal{D}_{N,l}(t)$ satisfying
$$
\mathcal{D}_{N,l}(t)\sim\displaystyle\sum_{1\leq|\alpha|\leq N<l}\left\|
\partial^{\alpha}\nabla\phi(t)\right\|^2+\displaystyle\sum_{1\leq|\alpha|\leq N<l}\left\|
\partial^{\alpha}u_1(t)\right\|^2+\displaystyle\sum_{|\alpha|+|\beta|\leq  N<l}\left\|\langle\xi\rangle^{l-|\alpha|-|\beta|}
\partial^{\alpha}_{\beta}u_2(t)\right\|^2_{\nu},
$$
to absorb the term $\|\partial_t\phi\|_{L^{\infty}}\mathcal{E}_{N,l}(t)$ except the term consisting of the $L^2-$norm of macroscopic terms $(a,b,c)$ and the electrostatic field $\nabla_x\phi$. Then, motivated by the arguments developed in \cite{Duan-Strain-ARMA-2011, Duan-Strain-CPAM-2011, Duan-Ukai-Yang-Zhao-CMP-2008, Yang-Yu-CMP-2011} to deduce the optimal temporal decay estimates on the solutions to the Boltzmann type equations with hard potential intermolecular interactions, one can deduce the desired optimal temporal decay estimates on the solutions of the one-species VPB system \eqref{u} and based on these optimal temporal decay estimates, one can thus close the whole analysis. Here the fact that for the hard potentials case, the arguments developed in \cite{Duan-Strain-ARMA-2011, Duan-Strain-CPAM-2011, Duan-Ukai-Yang-Zhao-CMP-2008, Yang-Yu-CMP-2011} can be used to get the optimal temporal decay estimates on both the solution $u(t,x,\xi)$ itself and some orders of its derivatives with respect to the spatial variable plays an essential role in the analysis. See \cite{Xiao-Xiong-Zhao-Science China-2014} for details (It is worth to pointing out that although \cite{Xiao-Xiong-Zhao-Science China-2014} deals with the non-cutoff hard potential case, the argument used there can be applied also to the cutoff hard potentials directly).

For cutoff soft potentials, the story is quite different. Although the arguments developed in \cite{Strain-KRM-2012, Strain-Guo-ARMA-2008} to deduce the temporal decay estimates on the solutions of the Boltzmann-type equations with soft potentials, which are based on a time-velocity splitting argument developed in \cite{Strain-Guo-ARMA-2008}, the temporal decay estimates on the solution operator of the corresponding linearized system, the energy method together with the Duhamel principle, can also be used to deduce the desired decay estimates on the solution $u(t,x,\xi)$ of the one-species VPB system \eqref{u} together with its certain orders of derivatives with respect to the $x-$variable, which are the key point to yield the global solvability result, one encounters the problem of the loss of $1-p$ order of the corresponding decay rates which is mainly due to the following inequality
\begin{equation*}
    \int_0^te^{-\lambda(t^p-\tau^p)}(1+\tau)^{-m}d\tau\lesssim (1+t)^{-m-p+1},\quad 0<p<1.
  \end{equation*}
A directly consequence of such a fact is that one can not hope to use the arguments developed in \cite{Strain-KRM-2012, Strain-Guo-ARMA-2008} to deduce the desired optimal temporal decay estimates on certain orders of derivatives of $u(t,x,\xi)$ with respect to $x-$variables as for the case of hard potentials and hence the argument which has been proved to be effective for the one-species VPB system \eqref{u} in \cite{Xiao-Xiong-Zhao-Science China-2014} for the hard potential case can not be used any longer. To overcome such a difficulty, the main idea in \cite{Duan-Yamg-Zhao-JDE-2012, Duan-Yang-Zhao-M3AS-2013} is to introduce a new time-velocity weighted energy method based on the following new time-velocity weight function
\begin{equation}\label{weight-DYZ}
\widetilde{w}_{\ell-|\beta|}(t,\xi)=
\langle\xi\rangle^{\gamma(|\beta|-\ell)}e^{\frac{q\langle\xi\rangle^2}{(1+t)^\vartheta}},
\quad \vartheta>0,\ q>0,\ \ell\geq |\beta|,
\end{equation}
where the role of the exponential factor $w^e_{\ell-|\beta|}(t,\xi)=e^{\frac{q\langle\xi\rangle^2}{(1+t)^\vartheta}}$ of the weight function $\widetilde{w}_{\ell-|\beta|}(t,\xi)$ is to yield an extra dissipative term like $(1+t)^{-(1+\vartheta)}\left\|\langle\xi\rangle w_{\ell-|\beta|}(t,\xi)\partial^\alpha_\beta u_2\right\|^2$, but, as pointed out in \cite{Duan-Yang-Zhao-M3AS-2013}, since, unlike the weight function $\overline{w}_{\ell-|\alpha|-|\beta|}(\xi)$, the corresponding algebraic factor $w_{|\beta|-\ell}^a(\xi)=\langle\xi\rangle^{\gamma(|\beta|-\ell)}$ varies only when the order of the $\xi-$derivatives changes, it indeed produces an additional difficulty on the nonlinear term $\nabla_x\phi\cdot\nabla_{\xi}u_2$ in the presence of the self-consistent electrostatic field $\nabla_x\phi$ for the one-species VPB system \eqref{VPB}-\eqref{P} with soft potentials. To obtain the velocity weighted derivative estimate on such a nonlinear term, one should put an extra negative-power function $\langle\xi\rangle^{\gamma}$ in front of $\nabla_{\xi}u_2$ so that the velocity growth $\langle\xi\rangle^{-\gamma}$ comes up to have a balance. Then, only if $-2\leq\gamma<0$, it is fortunate that the extra dissipative term mentioned above which contains the second-order moment $\langle\xi\rangle^{2}$ can be used to control such term provided that the electrostatic field $\nabla_x\phi$ decays sufficiently fast. Based on these ideas, the case for the cutoff moderately soft potentials, i.e. $-2\leq \gamma<0$ under Grad's cutoff assumption \cite{Grad-1963}, was solved in \cite{Duan-Yang-Zhao-M3AS-2013} provided that the initial perturbation $u_0(x,\xi)$ is assumed to satisfy the neutral condition \eqref{neutral} in addition to some usual smallness assumptions. It is worth to pointing out that the neutral condition \eqref{neutral} is imposed to guarantee that the electrostatic field $\nabla_x\phi$ decays sufficiently fast while the restriction on the range of $\gamma\in[-2,0)$ is an essential requirement of the argument used in \cite{Duan-Yang-Zhao-M3AS-2013}. Thus it is natural to ask the following two questions:
\begin{itemize}
\item[$\bullet$] Firstly, for the case of $-2\leq \gamma<0$, does similar global solvability result as obtained in \cite{Duan-Yang-Zhao-M3AS-2013} still hold even without the neutral condition \eqref{neutral}?

\item[$\bullet$] Secondly, for cutoff intermolecular interactions, how to deal with the case of the very soft potentials, i.e. $-3<\gamma<-2$, under Grad's cutoff assumption \cite{Grad-1963}?
\end{itemize}
For the first problem, partial result has been obtained in \cite{Xiao-Xiong-Zhao-JDE-2013} for the case of $-1\leq \gamma<0$ which is based on the same weight function $\widetilde{w}_{\ell-|\beta|}(t,\xi)$ defined in \eqref{weight-DYZ} and the main idea there is to deduce the almost optimal decay estimates on the $L^2-$norm of certain higher order $x-$derivatives of $u(t,x,\xi)$ and $\nabla_x\phi(t,x)$. And the main purpose of our present paper is to give a positive answer to the above two questions for the whole range of cutoff soft potentials without the neutral condition \eqref{neutral}.

Now we turn to sate our main result. To this end, motivated by \cite{Duan-Yamg-Zhao-JDE-2012, Duan-Yang-Zhao-M3AS-2013}, we introduce the following mixed time-velocity weight function
\begin{equation}\label{weight}
 w_{\ell}(t,\xi)=\langle\xi\rangle^{\frac{\gamma}{2}\ell}e^{\frac{q\langle\xi\rangle^2}{(1+t)^{\vartheta}}},
\end{equation}
where $\ell\in \R,~0<q\ll1$. We also define the temporal energy functional $\mathcal{E}_{\infty}(t)$ as follows
\begin{equation}\label{Energy}
  \begin{split}
         \mathcal{E}_{\infty}(t)=&\sup_{0\leq\tau\leq t}\left\{\sum_{(\alpha,\beta)\in U_{\alpha,\beta}}
     (1+\tau)^{r_{|\alpha|}}\left(\|w_{|\beta|-\ell}\partial^{\alpha}_{\beta}u\|^2
     +\|\nabla_x^{|\alpha|+1}\phi\|^2\right)\right.\\
     &\left.+\sum_{|\alpha|\leq N-2}(1+\tau)^{-\left(|\alpha|+\frac{3}{2}\right)+1-p}\left\|\partial^{\alpha}u_2\right\|^2
     +(1+\tau)^{-N+\frac{1}{2}+2(1-p)}\left\|\nabla_x^{N-2}\nabla_{\xi}u_2\right\|^2\right\},
  \end{split}
\end{equation}
where
\begin{equation}\label{decay_rates}
r_{|\alpha|}=\left\{
\begin{array}{ll}
|\alpha|+\frac{1}{2}, \quad& \text{{when}}\ (\alpha,\beta)\in U_{\alpha,\beta}^{\text{low}},\\[2mm]
\frac{N+2|\alpha|}{3}-\frac{1}{2},\quad& \text{{when}}\ (\alpha,\beta)\in U_{\alpha,\beta}^{\text{high}}
\end{array}
\right.
\end{equation}
with
\begin{eqnarray*}
U_{\alpha,\beta}&=&\left\{(\alpha,\beta)\mid|\alpha|+|\beta|\leq N\right\}
 =U_{\alpha,\beta}^{\text{low}}\cup U_{\alpha,\beta}^{\text{high}},\\
U_{\alpha,\beta}^{\text{low}}&\equiv& \left\{(\alpha,\beta)\mid|\alpha|+|\beta|\leq N-1\right\}\cup
 \left\{(\alpha,\beta)\mid|\alpha|+|\beta|=N, |\alpha|<N-2\right\},\\
U_{\alpha,\beta}^{\text{high}}&\equiv& \left\{(\alpha,\beta)\mid|\alpha|+|\beta|=N,N-2\leq|\alpha|\leq N\right\}.
\end{eqnarray*}
With the above preparations in hand, the main result of this paper is stated as follows. Some notations will be explained at the end of this section.
\begin{theorem}\label{Thm}
   Assume that $-3< \gamma<0$ and $N\geq 4$ and let $\frac{3}{4}< p<1$, $0<\vartheta\leq\frac{1}{9}$ be two given constants. For $l_1=\frac{2N-1}{1-p}$, $l_2>N+\frac{1}{2}$, and $\ell\geq\max\left\{ l_1, \frac{2(2p-1)l_2}{4p-3}+1,l_2-\frac{2}{\gamma}\right\}$, if we assume further that
   $$
   f_0(x,\xi)={\bf M}(\xi)+{\bf M}^{\frac{1}{2}}(\xi)u_0(x,\xi)\geq 0
   $$
   and that
   $$
   \epsilon_0=\dis\sum_{|\alpha|+|\beta|\leq N}\left\|w_{|\beta|-\ell,q}\partial^{\alpha}_{\beta}u_0\right\|^2+\|\langle\xi\rangle^{\frac{-l_2\gamma}{2}} u_0\|_{Z_1}^2
   $$
   is sufficiently small, then the Cauchy problem \eqref{u} admits a unique global solution $u(t,x,\xi)$ satisfying $f(t,x,\xi)={\bf M}(\xi)+{\bf M}^{\frac{1}{2}}(\xi)u(t,x,\xi)\geq0$ and
  \begin{equation}\label{decay}
    \CE_{\infty}(t)\lesssim \epsilon_0.
  \end{equation}
\end{theorem}
\begin{remark}
  We now give several remarks concerning Theorem \ref{Thm}:
  \begin{itemize}
  \item Theorem 1.1 covers the case of all cutoff soft potentials without the neutral condition \eqref{neutral}, such a result together with the result obtained in \cite{Xiao-Xiong-Zhao-Science China-2014} provide a satisfactory global well-posedness theory on the global solvability of the Cauchy problem of the one-species VPB system \eqref{u} near a given global Maxwellian for the whole range of cutoff intermolecular interactions in the perturbation framework.
  \item The argument used in this paper applies also to the Cauchy problem of the two-species VPB system \eqref{2-VPB}-\eqref{2-P} for the whole range of cutoff intermolecular interactions and moreover, based on the observation of Y.-J. Wang \cite{Wang-JDE-2013} for the two-species VPB system \eqref{2-VPB}-\eqref{2-P} for the hard sphere model, one can hope that the electrostatic field $\nabla_x\phi(t,x)$ together with the solutions $u_\pm(t,x,\xi)={\bf M}^{-1/2}(\xi)\left(f_\pm(t,x,\xi)-{\bf M}(\xi)\right)$ can have better decay estimates.
    \item It is worth to pointing out that, the analysis in \cite{Duan-Yamg-Zhao-JDE-2012, Duan-Yang-Zhao-M3AS-2013, Xiao-Xiong-Zhao-JDE-2013} is based on the weight function $\widetilde{w}_{\ell-|\beta|}(t,\xi)$ given by \eqref{weight-DYZ}, while in this paper we use the weight function $w_{\ell-|\beta|}(t,\xi)$ $(|\beta|\leq \ell)$ defined in \eqref{weight}, it is easy to see that the smallness conditions we imposed on the initial perturbation $u_0(x,\xi)$ in this paper are weaker than those imposed in \cite{Duan-Yamg-Zhao-JDE-2012, Duan-Yang-Zhao-M3AS-2013, Xiao-Xiong-Zhao-JDE-2013}, while from the estimate \eqref{decay} together with \eqref{Energy} and \eqref{decay_rates}, we can deduce that
  \begin{equation}\label{lower_decay_rates}
   \left\|w_{|\beta|-\ell}\partial^{\alpha}_{\beta}u\right\|^2
     +\left\|\nabla_x^{|\alpha|+1}\phi\right\|^2\lesssim (1+t)^{-|\alpha|-\frac 12},\quad (\alpha,\beta)\in  U_{\alpha,\beta}^{\text{low}}.
  \end{equation}
        These temporal decay estimates \eqref{lower_decay_rates} are optimal in the sense that they coincide with those rates given in Lemma \eqref{lem.-decay} at the level of linearization. Such a result improves the almost optimal temporal decay rates for $L^2-$norm of certain higher order derivatives of the solution $u(t,x,\xi)$ to the one-species VPB system \eqref{u} with respect to $x-$variables obtained in \cite{Xiao-Xiong-Zhao-JDE-2013} to optimal.
    \item In our main result, we do not ask the initial perturbation $u_0(x,\xi)$ to satisfy the neutral condition \eqref{neutral}, even for the the case when such a condition is assumed to hold further, the arguments used in this manuscript can be adapted to yield an improved result. In fact, one can get a similar result by replacing the temporal decay rates $r_{|\alpha|}$ in \eqref{Energy} by $\tilde{r}_{|\alpha|}=r_{|\alpha|}+1$ and such a result improves the temporal decay result obtained in \cite{Duan-Yang-Zhao-M3AS-2013} for the case of moderately soft potentials.
    \item Our analysis also shows that, by a time-velocity splitting method as in \cite{Duan-Liu-CMP-2013} and \cite{Strain-KRM-2012}, for the microscopic part $u_2$, one can deduce the following improved temporal decay estimates
  \begin{equation}\label{micro_decay}
  \left\|\partial^\alpha u_2(t)\right\|^2\lesssim(1+t)^{-\left(|\alpha|+\frac 32\right)+(1-p)},\quad \left\|\nabla^{N-2}_x\nabla_{\xi}u_2(t)\right\|^2\lesssim (1+t)^{-\left(N-2+\frac{3}{2}\right)+2(1-p)}
  \end{equation}
        hold for any $\frac{3}{4}<p<1$ and $|\alpha|\leq N-2$. The temporal decay estimates \eqref{micro_decay} is almost optimal if we choose $1-p$ sufficiently small, see \eqref{A-u2-opt1.} and \eqref{A-u2-opt2.} for details.
  \item This paper is concerned with the one species VPB system \eqref{VPB}-\eqref{P} for cutoff intermolecular interactions. For the non-cutoff case, by combining the argument employed in \cite{Duan-Liu-CMP-2013} to treat the two-species VPB system \eqref{2-VPB}-\eqref{2-P} for non-cutoff intermolecular interactions under neutral condition imposed on the initial perturbation with the method used in this paper, a similar global solvability result can be obtained for the one-species VPB system \eqref{VPB}-\eqref{P} even without the neutral condition.
\end{itemize}
\end{remark}
Now we outline the main ideas used in this manuscript to deduce our main result. Our analysis is also based on an elaborated weighted energy method and our main observations are as follows: First of all, similar to that of \cite{Duan-Yamg-Zhao-JDE-2012, Duan-Yang-Zhao-M3AS-2013, Duan-Yang-Zhao-VPL}, our introduction of the new time-velocity weight function $w_{\ell-|\beta|}(t,\xi)$ $(|\beta|\leq \ell)$ defined by \eqref{weight}, especially the exponential factor
$w^e_{\ell}(t,\xi)=e^{\frac{q\langle\xi\rangle^2}{(1+t)^{\vartheta}}}$ will induce an extra dissipative term like
\begin{equation}\label{extra_dissipative}
(1+t)^{-1-\vartheta}\left\|\langle\xi\rangle w_{\ell-|\beta|}\partial^\alpha_\beta u_2\right\|^2,
\end{equation}
while with the algebraic factor $w^a_{\ell-|\beta|}(t,\xi)=\langle\xi\rangle^{\frac{\gamma}{2}(|\beta|-\ell)}$ $(|\beta|\leq \ell)$ of the new weight function we introduced, the problem encountered in \cite{Duan-Yang-Zhao-M3AS-2013} on the nonlinear term $\nabla_x\phi\cdot\nabla_{\xi}u_2$ when the weight function is chosen as $\widetilde{w}_{\ell-|\beta|}(t,\xi)$ given by  \eqref{weight-DYZ} and introduced in \cite{Duan-Yamg-Zhao-JDE-2012, Duan-Yang-Zhao-M3AS-2013}, which leads to the restriction of $\gamma$ to $-2\leq \gamma<0$, is no longer a problem. In fact, to obtain the desired weighted estimates on the terms involving the mixed spatial and velocity derivatives of such a nonlinear term, one should put an extra negative-power function $\langle\xi\rangle^{\gamma/2}$ in front of $\nabla_{\xi}u_2$ and consequently only the velocity growth $\langle\xi\rangle^{-\gamma/2}$ comes up to have a balance which can be bounded by the second-order moment $\langle\xi\rangle^{2}$ for the whole range of $\gamma>-3$. However, since the new weight function $w_{\ell-|\beta|}(t,\xi)$ we introduced in \eqref{weight} satisfies
$$
\left|w_{|\beta|-\ell}(t,\xi)\right|^2=w_{|\beta|-\ell}(t,\xi)\langle\xi\rangle^{\frac{\gamma}{2}} w_{|\beta-e_i|-\ell}(t,\xi),
$$
another trouble arises when one deals with term involving the mixed spatial and velocity derivatives of the linear transport term $\xi\cdot \nabla_xu_2$, which now can only be controlled as follows
\begin{equation*}
\begin{split}
   &\sum_{|\beta_1|=1}\left(\partial_{\beta_1}\xi\cdot\partial^{\alpha}_{\beta-\beta_1}\nabla_xu_2,
   w_{|\beta|-\ell}^2\partial^{\alpha}_{\beta}u_2\right)\\
   \lesssim&\eta\left\|w_{|\beta|-\ell}\partial^{\alpha}_{\beta}u_2\right\|_{\nu}^2
   +C_{\eta}\left\|w_{|\beta-e_i|-\ell}\partial^{\alpha+e_i}_{\beta-e_i}u_2\right\|^2.
 \end{split}
\end{equation*}
In fact, although the first term in the right hand side of the above inequality can be absorbed by the coercive estimate \eqref{coercive} of the linearized Boltzmann collision operator ${\bf L}$, since for soft potentials, the collision frequency $\nu(\xi)\sim(1+|\xi|)^\gamma$ which is degenerate for sufficiently large $|\xi|$, we can not expect the second term in the right hand side of the above inequality to be controlled by the corresponding weaker dissipative term $\|w_{|\beta-\beta_1|-\ell}\partial^{\alpha+e_i}_{\beta-e_i}u_2\|_\nu^2$ induced by the dissipation of the linearized operator ${\bf L}$.

The key point to overcome such a difficulty is based on an observation that, as the order of the $x-$derivatives of the solutions of the one-species VPB system \eqref{u} increases, the corresponding temporal decay rate also increases. So it is hopeful to control those difficult terms by taking advantage of such a fact, the linear dissipation term $\|w_{|\beta|-\ell}\partial^{\alpha}_\beta u_2\|_{\nu}^2$ induced by the linearized Boltzmann collision operator ${\bf L}$, and the extra dissipative term  $(1+t)^{-1-\theta}\|\langle\xi\rangle w_{|\beta|-\ell}\partial^{\alpha}_\beta u_2\|^2$ defined by \eqref{extra_dissipative} which is due to our introduction of the exponential factor in the new time-velocity weight function \eqref{weight}. In fact, by employing the interpolation technique, see Lemma \ref{in.} for details, one has the following estimate
\begin{equation*}
 \begin{split}
   &(1+t)^{r_{|\alpha|}+\vartheta}\left\|w_{|\beta-e_i|-\ell}\partial^{\alpha+e_i}_{\beta-e_i}u_2(t)\right\|^2\\
   \lesssim& (1+t)^{r_{|\alpha+e_i|}-1}\left\|\langle\xi\rangle w_{|\beta-e_i|-\ell}
   \partial^{\alpha+e_i}_{\beta-e_i}u_2\right\|^2+(1+t)^{r_{|\alpha+e_i|}+\vartheta}
   \left\|w_{|\beta-e_i|-\ell}\partial^{\alpha+e_i}_{\beta-e_i}u_2\right\|_{\nu}^2,
 \end{split}
\end{equation*}
which can be controlled by the mathematical principle of induction according to the order of velocity derivatives since such trouble terms vanish when $|\beta|=0$. A key point here in our analysis is to assign the special temporal decay rates to certain $L^2-$norm of $\partial^\alpha_\beta u$ as in \eqref{decay_rates} for each pair of multiindex $(\alpha,\beta)$ satisfying $|\alpha|+|\beta|\leq N$, especially for the case when $(\alpha,\beta)\in U_{\alpha,\beta}^{\text{high}}$. The very reason, and in fact it is the only reason, to do so is to guarantee that the above estimate, cf. also the estimate \eqref{key}, holds for each pair of multiindex $(\alpha,\beta)$ satisfying $|\alpha|+|\beta|\leq N$ so that the linear combinations performed in \eqref{h.-decay2} and \eqref{NN} work well. The reason to cause such a difficulty is due to the fact that the temporal decay rate we can obtain in Lemma \ref{high-oder-estimates} on $\left\|w_\ell\nabla_x^N u\right\|$ is just $N-\frac 12$ but not $N+\frac 12$.

Finally, to close our analysis, Lemma \ref{non-weight} plays a significant role, but it is not so easy to deduce this lemma since our special designed decay rate for the highest order. Precisely speaking, when we deal with the estimate \eqref{trouble}, we will meet the term $\left\|\langle\xi\rangle^{-\frac{\gamma}{2}l_2}\nabla^{N-1}_x\nabla_{\xi}u\right\|^2$ with decay $(1+t)^{-\left(N-\frac{1}{2}-\frac{2}{3}\right)}$ which is much slower than what we wanted. To overcome this difficult, on one hand we rewrite such a term as
\begin{equation*}
  \nabla^{N-1}_x\nabla_{\xi}u=\nabla^{N-1}_x\nabla_{\xi}{\bf P}u+\nabla^{N-1}_x\nabla_{\xi}{\{\bf I-P\}}u
\end{equation*}
and for the first part, noticing that $\left\|\langle\xi\rangle^{-\frac{\gamma}{2}l_2}\nabla^{N-1}_x\nabla_{\xi}{\bf P}u\right\|^2\sim\left\|\nabla^{N-1}_x{\bf P}u\right\|^2$ which decays exactly like $(1+t)^{-\left(N-\frac{1}{2}\right)}$, while for the microscopic part, by splitting the Fourier frequency as $|k|^{2(N-1)}=|k|^2|k|^{2(N-2)}$ and by repeating the argument used in \cite{Xiao-Xiong-Zhao-JDE-2013}, we can also derive our desired estimate since, from the estimate \eqref{A-u2-wei2.}, the term $\left\|\langle\xi\rangle^{-\frac{\gamma}{2}l_2}\nabla^{N-2}_x\nabla_{\xi}u_2(t)\right\|^2$ can enjoy a nice temporal decay estimate. 

Before concluding this section, it is worth to pointing out that besides the construction of classical solutions near Maxwellians to the VPB system \eqref{VPB}-\eqref{P} and/or \eqref{2-VPB}-\eqref{2-P} in the perturbation framework, the global existence of renormalized solutions with large initial data to the VPB system \eqref{VPB}-\eqref{P} was proved in \cite{Lions-1991} and this result was later generalized to the case with boundary in \cite{Mischler-CMP-2000}. The time asymptotic behavior of the renormalized solutions with extra regularity assumptions was studied in \cite{Desvillettes-Dolbeault-CPDE-1991, Li-JDE-2008}. The decay property of the solutions to the linearized VPB
system around Maxwellians was studied in \cite{Glassey-Strauss-TTSP-1999, Glassey-Strauss-DCDS-1999, Li-Yang-Zhong}. Finally, for the perturbation around vacuum, the results in \cite{Duan-Yang-Zhu--DCDS2006, Duan-Zhang-Zhu-M3AS-2006, Guo-CMP-2001} give the global existence of smooth small-amplitude solutions for the cut-off intermolecular interactions for $\gamma\in(-2,1]$.

The rest of this paper is organized as follows. Section 2 is concerned with some weighted estimates on the linearized collision operator $\bf L$ and the nonlinear collision terms, Section 3 and Section 4 are devoted to deducing certain lower order energy type estimates and higher order energy type estimates, respectively, and Section 5 is concentrated on the proof of Theorem \ref{Thm}. Finally the proof of the almost optimal decay rate of microscopic component will be given in Appendix A.\\

\noindent{\bf{Notations}.}\quad
Throughout this paper, $C$ denotes some generic positive constant (generally large) while $\kappa$ or $\lambda$ is used to denote some generic positive constant (generally small). Note that $C$, $\kappa$, and $\lambda$ may take different values in different places. $A\lesssim B$ means that there is a generic constant $C>0$ such that $A \leq CB$. $A\sim B$ means $A\lesssim B$ and $B\lesssim A$. The multi-indices $\alpha= [\alpha_1,\alpha_2, \alpha_3]$ and $\beta = [\beta_1, \beta_2, \beta_3]$ will be used to record spatial and velocity derivatives, respectively. And $\partial^{\alpha}_{\beta}=\partial^{\alpha_1}_{x_1}\partial^{\alpha_2}_{x_2}\partial^{\alpha_3}_{x_3}
\partial^{\beta_1}_{\xi_1}\partial^{\beta_2}_{\xi_2}\partial^{\beta_3}_{\xi_3}$. Similarly, the notation $\nabla_x^k=\partial^{\alpha}$ will be used when $\beta=0,~|\alpha|=k$.  The length
of $\alpha$ is denoted by $|\alpha|=\alpha_1+\alpha_2+\alpha_3$. And $\alpha'\leq  \alpha$ means that no component of $\alpha'$ is greater than the corresponding component of $\alpha$, and $\alpha'<\alpha$ means that $\alpha'\leq  \alpha$ and $|\alpha'|<|\alpha|$. We use $\langle\cdot,\cdot\rangle$ to denotes the ${L^2_{\xi}}$ inner product in $\R^3_{\xi}$ with the ${L^2}(\R^3_{\xi  })$ norm  $|\cdot|_{L^2}$. And it is convenient to define a weighted inner product as $\langle g_1,g_2\rangle_{\nu}=\langle g_1,\nu g_2\rangle$ in $\R^3_\xi$, with its corresponding Hilbert space denoted by $L^2_{\nu}(\R^3)$. For notational simplicity, $(\cdot, \cdot)$ denotes the ${L^2}$ inner product either in $\R^3_{x}\times\R^3_{\xi }$ or in $\R^3_{x}$ with the ${L^2}(\R^3_{x}\times\R^3_{\xi })$ or the ${L^2}(\R^3_{x})$  norm $\|\cdot\|$. A similar weighted inner product is defined as $(g_1,g_2)_{\nu}=(g_1,\nu g_2)$ with corresponding norm $\|g\|_{\nu}=(g,\nu g)$. For $q\geq 1$, $Z_q$ denotes the space $Z_q=L^2_{\xi }(\R^3_{\xi };L^q(\R^3_{x}))$ with the norm $\|u\|_{Z_q}=\left|\|u\|_{L^q_{x}}\right|^2_{L^2_{\xi}}$.

\section{Preliminaries}
This section is concerned with some basic velocity weighted estimates on the linearized collision operator $\bf L$ and the nonlinear collision term $\Gamma(u,u)$. For this purpose, we first list the velocity weighted estimates on the linearized collision operator $\bf L$ and the integral operator $K$ with respect to the velocity function $w_{\ell}(t,\xi)$ defined in \eqref{weight} whose proofs can be found in \cite{Strain-Guo-ARMA-2008}
\begin{lemma}\label{L-K} [cf. \cite{Strain-Guo-ARMA-2008}] Let $-3<\gamma<0,~\ell\in \R,$ and $0<q\ll1$. If $|\beta|>0$, then for any $\eta>0$, there is $C_{\eta}>0$ such that
\begin{equation}\label{L}
\begin{split}
\left\langle w^2_{|\beta|-\ell,q}\partial_{\beta}[{\bf L}g],\partial_{\beta}g\right\rangle
\geq\left|w_{|\beta|-\ell}\partial_{\beta} g\right|^2_{L^2_{\nu}}
-\eta\sum_{|\beta_1|\leq|\beta|}\left|w_{|\beta_1|-\ell}\partial_{\beta_1}g\right|^2_{L^2_{\nu}}
-C_{\eta}\left|\chi_{|\xi|<2C_{\eta}}\langle\xi\rangle^{-\gamma \ell}g\right|^2_{L^2}.
\end{split}
\end{equation}
If $|\beta|=0$, then for any $\eta>0$, there is $C_{\eta}>0$ such that
\begin{equation*}
\begin{split}
\left\langle w^2_{-\ell,q}Kg_1,g_2\right\rangle
\leq \left(\eta\left|w_{-\ell} g_1\right|_{L^2_{\nu}}
+C_{\eta}\left|\chi_{|\xi|<2C_{\eta}}\langle\xi\rangle^{-\gamma l} g_1\right|_{L^2}\right)
\left|w_{-\ell} g_2\right|_{L^2_{\nu}}.
\end{split}
\end{equation*}
Here $\chi_{|\xi|<2C_{\eta}}=\chi_{|\xi|<2C_{\eta}}(\xi)$ denotes the characteristic function of the set $\{\xi\in \R^3: |\xi|<2C_{\eta}\}$.
\end{lemma}

Now we turn to deduce corresponding weighted estimates on the terms related to the nonlinear collision term $\Gamma(u,u)$ which are fundamental in our analysis. Before stating our results, we recall that
\begin{equation*}
\partial^{\alpha}_{\beta}\Gamma\left(g_1,g_2\right)=\sum C^{\beta_0,\beta_1,\beta_2}_{\beta}C^{\alpha_1,\alpha_2}_{\alpha}
\Gamma^0\left(\partial^{\alpha_1}_{\beta_1}g_1,\partial^{\alpha_2}_{\beta_2}g_2\right),
\end{equation*}
where the summation is over $\beta_0+\beta_1+\beta_2=\beta$ and $\alpha_1+\alpha_2=\alpha$, and $\Gamma^0$ is given as follows
\begin{equation*}
\begin{split}
\Gamma^0\left(\partial^{\alpha_1}_{\beta_1}g_1,\partial^{\alpha_2}_{\beta_2}g_2\right)
=&\iint_{\R^3\times\mathbb{S}^2}|\xi-\xi_{*}|^{\gamma}q_0(\vartheta)
\partial_{\beta_0}\left[{\bf M}^{\frac{1}{2}}(\xi_{*})\right]
\partial^{\alpha_1}_{\beta_1}g_1(\xi'_{*})\partial^{\alpha_2}_{\beta_2}g_2(\xi')d\omega d\xi_{*}\\
&-\partial^{\alpha_2}_{\beta_2}g_2(\xi)\iint_{\R^3\times\mathbb{S}^2}
|\xi-\xi_{*}|^{\gamma}q_0(\vartheta)\partial _{\beta_0}\left[{\bf M}^{\frac{1}{2}}(\xi_{*})\right]\partial^{\alpha_1}_{\beta_1}g_1(\xi_{*})d \omega d\xi_{*}\\
=&\Gamma^0_{gain}-\Gamma^0_{loss}.
\end{split}
\end{equation*}
Moreover, we need the following result whose special version has been proved in \cite{Xiao-Xiong-Zhao-JDE-2013}:
\begin{lemma}\label{JDE-lem.}
Set $\nu=\nu(\xi)=\langle\xi\rangle^{\frac{\gamma}{2}}$ with $-3<\gamma<0$. For any $\ell\geq0$, it holds that
\begin{equation*}
 \begin{split}
   \left|\nu^{-\frac{l}{2}}w_{-\ell}{ \Gamma}(g_1,g_2)\right|^2_{L^2}
   \lesssim\displaystyle\sum_{|\beta|\leq2}\left|w_{|\beta|-\ell}\partial_{\beta}g_1\right|^2
   \left|w_{-\ell}g_2\right|^2,\\
   \left|\nu^{-\frac{1}{2}}w_{-\ell}{ \Gamma}(g_1,g_2)\right|^2_{L^2}
   \lesssim\displaystyle\sum_{|\beta|\leq2}\left|w_{|\beta|-\ell}\partial_{\beta}g_2\right|^2
   \left|w_{-\ell}g_1\right|^2.
 \end{split}
\end{equation*}
\end{lemma}
\begin{proof}
  The proof of this lemma is similar to that of Lemma 2.4 in \cite{Xiao-Xiong-Zhao-JDE-2013}, the only thing we need to pay attention to is that we need to change the algebraic weight $\nu^{-2l}$ for any $l\geq 0$ used in \cite{Xiao-Xiong-Zhao-JDE-2013} to the exponential weight $\nu^{-\frac{1}{2}}w_{-\ell}(t,\xi)$ formally here. Since the modification is straightforward, we omit the details for brevity.
\end{proof}
With the above lemma in hand, we now deal with the corresponding weighted estimates on those terms related to the nonlinear collision term in terms of the temporal energy functional $\mathcal{E}_\infty(t)$ defined by \eqref{Energy}:
\begin{lemma}\label{Gamma}
Let $|\alpha|+|\beta|=k\leq N$, $l_2\geq 0$, $N\geq4$, and $\ell\geq\max\left\{N, l_2+1,l_2-\frac 2\gamma\right\}$. It holds that
  \begin{equation}\label{estimate-on-Gamma}
   \begin{split}
    \left\|\nu^{-\frac{1}{2}}w_{|\beta|-\ell}\partial^{\alpha}_{\beta}\Gamma(u,u)\right\|^2
    \lesssim&
    \begin{cases}
      (1+t)^{-\left(|\alpha|+\frac{5}{2}-\frac 16\right)}\CE^2_{\infty}(t),~\text{if}~k\leq N-2,\\
      (1+t)^{-\left(|\alpha|+\frac{5}{2}-\frac{5}{6}\right)}\CE^2_{\infty}(t),~\text{if}~k=N-1,\\
      (1+t)^{-\left(|\alpha|+\frac{3}{2}-\frac{1}{6}\right)}\CE^2_{\infty}(t),~\text{if}~k=N
    \end{cases}
   \end{split}
  \end{equation}
and
\begin{eqnarray}\label{Z_1-on-Gamma}
\left\|\nu^{-l_2}\nabla^{N-1}_x\Gamma(u,u)\right\|_{Z_1}^2\lesssim (1+t)^{-N}\CE^2_{\infty}(t).
\end{eqnarray}
\end{lemma}
\begin{proof} Noticing that
$$
\partial^{\alpha}_{\beta}\Gamma(u,u)=\sum\limits_{\alpha_1+\alpha_2=\alpha\atop \beta_1+\beta_2\leq\beta}\Gamma\left(\partial^{\alpha_1}_{\beta_1}u,\partial^{\alpha_2}_{\beta_2}u\right)
\equiv \sum\limits_{\alpha_1+\alpha_2\leq\alpha\atop \beta_1+\beta_2\leq\beta}J^{\alpha_1,\alpha_2}_{\beta_1,\beta_2},
$$
we only need to deduce an suitable estimate on $J^{\alpha_1,\alpha_2}_{\beta_1,\beta_2}$ for $\alpha_1+\alpha_2=\alpha, \beta_1+\beta_2\leq\beta$ and for this purpose, we can assume without loss of our generality that $|\alpha_1|+|\beta_1|\leq \frac{k}{2}$. Our discussion will be divided into three cases:\\
\vskip 0.5mm
\noindent {\bf Case I:  $|\alpha|+|\beta|=k\leq N-2$.}\\
In this case, by employing Lemma \ref{JDE-lem.}, we have from the Gagliardo-Nirenberg interpolation inequality that
\begin{eqnarray*}
    J_{\beta_1,\beta_2}^{\alpha_1,\alpha_2}
    &\lesssim&\sum_{|\alpha_1|+|\beta_1|\leq \frac{k}{2}\atop |\bar{\beta}|\leq2}
    \left\|w_{|\beta_1+\bar{\beta}|-\ell}\partial^{\alpha_1}_{\beta_1+\bar{\beta}}u\right\|^2_{L^3_x(L^2_{\xi})}
    \left\|w_{|\beta_2|-\ell}\partial^{\alpha_2}_{\beta_2}u\right\|^2_{L^6_x(L^2_{\xi})}\\
    &\lesssim&\sum_{|\alpha_1|+|\beta_1|\leq \frac{k}{2}\atop |\bar{\beta}|\leq2}
    \left\|w_{|\beta_1+\bar{\beta}|-\ell}\partial^{\alpha_1}_{\beta_1+\bar{\beta}}u\right\|
    \left\|w_{|\beta_1+\bar{\beta}|-\ell}\partial^{\alpha_1}_{\beta_1+\bar{\beta}}\nabla_xu\right\|
    \left\|w_{|\beta_2|-\ell}\partial^{\alpha_2}_{\beta_2}\nabla_xu\right\|^2.
\end{eqnarray*}

Notice further that for $N>4$, the fact $|\alpha_1|+|\beta_1|\leq \frac{k}{2}\leq \frac N2-1$ implies that
$$
|\alpha_1|+|\beta_1+\bar{\beta}|\leq |\alpha_1|+|\beta_1|+2\leq \frac N2+1\leq N-1,
$$
$$
|\alpha_1|+|\beta_1+\bar{\beta}|+1\leq |\alpha_1|+|\beta_1|+3\leq \frac N2+2\leq N,
$$
but
$$
|\alpha_1|+1\leq \frac N2<N-2,
$$
and
$$
|\alpha_2|+|\beta_2|+1\leq|\alpha|+|\beta|+1\leq N-1,
$$
we thus can bound $J_{\beta_1,\beta_2}^{\alpha_1,\alpha_2}$ further by employing the definition of $\CE_{\infty}(t)$ defined by \eqref{Energy} as follows
\begin{eqnarray*}
    J_{\beta_1,\beta_2}^{\alpha_1,\alpha_2}
    &\lesssim&(1+t)^{-\left(\frac{|\alpha_1|}{2}+\frac 14\right)-\left(\frac{|\alpha_1|+1}{2}+\frac 14\right)-\left((|\alpha_2|+1)+\frac 12\right)}\CE^2_{\infty}(t)\\
    &\lesssim&(1+t)^{-\left(|\alpha|+\frac{5}{2}\right)}\CE^2_{\infty}(t).
\end{eqnarray*}

For the case $N=4$, we have $|\alpha_1|+|\beta_1|\leq 1$ and we only consider the more difficult case $|\alpha_1|=1$. In such a case, one has $|\alpha_1|+|\bar{\beta}|+1\leq 4=N$, but in such a case $|\alpha_1|+1=2=N-2$. Thus by the definition of $\CE_{\infty}(t)$ defined by \eqref{Energy}, one can only has
$$
\left\|w_{|\bar{\beta}|-\ell}\partial^{\alpha_1}_{\bar{\beta}}\nabla_xu\right\|\lesssim
(1+t)^{-\frac{|\alpha_1|+1}{2}-\frac 14+\frac 16}\CE^{\frac 12}_{\infty}(t),
$$
and consequently
\begin{eqnarray*}
    J_{\beta_1,\beta_2}^{\alpha_1,\alpha_2}
    &\lesssim&(1+t)^{-\left(\frac{|\alpha_1|}{2}+\frac 14-\frac 16\right)-\left(\frac{|\alpha_1|+1}{2}+\frac 14\right)-\left((|\alpha_2|+1)+\frac 12\right)}\CE^2_{\infty}(t)\\
    &\lesssim&(1+t)^{-\left(|\alpha|+\frac{5}{2}-\frac 16\right)}\CE^2_{\infty}(t).
\end{eqnarray*}
Putting the above estimates together yields the first estimate of \eqref{estimate-on-Gamma}.\\

\noindent{\bf Case II: $|\alpha|+|\beta|=N$.}\\
Such a case can be treated as in the case I by considering the cases $N>4$ and $N=4$ separately. In fact for such a case, from the definitions of $\CE_{\infty}(t)$ and  $r_{|\alpha|}$, we have
  \begin{equation*}
   \begin{tabular}{|c|c|c|c|c|}
     \hline
     $|\alpha|+|\beta|=N$& $|\alpha|=N$ & $|\alpha|=N-1$ & $|\alpha|=N-2$ & $|\alpha|\leq N-3$\\ \hline
     Decay of $\left\|w_{|\beta|-\ell}\partial^{\alpha}_{\beta}u\right\|^2$ & $(1+t)^{-|\alpha|+\frac{1}{2}}$ & $(1+t)^{-|\alpha|-\frac{1}{2}+\frac{2}{3}}$ & $(1+t)^{-|\alpha|-\frac{1}{2}+\frac{1}{3}}$ &$(1+t)^{-|\alpha|-\frac{1}{2}}$\\ \hline
   \end{tabular}
  \end{equation*}
Based on Lemma \ref{JDE-lem.} and the decay estimates above we designed, we can use Sobolev inequalities to estimate $J^{\alpha_1,\alpha_2}_{\beta_1,\beta_2}$ in different subcases: If $|\alpha_1|+|\beta_1|=0$, we have
  \begin{eqnarray*}
    J^{0,\alpha_2}_{0,\beta_2}
    &\lesssim&\sum_{|\bar{\beta}|\leq2}
    \sup_x\left\{\left|w_{|\bar{\beta}|-\ell}\partial_{\bar{\beta}}u\right|^2_{L^2_{\xi}}\right\}
    \left\|w_{|\beta_2|-\ell}\partial^{\alpha_2}_{\beta_2}u\right\|^2\\
    &\lesssim&\sum_{|\bar{\beta}|\leq2}\left\|w_{|\bar{\beta}|-\ell}\partial_{\bar{\beta}}\nabla_xu\right\|
    \left\|w_{|\bar{\beta}|-\ell}\partial_{\bar{\beta}}\nabla_x^2u\right\|
    \left\|w_{|\beta_2|-\ell}\partial^{\alpha_2}_{\beta_2}u\right\|^2\\
    &\lesssim&\begin{cases}
      (1+t)^{-\left(|\alpha|+\frac{3}{2}-\frac{1}{6}\right)}\CE^2_{\infty}(t),~\mbox{when~}|\alpha_2|=N,~|\beta_2|=0,\\
      (1+t)^{-\left(|\alpha|+\frac{5}{2}-\frac{2}{3}-\frac{1}{6}\right)}\CE^2_{\infty}(t),
      ~\mbox{when~}|\alpha_2|=N-1,~|\beta_2|=1,\\
      (1+t)^{-\left(|\alpha|+\frac{5}{2}-\frac{1}{3}-\frac{1}{6}\right)}\CE^2_{\infty}(t),
      ~\mbox{when~}|\alpha_2|=N-2,~|\beta_2|=2,\\
      (1+t)^{-\left(|\alpha|+\frac{5}{2}-\frac{1}{6}\right)}\CE^2_{\infty}(t),~\mbox{when~}|\alpha_2|\leq N-3;
    \end{cases}
\end{eqnarray*}
 If $1\leq |\alpha_1|+|\beta_1|<\frac{N}{2}$, we obtain from the fact $|\alpha_2|+|\beta_2|+1=N-|\alpha_1|-|\beta_1|+1\leq N$ that
   \begin{eqnarray*}
   J^{\alpha_1,\alpha_2}_{\beta_1,\beta_2}
    &\lesssim&\sum_{|\bar{\beta}|\leq2}\left\|w_{|\beta_1+\bar{\beta}|-\ell}
    \partial^{\alpha_1}_{\beta_1+\bar{\beta}}u\right\|^2_{L^3_x(L^2_{\xi})}
    \left\|w_{|\beta_2|-\ell}\partial^{\alpha_2}_{\beta_2}u\right\|^2_{L^6_x(L^2_{\xi})}\\
    &\lesssim&\sum_{|\bar{\beta}|\leq2}
    \left\|w_{|\beta_1+\bar{\beta}|-\ell}\partial^{\alpha_1}_{\beta_1+\bar{\beta}}u\right\|
    \left\|w_{|\beta_1+\bar{\beta}|-\ell}\partial^{\alpha_1}_{\beta_1+\bar{\beta}}\nabla_xu\right\|
    \left\|w_{|\beta_2|-\ell}\partial^{\alpha_2}_{\beta_2}\nabla_xu\right\|^2\\
    &\lesssim&(1+t)^{-\frac{|\alpha_1|+\frac{1}{2}}{2}}
    (1+t)^{-\frac{1}{2}\left(|\alpha_1|+\frac{3}{2}-\frac{1}{3}\right)}
      (1+t)^{-\left(|\alpha_2|+1-\frac{1}{2}\right)}\CE^2_{\infty}(t)\\
    &\lesssim&(1+t)^{-\left(|\alpha|+\frac{3}{2}-\frac 16\right)}\CE^2_{\infty}(t);
  \end{eqnarray*}
 If $|\alpha_1|+|\beta_1|=\frac{N}{2}$, we can deduce from the definitions of $\CE_{\infty}(t)$ that
  \begin{eqnarray*}
   J^{\alpha_1,\alpha_2}_{\beta_1,\beta_2} &\lesssim&\sum_{|\bar{\beta}|\leq2}
   \left\|w_{|\beta_1+\bar{\beta}|-\ell}\partial^{\alpha_1}_{\beta_1+\bar{\beta}}u\right\|^2
    \left\|w_{|\beta_2|-\ell}\partial^{\alpha_2}_{\beta_2}u\right\|^2_{L^{\infty}_x(L^2_{\xi})}\\
    &\lesssim&\sum_{|\bar{\beta}|\leq2}
    \left\|w_{|\beta_1+\bar{\beta}|-\ell}\partial^{\alpha_1}_{\beta_1+\bar{\beta}}u\right\|^2
    \left\|w_{|\beta_2|-\ell}\partial^{\alpha_2}_{\beta_2}\nabla_xu\right\|
    \left\|w_{|\beta_2|-\ell}\partial^{\alpha_2}_{\beta_2}\nabla_x^2u\right\|\\
    &\lesssim&(1+t)^{-\left(|\alpha_1|+\frac{1}{2}-\frac{1}{3}\right)}
    (1+t)^{-\frac{1}{2}\left(|\alpha_2|+\frac{3}{2}-\frac 23\right)}
    (1+t)^{-\frac{1}{2}\left(|\alpha_2|+2-\frac{1}{2}\right)}\CE^2_{\infty}(t)\\
    &\lesssim&(1+t)^{-\left(|\alpha|+\frac{3}{2}-\frac 16\right)}\CE^2_{\infty}(t).
  \end{eqnarray*}

\noindent{\bf Case III:  $|\alpha|+|\beta|=N-1$.}\\
For such a case, by repeating the argument used in dealing with the Case I, one can deduce that
  \begin{equation*}
    J^{\alpha_1,\alpha_2}_{\beta_1,\beta_2}
    \lesssim(1+t)^{-\left(|\alpha|+\frac{5}{2}-\frac{5}{6}\right)}\CE^2_{\infty}(t).
  \end{equation*}
Putting the estimates obtained in the above three cases together yield the estimate \eqref{estimate-on-Gamma} stated in Lemma \ref{Gamma}.

As to the estimate \eqref{Z_1-on-Gamma}, we have from the H$\ddot{o}$lder inequality, the fact that $N-1-m+|\beta|\leq N, N-1-m\leq N-3$ hold for $|\beta|\leq 2$ and $2\leq m\leq N-1$, the definition of $\CE_{\infty}(t)$, and Lemma \ref{JDE-lem.} that
\begin{eqnarray*}
&&\left\|\nu^{-l_2}\nabla^{N-1}_x\Gamma(u,u)\right\|_{Z_1}^2\\
&\lesssim& \sum\limits_{m\leq N-1}\left|\nu^{-l_2}\Gamma\left(\left\|\nabla_x^{N-1-m}u\right\|_{L^2_x},
\left\|\nabla_x^mu\right\|_{L^2_x}\right)\right|_{L^2_\xi}^2\\
&\lesssim& \sum\limits_{|\beta|\leq 2}\left\|w_{-\ell}\nabla_x^{N-1}u\right\|^2
\left\|w_{|\beta|-\ell}\partial_\beta u\right\|^2
+\sum\limits_{|\beta|\leq 2}\left\|w_{-\ell}\nabla_x^{N-2}u\right\|^2
\left\|w_{|\beta|-\ell}\partial_\beta\nabla_x u\right\|^2\\
&&+\sum\limits_{|\beta|\leq 2 \atop 2\leq m\leq N-1}\left\|w_{-\ell}\nabla_x^{m}u\right\|^2
\left\|w_{|\beta|-\ell}\partial_\beta \nabla_x^{N-1-m}u\right\|^2\\
&\lesssim&\left( (1+t)^{-\left(N-1+\frac 12\right)-\frac 12}
+(1+t)^{-\left(N-2+\frac 12\right)-\left(1+\frac 12\right)}
+(1+t)^{-\left(m+\frac 12\right)-\left(N-1-m+\frac 12\right)}\right)\CE^2_{\infty}(t)\\
&\lesssim&(1+t)^{-N}\CE^2_{\infty}(t).
\end{eqnarray*}
This is \eqref{Z_1-on-Gamma} and thus completes the proof of Lemma \ref{Gamma}.
\end{proof}

To deduce the corresponding estimates related to the other two nonlinear terms $\nabla_x\phi\cdot\nabla_{\xi}u_2$ and $\frac{1}{2}\xi\cdot\nabla_x\phi u_2$. To do so, we need the following interpolation lemma which will be used frequently in later sections.
\begin{lemma}\label{in.}
  For any $m\in \R$, $\eta>0$, $\gamma\in(-3,0)$, if furthermore provided $0<\vartheta\leq \frac{1}{4}$, one has
  \begin{eqnarray*}
    &&(1+t)^{m+\vartheta}\left\|\langle\xi\rangle^{\frac{1}{2}} w_{|\beta|-\ell}\partial^{\alpha}_{\beta}u\right\|^2\\
    &\lesssim&\eta(1+t)^m\left\|\langle\xi\rangle w_{|\beta|-\ell}
    \partial^{\alpha}_{\beta}u\right\|^2+C_{\eta}(1+t)^{m+(2-\gamma)\vartheta}\left\|w_{|\beta|-\ell}
    \partial^{\alpha}_{\beta}u\right\|_{\nu}^2\\
    &\lesssim&\eta\left((1+t)^m\left\|\langle\xi\rangle w_{|\beta|-\ell}
    \partial^{\alpha}_{\beta}u\right\|^2+(1+t)^{1+m+\vartheta}\left\|w_{|\beta|-\ell}
    \partial^{\alpha}_{\beta}u\right\|_{\nu}^2\right)
    +C_{\eta}\left\|w_{|\beta|-\ell}\partial^{\alpha}_{\beta}u\right\|_{\nu}^2,
  \end{eqnarray*}
  and
  \begin{eqnarray*}
    &&(1+t)^{m+\vartheta}\left\|w_{|\beta|-\ell}\partial^{\alpha}_{\beta}u\right\|^2\\
    &\lesssim&\eta(1+t)^m\left\|\langle\xi\rangle w_{|\beta|-\ell}
    \partial^{\alpha}_{\beta}u\right\|^2+C_{\eta}(1+t)^{m+\frac{2-\gamma}{2}\vartheta}\left\|w_{|\beta|-\ell}
    \partial^{\alpha}_{\beta}u\right\|_{\nu}^2\\
    &\lesssim&\eta\left((1+t)^m\left\|\langle\xi\rangle w_{|\beta|-\ell}
    \partial^{\alpha}_{\beta}u\right\|^2+(1+t)^{1+m+\vartheta}\left\|w_{|\beta|-\ell}
    \partial^{\alpha}_{\beta}u\right\|_{\nu}^2\right)
    +C_{\eta}\left\|w_{|\beta|-\ell}\partial^{\alpha}_{\beta}u\right\|_{\nu}^2.
  \end{eqnarray*}
\end{lemma}
\begin{proof} Note that for any $m\in \R$, $\eta>0$, $\gamma\in(-3,0)$, any $\vartheta>0$,
$$
  m+\vartheta=\frac{1-\gamma}{2-\gamma}m+\frac{1}{2-\gamma}\left(m+(2-\gamma)\vartheta\right)
  =\frac{-\gamma}{2-\gamma}m+\frac{2}{2-\gamma}\left(m+\frac{2-\gamma}{2}\vartheta\right)
$$
and furthermore, if we provide $0<\vartheta\leq \frac{1}{4}$,
$$
m+(2-\gamma)\vartheta<1+m+\vartheta,\quad
m+\frac{2-\gamma}{2}\vartheta<1+m+\vartheta.
$$
The following  inequalities
  \begin{equation*}
    (1+t)^{m+\vartheta}\left\|\langle\xi\rangle^{\frac{1}{2}} w_{|\beta|-\ell}\partial^{\alpha}_{\beta}u\right\|^2
  \leq\left((1+t)^{m}\left\|\langle\xi\rangle w_{|\beta|-\ell}
    \partial^{\alpha}_{\beta}u\right\|^2\right)^{\frac{1-\gamma}{2-\gamma}}
    \left((1+t)^{m+(2-\gamma)\vartheta}\left\|\langle\xi\rangle^{\frac{\gamma}{2}}w_{|\beta|-\ell}
    \partial^{\alpha}_{\beta}u\right\|^2\right)^{\frac{1}{2-\gamma}}
  \end{equation*}
and
  \begin{equation*}
    (1+t)^{m+\vartheta}\left\|w_{|\beta|-\ell}\partial^{\alpha}_{\beta}u\right\|^2
  \leq\left((1+t)^{m}\left\|\langle\xi\rangle w_{|\beta|-\ell}
    \partial^{\alpha}_{\beta}u\right\|^2\right)^{\frac{-\gamma}{2-\gamma}}
    \left((1+t)^{m+\frac{2-\gamma}{2}\vartheta}\left\|\langle\xi\rangle^{\frac{\gamma}{2}}w_{|\beta|-\ell}
    \partial^{\alpha}_{\beta}u\right\|^2\right)^{\frac{2}{2-\gamma}}
  \end{equation*}
hold. Then we can deduce our result by further using Young's inequality and Cauchy-Schwarz's inequality for any sufficiently small $\eta>0$. This completes the proof of Lemma \ref{in.}.
\end{proof}

With Lemma \ref{in.} in hand, we now turn to deal with the corresponding weighted estimates on the derivatives of the nonlinear terms $\nabla_x\phi\cdot\nabla_xu_2$ and $\frac 12\xi\cdot\nabla_x\phi u_2$ with respect to the pure spatial variable and the mixed spatial and velocity variables.
\begin{lemma}\label{non-LINEAR}
  For $1\leq|\alpha|\leq N-1$ and any $\eta>0$, for the weighted estimates on the pure spatial derivatives of the nonlinear terms $\nabla_x\phi\cdot\nabla_xu_2$ and $\frac{1}{2}\xi\cdot\nabla_x\phi u_2$, we can deduce for $-3<\gamma<0$ that
  \begin{equation}\label{good-1}
 \begin{split}
  &\sum_{|\alpha_1|<|\alpha|}\left|\left(\partial^{\alpha-\alpha_1}\nabla_x\phi
  \cdot\nabla_{\xi}\partial^{\alpha_1}u_2,w_{-\ell}^2(t,\xi)\partial^{\alpha}u_2\right)\right|\\
  \lesssim&\eta\left((1+t)^{-1-\vartheta}\left\|\langle\xi\rangle w_{-\ell}\partial^{\alpha}u_2\right\|^2
   +\left\|w_{-\ell}\partial^{\alpha}u_2\right\|^2_{\nu}\right)\\
  &+C_{\eta}\sum_{|\alpha_1|<|\alpha|}(1+t)^{-\left(|\alpha-\alpha_1|+2\right)+(1+\vartheta)\frac{1-\gamma}{2-\gamma}}
   \left\|\langle\xi\rangle w_{1-\ell}\partial^{\alpha_1}\nabla_{\xi}u_2\right\|^2\CE_{\infty}(t)
   \end{split}
 \end{equation}
and
\begin{equation}\label{good-3}
  \begin{split}
  &\sum_{|\alpha_1|<|\alpha|}\left|\left(\xi\cdot
  \partial^{\alpha-\alpha_1}\nabla_x\phi
  \partial^{\alpha_1}u_2,w_{-\ell}^2(t,\xi)\partial^{\alpha}u_2\right)\right|\\
  \lesssim&\eta\left((1+t)^{-1-\vartheta}\left\|\langle\xi\rangle w_{-\ell}\partial^{\alpha}u_2\right\|^2
   +\left\|w_{-\ell}\partial^{\alpha}u_2\right\|_{\nu}^2\right)\\
  &+C_{\eta}\sum_{|\alpha_1|<|\alpha|}(1+t)^{-(|\alpha-\alpha_1|+2)+(1+\vartheta)\frac{1-\gamma}{2-\gamma}}
   \left\|\langle\xi\rangle^{\frac{1}{2}}w_{-\ell}\partial^{\alpha_1}u_2\right\|^2\CE_{\infty}(t).
  \end{split}
\end{equation}
Similarly, for the weighted estimates on the mixed spatial and velocity derivatives of the nonlinear terms $\nabla_x\phi\cdot\nabla_xu_2$ and $\frac 12\xi\cdot\nabla_x\phi u_2$, it holds for $|\alpha|+|\beta|=k\leq N,~|\beta|\geq1$, $-3<\gamma<0$, and any $\eta>0$ that
\begin{equation}\label{good-4}
  \begin{split}
  &\sum\limits_{|\alpha_1|<|\alpha|}\left|\left(\partial^{\alpha-\alpha_1}\nabla_x\phi\cdot
  \nabla_{\xi}\partial^{\alpha_1}_{\beta}u_2,w_{|\beta|-\ell}^2\partial^{\alpha}_{\beta}u_2\right)\right|\\
  \lesssim&\eta\left((1+t)^{-1-\vartheta}\left\|\langle\xi\rangle w_{|\beta|-\ell}\partial^{\alpha}_{\beta}u_2\right\|^2
  +\left\|w_{|\beta|-\ell}\partial^{\alpha}_{\beta}u_2\right\|_{\nu}^2\right)\\
  &+C_{\eta}\CE_{\infty}(t)\sum_{1\leq|\alpha_1|<|\alpha|}
  (1+t)^{-(|\alpha-\alpha_1|+2)+(1+\vartheta)\frac{1-\gamma}{2-\gamma}}
  \left\|\langle\xi\rangle w_{|\beta+e_i|-\ell}\nabla_{\xi}\partial^{\alpha_1}_{\beta}u_2\right\|^2
  \end{split}
\end{equation}
and
\begin{equation}\label{good-6}
  \begin{split}
 &\sum_{|\alpha_1|<|\alpha|}\left|\left(\xi\cdot\partial^{\alpha-\alpha_1}
  \nabla_x\phi\partial^{\alpha_1}_{\beta}u_2,w_{|\beta|-\ell}^2\partial^{\alpha}_{\beta}u_2\right)\right|\\
 \lesssim&\eta\left((1+t)^{-1-\vartheta}\left\|\langle\xi\rangle^{\frac{1}{2}} w_{|\beta|-\ell}\partial^{\alpha}_{\beta}u_2\right\|^2
  +\left\|w_{|\beta|-\ell}\partial^{\alpha}_{\beta}u_2\right\|_{\nu}^2\right)\\
  &+C_{\eta}\CE_{\infty}(t)\sum_{1\leq|\alpha_1|<|\alpha|}
  (1+t)^{-(|\alpha-\alpha_1|+2)+(1+\vartheta)\frac{1-\gamma}{2-\gamma}}
  \left\|\langle\xi\rangle^{\frac{1}{2}} w_{|\beta|-\ell}\nabla_{\xi}\partial^{\alpha_1}_{\beta}u_2\right\|^2.
  \end{split}
\end{equation}
\begin{proof} We will only prove the estimates \eqref{good-1} and \eqref{good-4} in details in the following since the proofs of the estimates \eqref{good-3} and \eqref{good-6} follows essentially the same way, we thus omit the details for brevity.

Firstly, for \eqref{good-1}, by noticing the fact $\langle\xi\rangle^{-\frac{\gamma}{2}-1}\leq\langle\xi\rangle^{\frac{1}{2}}$ for $-3<\gamma<0$, and the interpolation
\begin{eqnarray*}
  \left\|\langle\xi\rangle^{-\frac{\gamma}{2}-1}w_{-\ell}(t,\xi)\partial^{\alpha}u_2\right\|
  &\lesssim&\left\|\langle\xi\rangle^{\frac{1}{2}}w_{-\ell}(t,\xi)\partial^{\alpha}u_2\right\|\\
  &\lesssim&\left\|\langle\xi\rangle w_{-\ell}\partial^{\alpha}u_2\right\|^{\frac{1-\gamma}{2-\gamma}}
  \left\|\langle\xi\rangle^{\frac{\gamma}{2}}w_{-\ell}\partial^{\alpha}u_2\right\|^{\frac{1}{2-\gamma}},
\end{eqnarray*}
we can deduce that
 \begin{equation*}
 \begin{split}
  &\sum_{|\alpha_1|<|\alpha|}\left|\left(\partial^{\alpha-\alpha_1}\nabla_x\phi
  \cdot\nabla_{\xi}\partial^{\alpha_1}u_2,w_{-\ell}^2(t,\xi)\partial^{\alpha}u_2\right)\right|\\
  \lesssim&\sum_{|\alpha_1|<|\alpha|}\left|\left(\partial^{\alpha-\alpha_1}\nabla_x\phi
  \cdot\nabla_{\xi}\partial^{\alpha_1}u_2\langle\xi\rangle w_{1-\ell},\langle\xi\rangle^{-\frac{\gamma}{2}-1}
  w_{-\ell}(t,\xi)\partial^{\alpha}u_2\right)\right|\\
  \lesssim&\sum_{|\alpha_1|=|\alpha|-1}\left\|\partial^{\alpha-\alpha_1}\nabla_x\phi\right\|_{L^{\infty}}
  \left\|\langle\xi\rangle w_{1-\ell}\nabla_{\xi}\partial^{\alpha_1}u_2\right\|
  \left\|\langle\xi\rangle w_{-\ell}\partial^{\alpha}u_2\right\|^{\frac{1-\gamma}{2-\gamma}}
  \left\|\langle\xi\rangle^{\frac{\gamma}{2}}w_{-\ell}\partial^{\alpha}u_2\right\|^{\frac{1}{2-\gamma}}\\
  &+\sum_{|\alpha_1|\leq|\alpha|-2}\left\|\partial^{\alpha-\alpha_1}\nabla_x\phi\right\|_{L^3}
  \left\|\langle\xi\rangle w_{1-\ell}\nabla_{\xi}\partial^{\alpha_1}u_2\right\|_{L_x^6(L^2_{\xi})}
  \left\|\langle\xi\rangle w_{-\ell}\partial^{\alpha}u_2\right\|^{\frac{1-\gamma}{2-\gamma}}
  \left\|\langle\xi\rangle^{\frac{\gamma}{2}}w_{-\ell}\partial^{\alpha}u_2\right\|^{\frac{1}{2-\gamma}}\\
  \lesssim&\eta\left((1+t)^{-1-\vartheta}\left\|\langle\xi\rangle w_{-\ell}\partial^{\alpha}u_2\right\|^2
   +\left\|w_{-\ell}\partial^{\alpha}u_2\right\|^2_{\nu}\right)\\
  &+C_{\eta}\sum_{|\alpha_1|<|\alpha|}(1+t)^{-\left(|\alpha-\alpha_1|+2\right)+(1+\vartheta)\frac{1-\gamma}{2-\gamma}}
   \left\|\langle\xi\rangle w_{1-\ell}\partial^{\alpha_1}\nabla_{\xi}u_2\right\|^2\CE_{\infty}(t).
\end{split}
\end{equation*}
Here to deduce the last inequality in the above analysis, we have used the following estimate which is obtained by employing Young's inequality and Cauchy's inequality (Here we only write down the most difficult case $|\alpha_1|=|\alpha|-1$, similar estimates can also be established for the cases $|\alpha_1|\leq|\alpha|-2$.)
\begin{equation*}
  \begin{split}
    &\sum_{|\alpha_1|=|\alpha|-1}\left\|\partial^{\alpha-\alpha_1}\nabla_x\phi\right\|_{L^{\infty}}
  \left\|\langle\xi\rangle w_{1-\ell}\nabla_{\xi}\partial^{\alpha_1}u_2\right\|
  \left\|\langle\xi\rangle w_{-\ell}\partial^{\alpha}u_2\right\|^{\frac{1-\gamma}{2-\gamma}}
  \left\|\langle\xi\rangle^{\frac{\gamma}{2}}w_{-\ell}\partial^{\alpha}u_2\right\|^{\frac{1}{2-\gamma}}\\
  \lesssim&\sum_{|\alpha_1|=|\alpha|-1}(1+t)^{\frac{1}{2}(1+\vartheta)\times\frac{1-\gamma}{2-\gamma}}
  \left\|\partial^{\alpha-\alpha_1}\nabla_x\phi\right\|_{L^{\infty}}
  \left\|\langle\xi\rangle w_{1-\ell}\nabla_{\xi}\partial^{\alpha_1}u_2\right\|\\
  &\times\left[(1+t)^{-\frac{1}{2}(1+\vartheta)}\left\|\langle\xi\rangle w_{-\ell}\partial^{\alpha}u_2\right\|\right]^{\frac{1-\gamma}{2-\gamma}}
  \left\|\langle\xi\rangle^{\frac{\gamma}{2}}w_{-\ell}\partial^{\alpha}u_2\right\|^{\frac{1}{2-\gamma}}\\
   \lesssim&\sum_{|\alpha_1|=|\alpha|-1}(1+t)^{\frac{1}{2}(1+\vartheta)\times\frac{1-\gamma}{2-\gamma}}
  \left\|\partial^{\alpha-\alpha_1}\nabla_x\phi\right\|_{L^{\infty}}
  \left\|\langle\xi\rangle w_{1-\ell}\nabla_{\xi}\partial^{\alpha_1}u_2\right\|\\
  &\times\left[(1+t)^{-\frac{1}{2}(1+\vartheta)}
  \left\|\langle\xi\rangle w_{-\ell}\partial^{\alpha}u_2\right\|
  +\left\|w_{-\ell}\partial^{\alpha}u_2\right\|_{\nu}\right]\\
  \lesssim&\eta\left((1+t)^{-1-\vartheta}\left\|\langle\xi\rangle w_{-\ell}\partial^{\alpha}u_2\right\|^2
   +\left\|w_{-\ell}\partial^{\alpha}u_2\right\|^2_{\nu}\right)\\
   &+C_{\eta}\sum_{|\alpha_1|=|\alpha|-1}(1+t)^{(1+\vartheta)\frac{1-\gamma}{2-\gamma}}
   \left\|\partial^{\alpha-\alpha_1}\nabla_x\phi\right\|_{L^{\infty}}^2
   \left\|\langle\xi\rangle w_{1-\ell}\partial^{\alpha_1}\nabla_{\xi}u_2\right\|^2.
  \end{split}
\end{equation*}

Now we prove \eqref{good-4}, in fact, for any $|\alpha|+|\beta|\leq N,~|\beta|\geq1$, we have the following
\begin{eqnarray*}
  &&\sum\limits_{|\alpha_1|<|\alpha|}\left|\left(\partial^{\alpha-\alpha_1}\nabla_x\phi\cdot
  \nabla_{\xi}\partial^{\alpha_1}_{\beta}u_2,w_{|\beta|-\ell}^2\partial^{\alpha}_{\beta}u_2\right)\right|\\
  &\lesssim&\sum\limits_{\alpha_1<\alpha}\left|\left(\partial^{\alpha-\alpha_1}\nabla_x\phi\cdot
  \nabla_{\xi}\partial^{\alpha_1}_{\beta}u_2w_{|\beta+e_i|-\ell}\langle\xi\rangle^{\frac{\gamma}{2}},
  w_{|\beta|-\ell}\partial^{\alpha}_{\beta}u_2\right)\right|\\
  &\lesssim&\left(\sum_{|\alpha_1|=0}\Big[\chi_{|\alpha|=1}\left\|\nabla^2_x\phi\right\|_{L^{\infty}_x}
  \left\|\langle\xi\rangle w_{|\beta+e_i|-\ell}\nabla_{\xi}\partial_{\beta}u_2\right\|\right.
\end{eqnarray*}
\begin{eqnarray*}
  &&+\chi_{2\leq|\alpha|\leq N-2}\left\|\nabla_x\partial^{\alpha}\phi\right\|_{L^{3}_x}
  \left\|\langle\xi\rangle w_{|\beta+e_i|-\ell}\nabla_{\xi}\partial_{\beta}u_2\right\|_{L^6_x(L^2_{\xi})}\\
  &&+\chi_{|\alpha|=N-1,|\beta|=1}\left\|\nabla^N_x\phi\right\|
  \left\|\langle\xi\rangle w_{|\beta+e_i|-\ell}\nabla_{\xi}\partial_{\beta}u_2\right\|_{L^{\infty}_x(L^2_{\xi})}\Big]\\
  &&+\sum_{1\leq|\alpha_1|\leq |\alpha|-2}\chi_{|\alpha|\geq3}\left\|\partial^{\alpha-\alpha_1}\nabla_x\phi\right\|_{L^3}
  \left\|\langle\xi\rangle w_{|\beta+e_i|-\ell}\nabla_{\xi}\partial^{\alpha_1}_{\beta}u_2\right\|_{L^{6}_x(L^2_{\xi})}\\
  &&\left.+\sum_{|\alpha_1|=|\alpha|-1}\chi_{|\alpha|\geq1}
  \left\|\partial^{\alpha-\alpha_1}\nabla_x\phi\right\|_{L^{\infty}}
  \left\|\langle\xi\rangle w_{|\beta+e_i|-\ell}\nabla_{\xi}\partial^{\alpha_1}_{\beta}u_2\right\|\right)
  \left\|\langle\xi\rangle^{\frac{1}{2}}w_{|\beta|-\ell}\partial^{\alpha}_{\beta}u_2\right\|,
\end{eqnarray*}
which, by a similar interpolation techinuque as in the proof of \eqref{good-1}, can be further bounded by
\begin{eqnarray*}
  &&\eta\left((1+t)^{-1-\vartheta}\left\|\langle\xi\rangle w_{|\beta|-\ell}\partial^{\alpha}_{\beta}u_2\right\|^2
  +\left\|w_{|\beta|-\ell}\partial^{\alpha}_{\beta}u_2\right\|_{\nu}^2\right)\\
  &&+C_{\eta}\sum_{|\alpha_1|<|\alpha|}(1+t)^{-(|\alpha-\alpha_1|+2)+(1+\vartheta)\frac{1-\gamma}{2-\gamma}}
  \CE_{\infty}(t)\left\|\langle\xi\rangle w_{|\beta+e_i|-\ell}\nabla_{\xi}\partial^{\alpha_1}_{\beta}u_2\right\|^2\\
  &&+C_{\eta}\chi_{|\alpha|=N-1,|\beta|=1}(1+t)^{-N+\frac{1}{2}+(1+\vartheta)\frac{1-\gamma}{2-\gamma}}  \left\|\langle\xi\rangle w_{|\beta+e_i|-\ell}\nabla_{\xi}^2\nabla_{x}u_2\right\|\left\|\langle\xi\rangle w_{|\beta+e_i|-\ell}\nabla_{\xi}^2\nabla^2_{x}u_2\right\|\\
  &&+C_{\eta}\sum_{1\leq|\alpha_1|\leq |\alpha|-2}\chi_{|\alpha|\geq3}(1+t)^{-(|\alpha-\alpha_1|+1)+(1+\vartheta)\frac{1-\gamma}{2-\gamma}}
  \CE_{\infty}(t)\left\|\langle\xi\rangle w_{|\beta+e_i|-\ell}\nabla_{\xi}\partial^{\alpha_1+e_i}_{\beta}u_2\right\|^2\\
  &&+C_{\eta}\sum_{|\alpha_1|=|\alpha|-1}\chi_{|\alpha|\geq1}
  (1+t)^{-(|\alpha-\alpha_1|+2)+(1+\vartheta)\frac{1-\gamma}{2-\gamma}}
  \CE_{\infty}(t)\left\|\langle\xi\rangle w_{|\beta+e_i|-\ell}\nabla_{\xi}\partial^{\alpha_1}_{\beta}u_2\right\|^2\\
  &\lesssim&\eta\left((1+t)^{-1-\vartheta}\left\|\langle\xi\rangle w_{|\beta|-\ell}\partial^{\alpha}_{\beta}u_2\right\|^2
  +\left\|w_{|\beta|-\ell}\partial^{\alpha}_{\beta}u_2\right\|_{\nu}^2\right)\\
  &&+C_{\eta}\CE_{\infty}(t)\sum_{1\leq|\alpha_1|<|\alpha|}
  (1+t)^{-(|\alpha-\alpha_1|+2)+(1+\vartheta)\frac{1-\gamma}{2-\gamma}}
  \left\|\langle\xi\rangle w_{|\beta+e_i|-\ell}\nabla_{\xi}\partial^{\alpha_1}_{\beta}u_2\right\|^2.
 \end{eqnarray*}
Here to deduce the last inequality, we have used the following facts
\begin{equation*}
    \left\|\langle\xi\rangle w_{|\beta+e_i|-\ell}\nabla_{\xi}^2u_2\right\|^2_{L^{\infty}_x(L^2_{\xi})}
  \lesssim\left\|\langle\xi\rangle w_{|\beta+e_i|-\ell}\nabla_{x}\nabla_{\xi}^2u_2\right\|
  \left\|\langle\xi\rangle w_{|\beta+e_i|-\ell}\nabla_{x}^2\nabla_{\xi}^2u_2\right\|,
\end{equation*}
and
\begin{eqnarray*}
\begin{array}{rl}
 -\left(N-\frac{1}{2}\right)+(1+\vartheta)\frac{1-\gamma}{2-\gamma}=&
 \left\{-\frac{1}{2}\left[((N-1)-1+2)+(1+\vartheta)\frac{1-\gamma}{2-\gamma}\right]\right\}\\
 &+\left\{-\frac{1}{2}\left[((N-1)-2+2)+(1+\vartheta)\frac{1-\gamma}{2-\gamma}\right]\right\}.
\end{array}
\end{eqnarray*}

This completes the proof of Lemma \ref{non-LINEAR}.
\end{proof}
\end{lemma}

\section{Lower Order Energy Estimates}
The following two sections are concerned with certain energy type estimates on the solution $u(t,x,\xi)$ of the one-species VPB system \eqref{u} in terms of the temporal energy functional $\mathcal{E}_\infty(t)$. We first notice that the arguments used in \cite{Guo-CPAM-2002} and \cite{Guo-ARMA-2003} to deduce the corresponding local solvability, after a straightforward modification, can indeed be applied to the Cauchy problem of the one-species VPB system \eqref{u} for the whole range of cutoff intermolecular interactions to yield our desired local solvability result. Assume that such a local solution $u(t,x,\xi)$ has been extended to the time step $t=T$ for some $T>0$, that is, $u(t,x,\xi)$ is a solution of the Cauchy problem of the one-species VPB system \eqref{u} defined on the strip $\prod_T=[0,T]\times{\mathbb{R}_x^3}\times{\mathbb{R}_\xi^3}$. Now we turn to deduce certain energy type estimates on $u(t,x,\xi)$
in terms of the temporal energy functional $\mathcal{E}_\infty(t)$ defined by \eqref{Energy}, under the assumption that the temporal energy functional $\mathcal{E}_\infty(t)$ is sufficiently small for all $0\leq t\leq T$.

The desired energy type estimates will be carried out in this and the coming sections and the main purpose of this section is concerned with the lower order energy estimates. To make the presentation easy to follow, we divide this section into two subsections. The first subsection is about the non-weighted estimates.

\subsection{Non-weighted Estimates}
The main purpose of this subsection is to prove the following lemma which yields the optimal temporal decay estimates on the $L^2-$norm of the pure spatial derivatives of $u(t,x,\xi)$ up to the order $N-1$. Such an estimate will play an important role in our analysis.
\begin{lemma}\label{non-weight} Let $\frac 34<p<1,$ $\ell\geq\max\left\{l_2+1,l_2-\frac 2\gamma, \frac{4p-2}{4p-3}l_2+1\right\}$, $l_2>N+\frac 12$. Then for any $0\leq m\leq N-1$, the solution $(u(t,x,\xi),\phi(t,x))$ of the Cauchy problem of the one-species VPB system \eqref{u} satisfies the following decay estimates
  \begin{equation*}
    \left\|\nabla^m_xu\right\|^2+\left\|\nabla_x^{m+1}\phi\right\|^2\lesssim (1+t)^{-\left(m+\frac{1}{2}\right)}\left(\epsilon_0+\CE^2_{\infty}(t)\right)
  \end{equation*}
 for all $0\leq t\leq T$ provided that $(u(t,x,\xi),\phi(t,x))$ is assumed to satisfy the a priori assumption
\begin{equation}\label{a priori assumption}
\mathcal{E}_\infty(t)\leq \delta,\quad 0\leq t\leq T
\end{equation}
for some sufficiently small $\delta>0$. Here $\epsilon_0$ is defined in the statement of our Theorem \ref{Thm} which measures the smallness of the initial perturbation.
\end{lemma}

 Although the main ideas to deduce this lemma are along the same line as in \cite{Duan-Strain-ARMA-2011} for the hard sphere model and \cite{Xiao-Xiong-Zhao-JDE-2013} for cutoff moderately soft potentials, that is, by combining the temporal decay estimates on the solution operator of the corresponding linearized system of the one-species VPB system \eqref{u} together with Duhamel's principle, the analysis for the whole range of cutoff soft potentials is quite complex since we try to deduce the optimal temporal decay estimates on the $L^2-$norm of the pure spatial derivatives of $u(t,x,\xi)$ up to the order $N-1$, thus we prove it in details in the following. To this end, we first deduce some estimates on $G(t,x,\xi)=\frac{1}{2}\xi \cdot\nabla_{x}\phi u-\nabla_{x}\phi\cdot\nabla_{\xi }u+{\Gamma}(u,u)$.
\begin{lemma} If $\frac 34<p<1,$ $\ell\geq\max\left\{l_2+1,l_2-\frac 2\gamma, \frac{4p-2}{4p-3}l_2+1\right\}$, we have the following estimates on $G(t,x,\xi)$:
\begin{itemize}
\item For $0\leq m\leq N-1$, it holds that
\begin{equation}\label{2-z}
\left\|\langle\xi\rangle^{-\frac{\gamma}{2}l_2}\nabla^m_xG(t)\right\|^2
\lesssim(1+t)^{-\left(m+\frac{5}{2}-\frac{5}{6}\right)}\CE^2_{\infty}(t);
\end{equation}
\item For $0\leq m\leq N-2$, one has
\begin{equation}\label{2-z1}
\left\|\langle\xi\rangle^{-\frac{\gamma}{2}l_2}\nabla^m_xG(t)\right\|_{Z_1}^2
\lesssim (1+t)^{-(m+1)}\CE^2_{\infty}(\tau);
\end{equation}
\item For the case of $m=N-1$, if we decompose $G$ as $G=G_1+G_2$ where
$G_1=\frac{1}{2}\xi \cdot\nabla_{x}\phi u-\nabla_{x}\phi\cdot\nabla_{\xi}{\bf P}u+{\Gamma}(u,u)$ and
$G_2=-\nabla_{x}\phi\cdot\nabla_{\xi}{\{\bf I-P\}}u$, then we can get that
\begin{eqnarray}\label{N-1_z}
\left\|\langle\xi\rangle^{-\frac{\gamma}{2}l_2}\nabla^{N-1}_xG_1(t)\right\|_{Z_1}^2&\lesssim&(1+t)^{-N}\CE^2_{\infty}(\tau),\\
\left\|\langle\xi\rangle^{-\frac{\gamma}{2}l_2}\nabla^{N-2}_xG_2(t)\right\|_{Z_1}^2
&\lesssim&(1+t)^{-N+\frac{1}{2}}\CE^2_{\infty}(t).\nonumber
\end{eqnarray}
\end{itemize}
\end{lemma}
\begin{proof} The most difficult cases lie in how to deal with the highest order derivatives of $G(t,x,\xi)$ with respect to $x$ and $\xi$.

Noticing that $G(t,x,\xi)=\frac{1}{2}\xi \cdot\nabla_{x}\phi u-\nabla_{x}\phi\cdot\nabla_{\xi }u+{\Gamma}(u,u)$, we can get for $0\leq m\leq N-1$ that
\begin{equation}\label{2-z-proof}
\begin{split}
\left\|\langle\xi\rangle^{-\frac{\gamma}{2}l_2}\nabla^m_xG(t)\right\|^2
&\lesssim\sum_{m_1\leq\frac{m}{2}}
\left\|\nabla^{m_1+1}_x\phi\right\|_{L^{\infty}_x}^2
\left(\left\|\langle\xi\rangle^{-\frac{\gamma}{2}l_2+1}\nabla^{m-m_1}_xu\right\|^2
+\left\|\langle\xi\rangle^{-\frac{\gamma}{2}l_2}\nabla^{m-m_1}_x\nabla_{\xi}u\right\|^2\right)\\
+&\sum_{m_1>\frac{m}{2}}\left\|\nabla^{m_1+1}_x\phi\right\|^2
\left(\left\|\langle\xi\rangle^{-\frac{\gamma}{2}l_2+1}\nabla^{m-m_1}_xu\right\|_{L^{\infty}_x(L^2_{\xi})}^2
+\left\|\langle\xi\rangle^{-\frac{\gamma}{2}l_2}\nabla^{m-m_1}_x\nabla_{\xi}u\right\|_{L^{\infty}_x(L^2_{\xi})}^2\right)\\
&+\left\|\langle\xi\rangle^{-\frac{\gamma}{2}l_2}\nabla^m_x\Gamma(u,u)\right\|^2.
\end{split}
\end{equation}
Similarly, one can get for $0\leq m\leq N-2$ that
\begin{equation}\label{2-z1-proof}
\begin{split}
\left\|\langle\xi\rangle^{-\frac{\gamma}{2}l_2}\nabla^m_xG(t)\right\|_{Z_1}^2
\lesssim&\sum_{m_1\leq m}\left\|\nabla^{m_1+1}_x\phi\right\|^2
\left(\left\|\langle\xi\rangle^{-\frac{\gamma}{2}l_2}\nabla^{m-m_1}_xu\right\|^2
+\left\|\langle\xi\rangle^{-\frac{\gamma}{2}l_2}\nabla^{m-m_1}_x\nabla_{\xi}u\right\|^2\right)\\
&+\left\|\langle\xi\rangle^{-\frac{\gamma}{2}l_2}\nabla^m_x\Gamma(u,u)\right\|_{Z_1}^2.
\end{split}
\end{equation}
To estimate the terms in the right hand sides of \eqref{2-z-proof} and \eqref{2-z1-proof} in terms of $\CE_{\infty}(t)$, if we let $\ell\geq\max\left\{l_2+1,l_2-\frac 2\gamma\right\}$, we can first deduce from the definition of $\CE_{\infty}(t)$  that
\begin{eqnarray*}
\left\|\nabla^{m_1}_x\nabla_x\phi(t)\right\|^2&\lesssim& (1+t)^{-(m_1+\frac{1}{2})}\CE_{\infty}(t),\quad m_1\leq N-1,\\
\left\|\langle\xi\rangle^{-\frac{\gamma}{2}l_2+1}\nabla^{m-m_1}_xu(t)\right\|^2&\lesssim& (1+t)^{-\left(m-m_1+\frac{1}{2}\right)}\CE_{\infty}(t),\quad m-m_1\leq N-1,\\
\left\|\langle\xi\rangle^{-\frac{\gamma}{2}l_2}\nabla^{m-m_1}_x\nabla_{\xi}u(t)\right\|^2
  &\lesssim&
  \left\{
  \begin{array}{l}
  (1+t)^{-\left(m-m_1+\frac{1}{2}\right)}\CE_{\infty}(t),\quad m-m_1\leq N-2,\\[2mm]
  (1+t)^{-\left(m-m_1+\frac 12-\frac 23\right)}\CE_{\infty}(t),\quad m-m_1= N-1.
  \end{array}
  \right.
\end{eqnarray*}
Secondly, for the case $0\leq m\leq N-1$, $m_1\leq \frac m2$, we can deduce that
$$
m_1+2\leq \frac{N-1}{2}+2=\frac{N+3}{2}<N,
$$
we thus get from the definition of $\CE_{\infty}(t)$ and the Gagliardo-Nirenberg interpolation inequality that
\begin{eqnarray*}
\left\|\nabla^{m_1}_x\nabla_x\phi(t)\right\|_{L^{\infty}_x}^2\lesssim\left\|\nabla^{m_1+1}_x\phi(t)\right\|^{\frac 12}\left\|\nabla^{m_1+3}_x\phi(t)\right\|^{\frac 32}
\lesssim(1+t)^{-(m_1+2)}\CE_{\infty}(t).
\end{eqnarray*}
Thirdly, to deduce an estimate on $\left\|\langle\xi\rangle^{-\frac{\gamma}{2}l_2+1}\nabla^{m-m_1}_xu\right\|_{L^{\infty}_x(L^2_{\xi})}$ and $\left\|\langle\xi\rangle^{-\frac{\gamma}{2}l_2}\nabla^{m-m_1}_x\nabla_{\xi}u\right\|_{L^{\infty}_x(L^2_{\xi})}$ for the case $m_1>\frac m2$ and $0\leq m\leq N-1$. In such a case, since $m-m_1+2<\frac m2+2\leq \frac{N-1}{2}+2=\frac{N+3}{2}<N$, we can also get from the definition of $\CE_{\infty}(t)$, $\ell\geq\max\left\{l_2+1,l_2-\frac 2\gamma\right\}$, and the Gagliardo-Nirenberg interpolation inequality that
\begin{eqnarray*}
\left\|\langle\xi\rangle^{-\frac{\gamma}{2}l_2+1}\nabla^{m-m_1}_xu(t)\right\|_{L^{\infty}_x(L^2_{\xi})}^2
  &\lesssim& \left\|\langle\xi\rangle^{-\frac{\gamma}{2}l_2+1}\nabla^{m-m_1}_xu(t)\right\|^{\frac 12}\left\|\langle\xi\rangle^{-\frac{\gamma}{2}l_2+1}\nabla^{m-m_1+2}_xu(t)\right\|^{\frac 32}\\
  &\lesssim&  (1+t)^{-(m-m_1+2)}\CE_{\infty}(t)
\end{eqnarray*}
and
\begin{equation*}
 \begin{split}
  \left\|\langle\xi\rangle^{-\frac{\gamma}{2}l_2}\nabla^{m-m_1}_x\nabla_{\xi}u(t)\right\|_{L^{\infty}_x(L^2_{\xi})}^2
  \lesssim&\left\|\langle\xi\rangle^{-\frac{\gamma}{2}l_2}\nabla^{m-m_1+1}_x\nabla_{\xi}u(t)\right\|
  \left\|\langle\xi\rangle^{-\frac{\gamma}{2}l_2}\nabla^{m-m_1+2}_x\nabla_{\xi}u(t)\right\|\\
  \lesssim&
      (1+t)^{-\frac{1}{2}\left(m-m_1+1+\frac{1}{2}\right)}(1+t)^{-\frac{1}{2}\left(m-m_1+2+\frac{1}{2}-\frac{2}{3}\right)}
    \CE_{\infty}(t)\\
   \lesssim& (1+t)^{-\left(m-m_1+2-\frac 13\right)}\CE_{\infty}(t).
 \end{split}
\end{equation*}
Here we have used the fact that $m-m_1+2\leq N-1$ and $m-m_1+3\leq N$ but $m-m_1+2\leq N-1$.

On the other hand, by exploiting Lemma \ref{JDE-lem.} and Lemma \ref{Gamma}, $\ell\geq\max\left\{l_2+1,l_2-\frac 2\gamma\right\}$, and the definition of $\CE_{\infty}(t)$, one has
\begin{equation*}
 \left\|\langle\xi\rangle^{-\frac{\gamma}{2}l_2}\nabla^m_x\Gamma(u,u)\right\|^2
 \lesssim
 \begin{cases}
   (1+t)^{-\left(m+\frac{5}{2}-\frac 16\right)}\CE_{\infty}^2(t),\quad \text{if}~0\leq m\leq N-2\\
   (1+t)^{-\left(m+\frac{5}{2}-\frac{5}{6}\right)}\CE_{\infty}^2(t),\quad\text{if}~m=N-1.
 \end{cases}
\end{equation*}
Substituting the above estimates into \eqref{2-z-proof} and \eqref{2-z1-proof}, we can get \eqref{2-z} and \eqref{2-z1} immediately.

At last, we prove \eqref{N-1_z}. To this end, noticing
$G_1=\frac{1}{2}\xi \cdot\nabla_{x}\phi u-\nabla_{x}\phi\cdot\nabla_{\xi}u_1+{\Gamma}(u,u)$,
$$
\left\|\langle\xi\rangle^{-\frac{\gamma}{2}l_2}\nabla^{(N-1)-m_1}_x\nabla_{\xi}u_1(t)\right\|\lesssim
\left\|\nabla^{(N-1)-m_1}_xu_1(t)\right\|,
$$
and due to
$N-1-m_1+1\leq N$ with $N-1-m_1\leq N-1$, one can get from the estimate \eqref{Z_1-on-Gamma} and the definition of $\CE_{\infty}(t)$ that \eqref{N-1_z}$_1$ can be bounded by
\begin{equation*}
\begin{split}
\left\|\langle\xi\rangle^{-\frac{\gamma}{2}l_2}\nabla^{N-1}_xG_1(t)\right\|_{Z_1}^2\lesssim&\sum_{m_1\leq
N-1}
\left(\left\|\langle\xi\rangle^{-\frac{\gamma}{2}l_2+1}\nabla^{(N-1)-m_1}_xu(t)\right\|^2
+\left\|\langle\xi\rangle^{-\frac{\gamma}{2}l_2}\nabla^{(N-1)-m_1}_x\nabla_{\xi}u_1(t)\right\|^2\right)\\
&\times\left\|\nabla^{m_1+1}_x\phi(t)\right\|^2
+\left\|\langle\xi\rangle^{-\frac{\gamma}{2}l_2}\nabla^{N-1}_x\Gamma(u,u)\right\|_{Z_1}^2\\
\lesssim&(1+t)^{-m_1-\frac 12}(1+t)^{-\left(N-1-m_1+\frac 12\right)}\CE^2_{\infty}(t)+(1+t)^{-N}\CE^2_{\infty}(t)\\
\lesssim&(1+t)^{-N}\CE^2_{\infty}(t).
\end{split}
\end{equation*}
This proves \eqref{Z_1-on-Gamma}$_1$.

For \eqref{Z_1-on-Gamma}$_2$, notice that the standard interpolation argument together with the last term in definition of $\CE_{\infty}$ tell us that
\begin{eqnarray*}
\left\|\langle\xi\rangle^{-\frac{\gamma}{2}l_2}\nabla^{N-2}_x\nabla_{\xi}u_2\right\|^2
  &\lesssim&\left\|w_{1-\ell}\nabla^{N-2}\nabla_{\xi}u_2\right\|^{\frac{2l_2}{\ell-1}}
  \left\|\nabla^{N-2}\nabla_{\xi}u_2\right\|^{2\left(1-\frac{l_2}{\ell-1}\right)}\\
  &\lesssim&(1+t)^{-\left(N-\frac{3}{2}\right)\frac{l_2}{\ell-1}-\left[\left(N-\frac{1}{2}\right)-2(1-p)\right]
  \left(1-\frac{l_2}{\ell-1}\right)}
  \CE_{\infty}(t)\\
   &\lesssim&(1+t)^{-N+1}\CE_{\infty}(t)
 \end{eqnarray*}
hold if $\frac{3}{4}<p<1$ and $\ell\geq \frac{4p-2}{4p-3}l_2+1$ and from which we can get that
\begin{equation*}
\begin{split}
\left\|\langle\xi\rangle^{-\frac{\gamma}{2}l_2}\nabla^{N-2}_xG_2(t)\right\|_{Z_1}^2
\lesssim&\sum_{m_1\leq N-2}
\left\|\langle\xi\rangle^{-\frac{\gamma}{2}l_2}\nabla^{(N-2)-m_1}_x\nabla_{\xi}u_2(t)\right\|^2
\left\|\nabla^{m_1+1}_x\phi(t)\right\|^2\\
\lesssim&(1+t)^{-N+\frac{1}{2}}\CE^2_{\infty}(t).
\end{split}
\end{equation*}
This is exactly \eqref{N-1_z}$_2$ and the proof of Lemma 3.2 is complete.
\end{proof}

With Lemma 3.2 in hand, we now turn to prove Lemma \ref{non-weight}.
\begin{proof}
Firstly, by Duhamel's principle we can write the solution $u(t,x,\xi)$ to \eqref{u} as the solution of the following integral equation
\begin{equation}\label{Duhamel}
u(t)=e^{tB}u_0+\displaystyle\int_0^te^{(t-\tau)B}G(\tau)d\tau.
\end{equation}
Here $G=\frac{1}{2}\xi \cdot\nabla_{x}\phi u-\nabla_{x}\phi\cdot\nabla_{\xi }u+{\Gamma}(u,u)$ and $e^{tB}u_0$ denotes the  solution operator of the following linearized equation of the one-species VPB system \eqref{VPB}
\begin{eqnarray*}\label{linearized-solution-operator}
\left\{
\begin{array}{l}
u_t-Bu=0,\quad Bu=-\xi\cdot\nabla_x u+\nabla_x\phi\cdot\xi {\bf M}^{\frac 12}+{\bf L}u,\\[2mm]
\Delta_x\phi={\displaystyle\int_{{\mathbb{R}^3}}}{\bf M}^{\frac 12}ud\xi,\\[2mm]
u(0,x,\xi)=u_0(x,\xi),
\end{array}
\right.
\end{eqnarray*}
whose temporal decay estimates will be be given in Lemma \ref{lem.-decay}.

Therefore, applying Lemma \ref{lem.-decay} (with $l_0=0$) to \eqref{Duhamel} yields
\begin{eqnarray}\label{m}
&&\left\|\nabla^m_xu(t)\right\|^2+\left\|\nabla^{m+1}_x\phi(t)\right\|^2\\
&\lesssim&(1+t)^{-m-\frac{1}{2}}\left(\left\|\langle\xi\rangle^{-\frac{\gamma}{2}l_2}\nabla^{m}_xu_0\right\|^2
+\left\|\langle\xi\rangle^{-\frac{\gamma}{2}l_2}u_0\right\|^{2}_{Z_1}\right)
+\left\{\int_0^t\left\|\nabla^{m}_xe^{(t-\tau)B}G(\tau)\right\|d\tau\right\}^2.\nonumber
\end{eqnarray}
Noticing ${\bf P}_0G=0$, for any $t>0,~k\in{\mathbb{ R}}^3$, it further follows from (\ref{G-decay}) in  Lemma \ref{lem.-decay} that
\begin{equation*}
 \left|\widehat{G}(t,k,\xi)\right|^2_{L^2_{\xi}}\lesssim\left\{1+\frac{|k|^2}{1+|k|^2}(t-\tau)\right\}^{-l_2}
 \left|\langle\xi\rangle^{-\frac{\gamma}{2}l_2}\widehat{G}(\tau,k,\xi)\right|^2_{L^2_{\xi}}.
\end{equation*}
Consequently, we can deduce that
\begin{eqnarray}\label{trouble}
&&\int_0^t\left\|\nabla^{m}_xe^{(t-\tau)B}G(\tau)\right\|d\tau\nonumber\\
&\lesssim&\int_0^t\left\{\int_{{\bf R}^3}|k|^{2m}\left(1+\frac{|k|^2(t-\tau)}{1+|k|^2}\right)^{-l_2}
\left|\langle\xi\rangle^{-\frac{\gamma}{2}l_2}
\widehat{G}(\tau,k,\xi)\right|^2_{L^2_{\xi}}dk\right\}^{1/2}d\tau\nonumber\\
&\lesssim&\underbrace{\int_0^{t/2}\left\{\int_{|k|\leq 1}|k|^{2m}\left(1+|k|^2(1+t-\tau)\right)^{-l_2}dk\cdot
\sup_{k}\left\{\left|\langle\xi\rangle^{-\frac{\gamma}{2}l_2}
\widehat{G}(\tau,k,\xi)\right|^2_{L^2_{\xi}}\right\}\right\}^{1/2}d\tau}_{K_1}\\
&&+\underbrace{\int_{t/2}^t\left\{\int_{|k|\leq 1}|k|^{2m}\left(1+|k|^2(1+t-\tau)\right)^{-l_2}
\left|\langle\xi\rangle^{-\frac{\gamma}{2}l_2}
\widehat{G}(\tau,k,\xi)\right|^2_{L^2_{\xi}}dk\right\}^{1/2}d\tau}_{K_2}\nonumber\\
&&+\underbrace{\int_0^t(1+t-\tau)^{-l_2}
\left\|\langle\xi\rangle^{-\frac{\gamma}{2}l_2}\nabla_x^mG(\tau)\right\|d\tau}_{K_3}.\nonumber
\end{eqnarray}

To deduce desired bounds on $K_1$, $K_2$, and $K_3$, for the first term $K_1$, by \eqref{2-z1} with $m=0$ and by noticing that
\begin{equation*}
\sup_{k}\left\{\left|\langle\xi\rangle^{-\frac{\gamma}{2}l_2}\widehat{G}(\tau,k,\xi)\right|_{L^2_{\xi}}\right\}
\lesssim\left\|\langle\xi\rangle^{-\frac{\gamma}{2}l_2}G(\tau)\right\|_{Z_1},
\end{equation*}
it is easy to bound $K_1$ by
\begin{eqnarray*}
K_1&\lesssim& \CE_{\infty}(t)\left\{\int_0^{t/2}(1+t-\tau)^{-\frac{1}{2}\left(m+\frac{3}{2}\right)}(1+\tau)^{-\frac{1}{2}}d\tau \sup_{0\leq\tau\leq t/2}\int_0^{(1+t-\tau)^{1/2}}y^{2m+2}(1+y^2)^{-l_2}dy\right\}^{\frac{1}{2}}\\
    &\lesssim&(1+t)^{-\frac{1}{2}\left(m+\frac{1}{2}\right)}\CE_{\infty}(t)
\end{eqnarray*}
for $l_2>m+\frac{3}{2}$. And for $K_3$, since $0\leq m\leq N-1$, one can also get from \eqref{2-z} that
\begin{equation*}
  K_3\lesssim \CE_{\infty}(t)\int_0^t(1+t-\tau)^{-l_2}
  (1+\tau)^{-\frac{1}{2}\left(m+\frac{5}{2}-\frac{5}{6}\right)}d\tau
    \lesssim(1+t)^{-\frac{1}{2}\left(m+\frac{1}{2}\right)}\CE_{\infty}(t)
\end{equation*}
provided that $l_2>m+\frac{3}{2}$.

It is more subtle to deal with the term $K_2$. For this purpose, we divide the estimation of this term into two cases:  For the case $m\leq N-2$, similar to that of $K_1$, one use \eqref{2-z} to deduce that
\begin{eqnarray*}
   K_2&\lesssim&\int_{t/2}^t\left\{\int_{|k|\leq 1}\left(1+|k|^2(1+t-\tau)\right)^{-l_2}dk
  \cdot\sup_{k}\left\{|k|^{2m}
  \left|\langle\xi\rangle^{-\frac{\gamma}{2}l_2}\widehat{G}(\tau,k,\xi)\right|^2_{L^2_{\xi}}\right\}\right\}^{1/2}d\tau\\
  &\lesssim&\CE_{\infty}(t)\int_{t/2}^t(1+t-\tau)^{-\frac{3}{4}}(1+\tau)^{-\frac{1}{2}(m+1)}d\tau\\
  &\lesssim&(1+t)^{-\frac{1}{2}\left(m+\frac{1}{2}\right)}\CE_{\infty}(t),
 \end{eqnarray*}
while for the case of $m=N-1$, recalling the decomposition
$G=G_1+G_2$ with $G_1=\frac{1}{2}\xi \cdot\nabla_{x}\phi
u-\nabla_{x}\phi\cdot\nabla_{\xi}u_1+{\Gamma}(u,u)$ and
$G_2=-\nabla_{x}\phi\cdot\nabla_{\xi}u_2$, we use
 \eqref{N-1_z} and the fact
$$
\sup_{k}\left\{|k|^{2m}
  \left|\langle\xi\rangle^{-\frac{\gamma}{2}l_2}\widehat{G_2}(\tau,k,\xi)\right|^2_{L^2_{\xi}}\right\}
\lesssim \left\|\langle\xi\rangle^{-\frac{\gamma}{2}l_2}\nabla_x^mG_2(\tau)\right\|^2_{Z_1}
$$
 to obtain
\begin{eqnarray*}
   K_2&\lesssim&\int_{t/2}^t\left\{\int_{|k|\leq 1}\left(1+|k|^2(1+t-\tau)\right)^{-l_2}dk
  \cdot\sup_{k}\left\{|k|^{2(N-1)}
  \left|\langle\xi\rangle^{-\frac{\gamma}{2}l_2}\widehat{G_1}(\tau,k,\xi)\right|^2_{L^2_{\xi}}\right\}\right\}^{1/2}d\tau\\
  &&+\int_{t/2}^t\left\{\int_{|k|\leq 1}|k|^2\left(1+|k|^2(1+t-\tau)\right)^{-l_2}dk
  \cdot\sup_{k}\left\{|k|^{2(N-2)}
  \left|\langle\xi\rangle^{-\frac{\gamma}{2}l_2}\widehat{G_2}(\tau,k,\xi)\right|^2_{L^2_{\xi}}\right\}\right\}^{1/2}d\tau\\
  &\lesssim&\CE_{\infty}(t)\left(\int_{t/2}^t(1+t-\tau)^{-\frac{3}{4}}(1+\tau)^{-\frac{N}{2}}d\tau
  +\int_{t/2}^t(1+t-\tau)^{-\frac{5}{4}}(1+\tau)^{-\frac{1}{2}(N-\frac{1}{2})}d\tau\right)\\
  &\lesssim&(1+t)^{-\frac{1}{2}\left(N-1+\frac{1}{2}\right)}\CE_{\infty}(t)\\
  &=&(1+t)^{-\frac{1}{2}\left(m+\frac{1}{2}\right)}\CE_{\infty}(t).
 \end{eqnarray*}
 Here we have used the facts that $1/(1+|k|^2)\geq 1/2$ when $|k|\leq1$, $|k|^2/(1+|k|^2)\geq1/2$ when $|k|\geq1$, and the assumption $l_2>\frac{3}{2}+m$.

Inserting the above estimates on $K_1, K_2$, and $K_3$ into \eqref{m} yields our desired Lemma \ref{non-weight}.
\end{proof}

\subsection{Weighted Estimates}
This subsection is devoted to deducing the desired weighted estimates on $u(t,x,\xi)$. Noticing the fact that for the weight function $w_{|\beta|-\ell}(t,\xi)$ defined in \eqref{weight},
\begin{equation}\label{macro}
\left\|w_{|\beta|-\ell}\partial^{\alpha}_{\beta}{\bf P}u\right\|^2\sim \left\|\partial^{\alpha}{\bf P}u\right\|^2
\end{equation}
provided that the parameter $q>0$ in \eqref{weight} is chosen sufficiently small, thus we only need to  deduce the temporal decay estimates on the microscopic part $u_2={\bf \{I-P\}}u$. For this purpose, by applying ${\bf I-P}$ to the first equation of \eqref{u}, we can get that the time evolution of $u_2\equiv{\bf \{I-P\}}u$ can be described by
\begin{equation}\label{u2}
  \partial_tu_2+\xi\cdot\nabla_x u_2+\nabla_x\phi\cdot\nabla_{\xi}u_2-\frac{1}{2}\xi\cdot\nabla_x\phi u_2-{\bf L}u=\Gamma(u,u)+[[{\bf P},\tau_{\phi}]]u,
\end{equation}
where $[[{\bf P},\tau_{\phi}]]={\bf P}\tau_{\phi}-\tau_{\phi}{\bf P}$ denotes the commutator of two operators ${\bf P}$ and $\tau_{\phi}$ with $\tau_{\phi}$ being given by
\begin{equation*}
\tau_{\phi}=\xi\cdot\nabla_x+\nabla_x\phi\cdot\nabla_{\xi}-\frac{1}{2}\xi \cdot\nabla_x\phi.
\end{equation*}
For this commutator, we have the following estimate
\begin{lemma}\label{comm.-est.}
  For any $|\alpha|+|\beta|\leq N-1$, we have
  \begin{equation*}
    \left\|\nu^{-\frac{1}{2}}w_{|\beta|-\ell}\partial^{\alpha}_{\beta}
    \left[[{\bf P},\tau_{\phi}]\right]u(t)\right\|^2
    \lesssim(1+t)^{-\left(|\alpha|+\frac{5}{2}\right)}\CE^2_{\infty}(t)
    +\left\|\nabla_x^{|\alpha|+1}u\right\|_{\nu}^2.
  \end{equation*}
\end{lemma}
\begin{proof} From (\ref{macro}), we can deduce Lemma \ref{comm.-est.} immediately by employing Lemma \ref{non-weight}.
\end{proof}
We now deduce the desired weighted energy type estimates on $u_2$. Before this, we need the estimate of $u_2$ without weight.  To this end, we multiply \eqref{u2} by $u_2$ and integrate the resulting identity over $\R^3_{x}\times\R^3_{\xi}$ to yield
\begin{equation}\label{0-u2}
 \begin{split}
  \frac{1}{2}\frac{d}{dt}\|u_2(t)\|^2+\eta_0\|u_2\|^2_{\nu}
  \lesssim& (1+t)^{-\frac{3}{2}}\CE^2_{\infty}(t)
  +\left\|\nabla_x\phi\right\|_{L^{\infty}}^{\frac{2-\gamma}{1-\gamma}}\left\|\langle\xi\rangle u_2\right\|^2
  +|(\left[[{\bf P},\tau_{\phi}]\right]u(t),u_2)|\\
  \lesssim&(1+t)^{-\frac{3}{2}}\left(\epsilon_0+\CE^2_{\infty}(t)\right)
  +\left\|\nabla_x\phi\right\|_{L^{\infty}}^{\frac{2-\gamma}{1-\gamma}}\left\|\langle\xi\rangle u_2\right\|^2.
 \end{split}
\end{equation}
Here we have used \eqref{coercive}, Lemma \ref{Gamma}, Lemma \ref{comm.-est.}, Lemma\ref{non-weight} and the following facts for any sufficiently small $\eta>0$,
\begin{eqnarray}
\langle\xi\rangle=\left(\langle\xi\rangle^2\right)^{\frac{1-\gamma}{2-\gamma}} \nu(\xi)^{\frac{1}{2-\gamma}} \leq \eta\nu(\xi)+C_{\eta}\langle\xi\rangle^2, \label{basic1}\\
  \left(\tfrac{1}{2}\xi\cdot\nabla_x\phi u_2,u_2\right)\lesssim \eta\|u_2\|_{\nu}^2
  +C_{\eta}\|\nabla_x\phi\|_{L^{\infty}}^{\frac{2-\gamma}{1-\gamma}}\left\|\langle\xi\rangle u_2\right\|^2. \nonumber\\
  |(\left[[{\bf P},\tau_{\phi}]\right]u(t),u_2)|\lesssim \eta\|u_2\|_{\nu}^2+C_{\eta}\|\nu^{-\frac{1}{2}}
  \left[[{\bf P},\tau_{\phi}]\right]u(t)\|^2.\nonumber
\end{eqnarray}

Similarly, from Lemma \ref{L-K}, Lemma \ref{Gamma}, Lemma \ref{non-weight}, and Lemma \ref{comm.-est.}, we multiply \eqref{u2} by $w_{-\ell}^2(t,\xi)u_2$ and integrate the final result with respect to $x$ and $\xi$ over $\R^3_{x}\times\R^3_{\xi}$ to deduce that
\begin{equation}\label{0-w-u2}
 \begin{split}
   &\frac{1}{2}\frac{d}{dt}\left\|w_{-\ell}u_2(t)\right\|^2+2q\vartheta(1+t)^{-1-\vartheta}
   \left\|\langle\xi\rangle w_{-\ell} u_2\right\|^2+\frac{1}{2}\left\|w_{-\ell}u_2\right\|^2_{\nu}\\
   \lesssim&\|u_2\|_{\nu}^2+(1+t)^{-\frac{3}{2}}\CE^2_{\infty}(t)
   +\left\|\nabla_x\phi\right\|_{L^{\infty}}^{\frac{2-\gamma}{1-\gamma}}\left\|\langle\xi\rangle w_{-\ell} u_2\right\|^2,
 \end{split}
\end{equation}
Here we have taken $\eta=\frac{1}{2}$ in (\ref{L}) of Lemma 2.1 and used the fact that
\begin{equation*}
  \left(\frac{d}{dt}u_2(t),w_{-\ell}^2(t,\xi)u_2(t)\right)=\frac{1}{2}\frac{d}{dt}\left\|w_{-\ell}u_2(t)\right\|^2
  +\frac{2\vartheta q}{(1+t)^{1+\vartheta}}\left\|\langle\xi\rangle w_{-\ell}u_2(t)\right\|^2.
\end{equation*}
 Choosing $M>0$ suitably large, we take $M\times\eqref{0-u2}+\eqref{0-w-u2}$ to have
\begin{equation}\label{0-u2-cover}
 \begin{split}
   &\frac{d}{dt}\left(M\left\|u_2(t)\right\|^2+\left\|w_{-\ell}u_2(t)\right\|^2\right)+\eta_0\left((1+t)^{-1-\vartheta}
   \left\|\langle\xi\rangle w_{-\ell} u_2\right\|^2+\left\|w_{-\ell}u_2\right\|^2_{\nu}\right)\\
   \lesssim&(1+t)^{-\frac{3}{2}}\left(\epsilon_0+\CE^2_{\infty}(t)\right).
 \end{split}
\end{equation}
Here $\eta_0$ is some constant and we have used the following estimate:
\begin{equation*}
  \left\|\nabla_x\phi\right\|_{L^{\infty}}^{\frac{2-\gamma}{1-\gamma}}\lesssim (1+t)^{-1\times\frac{2-\gamma}{1-\gamma}}\CE_{\infty}^{\frac{1}{2}\times\frac{2-\gamma}{1-\gamma}}(t)
  \lesssim(1+t)^{-\frac{5}{4}}\CE_{\infty}^{\frac{2-\gamma}{2(1-\gamma)}}(t),\quad \gamma\in(-3,0),
\end{equation*}
by the definition of $\CE_\infty(t)$ in \eqref{Energy}.

Integrating \eqref{0-u2-cover} over $[0,t]$ yields
\begin{equation}\label{0-decay1}
  \left\|w_{-\ell}u_2(t)\right\|^2+\int_0^t\left((1+s)^{-1-\vartheta}
   \left\|\langle\xi\rangle w_{-\ell}u_2(s)\right\|^2+\left\|w_{-\ell}u_2(s)\right\|^2_{\nu}\right)ds
   \lesssim \epsilon_0+\CE^2_{\infty}(t).
\end{equation}
Now multiplying \eqref{0-u2-cover} by $(1+t)^{\frac{1}{2}+\vartheta}$, we can get from Lemma \ref{in.} with $m=-\frac{1}{2}$ that
\begin{eqnarray*}
   &&\frac{d}{dt}\left((1+t)^{\frac{1}{2}+\vartheta}\left(M\left\|u_2(t)\right\|^2
   +\left\|w_{-\ell}u_2(t)\right\|^2\right)\right)\\
   &&+(1+t)^{-\frac{1}{2}}\left\|\langle\xi\rangle w_{-\ell} u_2\right\|^2
   +\lambda(1+t)^{\frac{1}{2}+\vartheta}\left\|w_{-\ell}u_2\right\|^2_{\nu}\\
   &\lesssim&(1+t)^{-\frac{1}{2}+\vartheta}\left\|w_{-\ell}u_2(t)\right\|^2
   +(1+t)^{-1+\vartheta}\left(\epsilon_0+\CE^2_{\infty}(t)\right)\\
   &\lesssim&\eta\left((1+t)^{-\frac{1}{2}}\left\|\langle\xi\rangle w_{-\ell} u_2\right\|^2
   +(1+t)^{\frac{1}{2}+\vartheta}\left\|w_{-\ell}u_2\right\|^2_{\nu}\right)+C_{\eta}\left\|w_{-\ell}u_2\right\|^2_{\nu}\\
   &&+(1+t)^{-1+\vartheta}\CE^2_{\infty}(t).
 \end{eqnarray*}
Taking $\eta>0$ small enough and integrating above inequality over $[0,t]$ respect to $t$, we can deduce from \eqref{0-decay1} that
\begin{equation}\label{0-order}
 \begin{split}
   (1+t)^{\frac{1}{2}+\vartheta}\left\|w_{-\ell}u_2(t)\right\|^2
   +\int_0^t\left((1+s)^{-\frac{1}{2}}\left\|\langle\xi\rangle w_{-\ell}u_2(s)\right\|^2
   \right.&\left.+(1+s)^{\frac{1}{2}+\vartheta}\left\|w_{-\ell}u_2(s)\right\|^2_{\nu}\right)ds\\
   \lesssim&(1+t)^{\vartheta}\left(\epsilon_0+\CE^2_{\infty}(t)\right),
 \end{split}
\end{equation}
which is the desired weighted energy type estimates on $u_2$ itself.

To derive the weighted energy type estimates on the derivatives of $u_2$ with respect to $x-$variables, as before, we first perform the corresponding energy type estimates without weight. For this purpose, we apply $\partial^{\alpha} (1\leq|\alpha|\leq N-1)$ to the equation \eqref{u2} to get
\begin{eqnarray}\label{alpha-equ.}
  &&\partial_t\partial^{\alpha}u_2+\xi\cdot\nabla_x\partial^{\alpha}u_2
  +\nabla_x\phi\cdot\nabla_{\xi}\partial^{\alpha}u_2+\sum_{|\alpha_1|<|\alpha|}
  C_{\alpha}^{\alpha_1}\partial^{\alpha-\alpha_1}\nabla_x\phi\cdot\nabla_{\xi}\partial^{\alpha_1}u_2\nonumber\\
  &&-\frac{1}{2}\xi\cdot\nabla_x\phi\partial^{\alpha}u_2-\frac{1}{2}\sum_{|\alpha_1|<|\alpha|}
  C_{\alpha}^{\alpha_1}\xi\cdot\partial^{\alpha-\alpha_1}\nabla_x\phi\partial^{\alpha_1}u_2
  -{\bf L}\partial^{\alpha}u_2\\
  &=&\partial^{\alpha}\Gamma(u,u)+\partial^{\alpha}[[{\bf P},\tau_{\phi}]]u,\nonumber
 \end{eqnarray}
then one can get by multiplying \eqref{alpha-equ.} by $\partial^{\alpha}u_2$ and integrating the resulting identity with respect to $x$ and $\xi$ over $\R^3_{x}\times\R^3_{\xi}$ that for $1\leq|\alpha|\leq N-1$
\begin{equation}\label{x-u2}
  \frac{d}{dt}\left\|\partial^{\alpha}u_2(t)\right\|^2+\lambda\left\|\partial^{\alpha}u_2\right\|^2_{\nu}\lesssim (1+t)^{-|\alpha|-\frac{5}{2}}\CE^2_{\infty}(t)+\left\|\partial^{\alpha}\nabla_xu\right\|^2_{\nu}
  +\left\|\nabla_x\phi\right\|_{L^{\infty}}^{\frac{2-\gamma}{1-\gamma}}
  \left\|\langle\xi\rangle\partial^{\alpha}u_2\right\|^2.
\end{equation}
 Here we have used \eqref{coercive}, Lemma \ref{Gamma}, Lemma \ref{comm.-est.}, the definition \eqref{Energy} of $\CE_{\infty}(t)$, and the following two estimates:
\begin{equation*}
 \begin{split}
  &\sum_{\alpha_1<\alpha}C_{\alpha}^{\alpha_1}\left(\partial^{\alpha-\alpha_1}\nabla_x\phi
  \cdot\nabla_{\xi}\partial^{\alpha_1}u_2,\partial^{\alpha}u_2\right)\\
   \lesssim&\eta\left\|\partial^{\alpha}u_2\right\|_{\nu}^2+C_{\eta}\left\|\partial^{\alpha}\nabla_x\phi\right\|^2
  \left\|\langle\xi\rangle^{-\gamma/2}\nabla_{\xi}u_2\right\|_{L_x^{\infty}(L^2_{\xi})}^2\\
  &+C_{\eta}\sum_{1\leq|\alpha_1|\leq|\alpha|-2}\chi_{|\alpha|\geq3}
  \left\|\partial^{\alpha-\alpha_1}\nabla_x\phi\right\|_{L_x^{3}(L^2_{\xi})}^2
  \left\|\langle\xi\rangle^{-\gamma/2}\nabla_{\xi}\partial^{\alpha_1}u_2\right\|_{L_x^{6}(L^2_{\xi})}^2\\
  &+C_{\eta}\sum_{|\alpha_1|=|\alpha|-1}\chi_{|\alpha|\geq2}
  \left\|\partial^{\alpha-\alpha_1}\nabla_x\phi\right\|_{L^{\infty}}^2
  \left\|\langle\xi\rangle^{-\gamma/2}\nabla_{\xi}\partial^{\alpha_1}u_2\right\|^2\\
  \lesssim&\eta\left\|\partial^{\alpha}u_2\right\|_{\nu}^2+C_{\eta}(1+t)^{-|\alpha|}
  \left\|\langle\xi\rangle^{-\gamma/2}\nabla_{\xi}\nabla_xu_2\right\|
  \left\|\langle\xi\rangle^{-\gamma/2}\nabla_{\xi}\nabla^2_xu_2\right\|\\
 &+C_{\eta}\sum_{1\leq|\alpha_1|\leq|\alpha|-2}(1+t)^{-(|\alpha-\alpha_1|+1)}
 \left\|\langle\xi\rangle^{-\gamma/2}\partial^{\alpha_1}\nabla_{\xi}\nabla_x u_2\right\|^2\\
 &+C_{\eta}\sum_{|\alpha_1|=|\alpha|-1}(1+t)^{-(|\alpha-\alpha_1|+2)}
 \left\|\langle\xi\rangle^{-\gamma/2}\nabla_{\xi}\partial^{\alpha_1}u_2\right\|^2\\
  \lesssim&\eta\left\|\partial^{\alpha}u_2\right\|_{\nu}^2+C_{\eta}(1+t)^{-|\alpha|-\frac{5}{2}}\CE^2_{\infty}(t)
\end{split}
\end{equation*}
and
\begin{equation*}
 \begin{split}
  &\frac{1}{2}\sum_{|\alpha_1|<|\alpha|}C_{\alpha}^{\alpha_1}\left(\xi
  \cdot\partial^{\alpha-\alpha_1}\nabla_x\phi\partial^{\alpha_1}u_2,\partial^{\alpha}u_2\right)\\
  \lesssim&\eta\left\|\partial^{\alpha}u_2\right\|_{\nu}^2+C_{\eta}\sum_{|\alpha_1|=|\alpha|-1}
  \left\|\partial^{\alpha-\alpha_1}\nabla_x\phi\right\|_{L^{\infty}}^2
  \left\|\langle\xi\rangle^{1-\gamma/2}\partial^{\alpha_1}u_2\right\|^2\\
  &+C_{\eta}\sum_{1\leq|\alpha_1|\leq|\alpha|-2}
  \left\|\partial^{\alpha-\alpha_1}\nabla_x\phi\right\|_{L_x^{3}(L^2_{\xi})}^2
  \left\|\langle\xi\rangle^{-\gamma/2}\partial^{\alpha_1}u_2\right\|_{L_x^{6}(L^2_{\xi})}^2\\
  &+C_{\eta}\left\|\partial^{\alpha}\nabla_x\phi\right\|^2
  \left\|\langle\xi\rangle^{-\gamma/2}u_2\right\|_{L_x^{\infty}(L^2_{\xi})}^2\\
  \lesssim&\eta\left\|\partial^{\alpha}u_2\right\|_{\nu}^2+(1+t)^{-|\alpha|-\frac{5}{2}}\CE^2_{\infty}(t),
 \end{split}
\end{equation*}
which follows from the definition of $\CE_{\infty}(t)$ and the fact $\langle\xi\rangle^{1-\gamma/2}\leq w_{1-\ell}(t,\xi)$ when $\ell\geq N\geq 4$.

Now we proceed to deduce the weighted energy type estimate on $\partial^\alpha u_2 (1\leq|\alpha|\leq N-1)$. To this end, we first deal with some related terms. In fact, by employing Lemma \ref{in.}, Young's inequality, the definition of $\CE_{\infty}(t)$ in \eqref{Energy}, one can get that by \eqref{good-1} and \eqref{good-3} in Lemma \ref{non-LINEAR} that
\begin{equation*}
 \begin{split}
  &\left|\left(\sum_{|\alpha_1|<|\alpha|}C_{\alpha}^{\alpha_1}\partial^{\alpha-\alpha_1}\nabla_x\phi
  \cdot\nabla_{\xi}\partial^{\alpha_1}u_2,w_{-\ell}^2(t,\xi)\partial^{\alpha}u_2\right)\right|\\
  \lesssim&\eta\left((1+t)^{-1-\vartheta}\left\|\langle\xi\rangle w_{-\ell}\partial^{\alpha}u_2\right\|^2
   +\left\|w_{-\ell}\partial^{\alpha}u_2\right\|^2_{\nu}\right)\\
  &+C_{\eta}\sum_{|\alpha_1|<|\alpha|}(1+t)^{-\left(|\alpha-\alpha_1|+2\right)+(1+\vartheta)\frac{1-\gamma}{2-\gamma}}
   \left\|\langle\xi\rangle w_{1-\ell}\partial^{\alpha_1}\nabla_{\xi}u_2\right\|^2\CE_{\infty}(t)
   \end{split}
 \end{equation*}
and
\begin{equation*}
  \begin{split}
  &\left|\left(\frac{1}{2}\sum_{|\alpha_1|<|\alpha|}C_{\alpha}^{\alpha_1}\xi\cdot
  \partial^{\alpha-\alpha_1}\nabla_x\phi
  \partial^{\alpha_1}u_2,w_{-\ell}^2(t,\xi)\partial^{\alpha}u_2\right)\right|\\
  \lesssim&\eta\left((1+t)^{-1-\vartheta}\left\|\langle\xi\rangle w_{-\ell}\partial^{\alpha}u_2\right\|^2
   +\left\|w_{-\ell}\partial^{\alpha}u_2\right\|_{\nu}^2\right)\\
  &+C_{\eta}\sum_{|\alpha_1|<|\alpha|}(1+t)^{-(|\alpha-\alpha_1|+2)+(1+\vartheta)\frac{1-\gamma}{2-\gamma}}
   \left\|\langle\xi\rangle^{\frac{1}{2}}w_{-\ell}\partial^{\alpha_1}u_2\right\|^2\CE_{\infty}(t).
  \end{split}
\end{equation*}
Basing on the above two estimates, as in the derivation of \eqref{0-w-u2}, if we multiply \eqref{alpha-equ.} by $w_{-\ell}^2\partial^{\alpha}u_2$, then we can get from Lemma \ref{L-K}, Lemma \ref{Gamma}, Lemma \ref{comm.-est.}, and \eqref{basic1} that for $1\leq|\alpha|\leq N-1$
\begin{eqnarray}\label{x-w-u2}
   &&\frac{1}{2}\frac{d}{dt}\left\|w_{-\ell}\partial^{\alpha}u_2(t)\right\|^2+2q\vartheta(1+t)^{-1-\vartheta}
   \left\|\langle\xi\rangle w_{-\ell}\partial^{\alpha}u_2\right\|^2
   +\frac{1}{2}\left\|w_{-\ell}\partial^{\alpha}u_2\right\|^2_{\nu}\nonumber\\
   &\lesssim&\left\|\partial^{\alpha}u_2\right\|_{\nu}^2+\left\|\partial^{\alpha}\nabla_x u_2\right\|_{\nu}^2
   +(1+t)^{-\frac{3}{2}-|\alpha|}\CE^2_{\infty}(t)
   +\left\|\nabla_x\phi\right\|_{L^{\infty}}^{\frac{2-\gamma}{1-\gamma}}
   \left\|\langle\xi\rangle w_{-\ell}\partial^{\alpha}u_2\right\|^2\\
   &&+\sum_{|\alpha_1|<|\alpha|}(1+t)^{-(|\alpha-\alpha_1|+2)+(1+\vartheta)\frac{1-\gamma}{2-\gamma}}
   \left\|\langle\xi\rangle w_{1-\ell}\partial^{\alpha_1}\nabla_{\xi}u_2\right\|^2\CE_{\infty}(t).\nonumber
 \end{eqnarray}

Choosing $M_2>0$ suitably large, we take $M\times\eqref{x-u2}+\eqref{x-w-u2}$ to yield that for certain constant $\lambda>0$,
\begin{eqnarray}\label{x-u2-cover}
   &&\frac{d}{dt}\left(M\left\|\partial^{\alpha}u_2(t)\right\|^2
   +\left\|w_{-\ell}\partial^{\alpha}u_2(t)\right\|^2\right)
   +\lambda\left((1+t)^{-1-\vartheta}\left\|\langle\xi\rangle w_{-\ell}\partial^{\alpha}u_2\right\|^2
   +\left\|w_{-\ell}\partial^{\alpha}u_2\right\|^2_{\nu}\right)\nonumber\\
   &\lesssim&\sum_{|\alpha_1|<|\alpha|}(1+t)^{-(|\alpha-\alpha_1|+2)+(1+\vartheta)\frac{1-\gamma}{2-\gamma}}
   \left(\left\|\langle\xi\rangle w_{-\ell}\partial^{\alpha_1}\nabla_{\xi}u_2\right\|^2
   +\left\|\langle\xi\rangle^{\frac{1}{2}}w_{-\ell}\partial^{\alpha_1}u_2\right\|^2\right)\CE_{\infty}(t)\\
   &&+(1+t)^{-|\alpha|-\frac{5}{2}+\frac{2}{3}}\CE^2_{\infty}(t)
   +\left\|\partial^{\alpha}\nabla_xu\right\|^2_{\nu}, \quad 1\leq|\alpha|\leq N-1.\nonumber
\end{eqnarray}
\begin{remark}\label{u2-(N-1)}
We want to emphasize that the above estimate \eqref{x-u2-cover} holds for $|\alpha|=N-1$, such a fact plays an important role in our analysis.
\end{remark}

For the mixed derivatives of $u_2$ with respect to the spatial variable $x$ and the velocity variable $\xi$, we can get by applying $\partial^{\alpha}_{\beta}$ with $|\alpha|+|\beta|\leq N,~|\beta|\geq1$ to \eqref{u2} that
\begin{eqnarray}\label{mixed-equ.}
 &&\partial_t\partial^{\alpha}_{\beta}u_2+\xi\cdot\nabla_x\partial^{\alpha}_{\beta}u_2
  +\nabla_x\phi\cdot\nabla_{\xi}\partial^{\alpha}_{\beta}u_2
  -\frac{1}{2}\xi\cdot\nabla_x\phi\partial^{\alpha}_{\beta}u_2\nonumber\\
  &&+\sum_{|\beta_1|=1}C_{\beta}^{\beta_1}
  \partial_{\beta_1}\xi\cdot\nabla_{x}\partial^{\alpha}_{\beta-\beta_1}u_2
  +\sum_{|\alpha_1|<|\alpha|}C_{\alpha}^{\alpha_1}\partial^{\alpha-\alpha_1}\nabla_x\phi\cdot
  \nabla_{\xi}\partial^{\alpha_1}_{\beta}u_2\\
  &&-\frac{1}{2}\sum_{|\beta_1|=1}C_{\beta}^{\beta_1}
  \partial_{\beta_1}\xi\cdot\partial^{\alpha}\left(\nabla_x\phi\partial_{\beta-\beta_1}u_2\right)
  -\frac{1}{2}\sum_{|\alpha_1|<|\alpha|}C_{\alpha}^{\alpha_1}\xi\cdot\partial^{\alpha-\alpha_1}
  \nabla_x\phi\partial^{\alpha_1}_{\beta}u_2\nonumber\\
  &=&\partial_{\beta}{\bf L}\partial^{\alpha}u_2
  +\partial^{\alpha}_{\beta}\Gamma(u,u)+\partial^{\alpha}_{\beta}[[{\bf P},\tau_{\phi}]]u.\nonumber
\end{eqnarray}
Multiplying \eqref{mixed-equ.} by $w_{|\beta|-\ell}^2\partial^{\alpha}_{\beta}u_2$ and integrating the resulting identity with respect to $x$ and $\xi$ over $\R^3_{x}\times\R^3_{\xi}$, we can deduce from
\begin{equation*}
  \left(\partial^\alpha_\beta u_2,w^2_{|\beta|-\ell}\nabla_x\partial^\alpha_\beta u_2\right)=0
\end{equation*}
and
\begin{equation*}
\left(\partial^\alpha_\beta u_2, w^2_{|\beta|-\ell}\partial_t\partial^\alpha_\beta u_2\right)=\frac{1}{2} \frac{d}{dt}\left\|w_{|\beta|-\ell}\partial^{\alpha}_{\beta}u_2(t)\right\|^2+\frac{2q\vartheta}{(1+t)^{1+\vartheta}}
\left\|\langle\xi\rangle w_{|\beta|-\ell}\partial^{\alpha}_{\beta}u_2\right\|^2
\end{equation*}
that
\begin{equation}\label{mixed-w-u2-draft}
 \begin{split}
   &\frac{1}{2}\frac{d}{dt}\left\|w_{|\beta|-\ell}\partial^{\alpha}_{\beta}u_2(t)\right\|^2
   +\frac{2q\vartheta}{(1+t)^{1+\vartheta}}
   \left\|\langle\xi\rangle w_{|\beta|-\ell}\partial^{\alpha}_{\beta}u_2\right\|^2
   =\underbrace{-\left(\partial^\alpha_\beta u_2,w^2_{|\beta|-\ell} \nabla_x\phi\cdot\nabla_{\xi}\partial^{\alpha}_{\beta}u_2\right)}_{J_1}\\
   &+\underbrace{\frac 12\left(\partial^\alpha_\beta u_2,w^2_{|\beta|-\ell}\xi\cdot\nabla_x\phi\partial^{\alpha}_{\beta}u_2\right)}_{J_2}
   -\underbrace{\sum_{|\beta_1|=1}C_{\beta}^{\beta_1}\left(\partial^\alpha_\beta u_2,w^2_{|\beta|-\ell} \partial_{\beta_1}\xi\cdot\nabla_{x}\partial^{\alpha}_{\beta-\beta_1}u_2\right)}_{J_3}\\
   &-\underbrace{\sum_{|\alpha_1|<|\alpha|}C_{\alpha}^{\alpha_1}\chi_{|\alpha|\geq1}\left(\partial^\alpha_\beta u_2,w^2_{|\beta|-\ell} \partial^{\alpha-\alpha_1}\nabla_x\phi\cdot
  \nabla_{\xi}\partial^{\alpha_1}_{\beta}u_2\right)}_{J_4}\\
  &+\underbrace{\frac 12\sum_{|\beta_1|=1}C_{\beta}^{\beta_1}\left(\partial^\alpha_\beta u_2,w^2_{|\beta|-\ell} \partial_{\beta_1}\xi\cdot\partial^{\alpha}\left(\nabla_x\phi\partial_{\beta-\beta_1}u_2\right)\right)}_{J_5}\\
  &+\underbrace{\frac 12\sum_{|\alpha_1|<|\alpha|}\chi_{|\alpha|\geq1}C_{\alpha}^{\alpha_1}\left(\partial^\alpha_\beta u_2,w^2_{|\beta|-\ell} \xi\cdot\partial^{\alpha-\alpha_1}
  \nabla_x\phi\partial^{\alpha_1}_{\beta}u_2\right)}_{J_6}\\
  &+\underbrace{\left(\partial^\alpha_\beta u_2,w^2_{|\beta|-\ell} \partial_{\beta}{\bf L}\partial^{\alpha}u_2\right)}_{J_7}
  +\underbrace{\left(\partial^\alpha_\beta u_2,w^2_{|\beta|-\ell} \partial^{\alpha}_{\beta}\Gamma(u,u)\right)}_{J_8}
  +\underbrace{\left(\partial^\alpha_\beta u_2,w^2_{|\beta|-\ell} \partial^{\alpha}_{\beta}
  [[{\bf P},\tau_{\phi}]]u\right)}_{J_9}.
 \end{split}
\end{equation}
Now we estimate $J_i$ $(i=1,2,\cdots,9)$ term by term. For $J_1$ and $J_2$, noticing
\begin{equation*}
  \left|\nabla_\xi w^2_{|\beta|-\ell}(t,\xi)\right|\lesssim \langle\xi\rangle w^2_{|\beta|-\ell}(t,\xi),\quad
 \langle\xi\rangle=\left(\langle\xi\rangle^2\right)^{\frac{1-\gamma}{2-\gamma}} \nu(\xi)^{\frac{1}{2-\gamma}},
\end{equation*}
we use Young's inequality to deduce that
\begin{eqnarray*}
|J_1|+|J_2|&\lesssim&\left|\left(\nabla_x\phi, \langle\xi\rangle w^2_{|\beta|-\ell}\left|\partial^\alpha_\beta u_2\right|^2\right)\right|\\
&\lesssim&\eta \left\|w_{|\beta|-\ell}\partial^\alpha_\beta u_2\right\|^2_\nu+C_\eta\left\|\nabla_x\phi(t)\right\|^{\frac{2-\gamma}{1-\gamma}}_{L^\infty(\R^3_x)}
 \left\|\langle\xi\rangle w_{|\beta|-\ell}\partial^\alpha_\beta u_2\right\|^2\\
&\lesssim&\eta \left\|w_{|\beta|-\ell}\partial^\alpha_\beta u_2\right\|^2_\nu+(1+t)^{-\frac{5}{4}}\CE_{\infty}^{\frac{2-\gamma}{2(1-\gamma)}}(t)\left\|\langle\xi\rangle w_{|\beta|-\ell}\partial^\alpha_\beta u_2\right\|^2.
\end{eqnarray*}
Here and below $\eta>0$ can be any sufficiently small positive constant. It is worth to pointing out that the first term in the right hand side of the estimate on $|J_1|+|J_2|$ stated above can be absorbed by the dissipation term induced by the coercive estimate of the linearized collision operator ${\bf L}$, cf. the first term in the right hand side of the estimate on $J_7$, while the corresponding second term can be absorbed by the second term in the left hand side of \eqref{mixed-w-u2-draft} since $\vartheta$ is a positive constant chosen to satisfy $0<\vartheta\leq \frac 14$ and the temporal energy functional $\mathcal{E}_\infty(t)$ is chosen suitably small.

For $J_3$, noticing $w^2_{|\beta|-\ell}=w_{|\beta-\beta_1|-\ell}w_{|\beta|-\ell}\langle\xi\rangle^{\frac\gamma 2}$ for $|\beta_1|=1$, we can get from Cauchy's inequality that
\begin{eqnarray*}
  J_3&\lesssim&\sum_{|\beta_1|=1}\left|\left(\partial_{\beta_1}\xi\cdot\nabla_{x}\partial^{\alpha}_{\beta-\beta_1}u_2,
  w_{|\beta|-\ell}^2\partial^{\alpha}_{\beta}u_2\right)\right|\\
 & =&\sum_{|\beta_1|=1}\left|\left(w_{|\beta-\beta_1|-\ell}\partial_{\beta_1}\xi\cdot
  \nabla_{x}\partial^{\alpha}_{\beta-\beta_1}u_2\langle\xi\rangle^{\frac{\gamma}{2}},
  w_{|\beta|-\ell}\partial^{\alpha}_{\beta}u_2\right)\right|\\
  &\lesssim&\eta\left\|w_{|\beta|-\ell}\partial^{\alpha}_{\beta}u_2\right\|^2_{\nu}
  +C_{\eta}\sum_{|\beta_1|=1}\left\|w_{|\beta-\beta_1|-\ell}\partial^{\alpha+e_i}_{\beta-\beta_1}u_2\right\|^2,
 \end{eqnarray*}

Now we turn to deal with $J_5$ which can be controlled as follows
\begin{eqnarray*}
J_5&\lesssim&\sum_{|\alpha_1|\leq|\alpha|}\left|\left(\partial^{\alpha-\alpha_1}\nabla_x\phi
w_{|\beta-e_i|-\ell}\langle\xi\rangle^{\frac{\gamma}{2}}\partial^{\alpha_1}_{\beta-e_i}u_2,
w_{|\beta|-\ell}\partial^{\alpha}_{\beta}u_2\right)\right|\\
&\lesssim&\eta\left\|w_{|\beta|-\ell}\partial^{\alpha}_{\beta}u_2\right\|_{\nu}^2
+\sum_{|\alpha_1|+|\beta|\leq\frac{N}{2}}\left\|\partial^{\alpha-\alpha_1}\nabla_x\phi\right\|^2
\left\|w_{|\beta-e_i|-\ell}\partial^{\alpha_1}_{\beta-e_i}u_2\right\|_{L^{\infty}_x(L^2_{\xi})}^2\\
&&+\sum_{|\alpha_1|+|\beta|>\frac{N}{2}}
\left\|\partial^{\alpha-\alpha_1}\nabla_x\phi\right\|_{L^{\infty}}^2
\left\|w_{|\beta-e_i|-\ell}\partial^{\alpha_1}_{\beta-e_i}u_2\right\|^2\\
&\lesssim&\eta\left\|w_{|\beta|-\ell}\partial^{\alpha}_{\beta}u_2\right\|_{\nu}^2
+(1+t)^{-|\alpha|-\frac{5}{2}}\CE^2_{\infty}(t).
 \end{eqnarray*}

Noticing that $J_4$ and $J_6$ have been estimated in Lemma \ref{non-LINEAR} which tell us that they can be controlled by the right hand sides of \eqref{good-4} and \eqref{good-6}, we only need to control the terms $J_7, J_8$, and $J_9$ suitably. For this purpose, we can get by exploiting Lemma \ref{L}, Lemma \ref{Gamma}, and Lemma \ref{comm.-est.} that
\begin{equation*}
 J_7 \lesssim -\left\|w_{|\beta|-\ell}\partial^{\alpha}_{\beta}u_2\right\|_{\nu}^2
 +\eta\sum_{|\beta_1|\leq|\beta|}\left\|w_{|\beta_1|-\ell}\partial^{\alpha}_{\beta_1}u_2\right\|_{\nu}^2
 +C_{\eta}\left\|\partial^{\alpha}u_2\right\|_{\nu}^2,
\end{equation*}
\begin{equation*}
 J_8\lesssim \eta\left\|w_{|\beta|-\ell}\partial^{\alpha}_{\beta}u_2\right\|_{\nu}^2
 +C_{\eta}(1+t)^{-|\alpha|-\frac{3}{2}}\CE_{\infty}^2(t),
\end{equation*}
and
\begin{equation*}
 J_9\lesssim\eta\left\|w_{|\beta|-\ell}\partial^{\alpha}_{\beta}u_2\right\|_{\nu}^2
 +C_{\eta}(1+t)^{-|\alpha|-\frac{3}{2}}\CE_{\infty}^2(t)
 +\chi_{|\alpha|=N-1}\left\|\nabla_x^Nu\right\|^2_{\nu},
  \end{equation*}
respectively.

Substituting the above estimates on $J_i$ $(i=1,2,\cdots,9)$ into
\eqref{mixed-w-u2-draft} and choosing $\eta$ suitably small, we
finally arrive at
\begin{eqnarray}\label{mixed-w-u2}
   &&\frac{d}{dt}\left\|w_{|\beta|-\ell}\partial^{\alpha}_{\beta}u_2(t)\right\|^2+q\vartheta(1+t)^{-1-\vartheta}
   \left\|\langle\xi\rangle w_{|\beta|-\ell}\partial^{\alpha}_{\beta}u_2\right\|^2
   +\frac{1}{2}\left\|w_{|\beta|-\ell}\partial^{\alpha}_{\beta}u_2\right\|^2_{\nu}\nonumber\\
   &\lesssim&\sum_{|\beta_1|<|\beta|}\left\|w_{|\beta_1|-\ell}\partial^{\alpha}_{\beta_1}u_2\right\|^2_{\nu}
   +\sum_{|\beta_1|=1}\left\|w_{|\beta-\beta_1|-\ell}\partial^{\alpha+e_i}_{\beta-\beta_1}u_2\right\|^2\\
   &&+\sum_{|\alpha_1|<|\alpha|}\chi_{|\alpha|\geq1}(1+t)^{-(|\alpha-\alpha_1|+2)
   +(1+\vartheta)\frac{1-\gamma}{2-\gamma}}\left(\left\|\langle\xi\rangle w_{|\beta+e_i|-\ell}
   \nabla_{\xi}\partial^{\alpha_1}_{\beta}u_2\right\|^2\right.\nonumber\\
   &&\left.+\left\|\langle\xi\rangle^{\frac{1}{2}}w_{|\beta|-\ell}\partial^{\alpha_1}_{\beta}u_2\right\|^2\right)
   \CE_{\infty}(t)
   +(1+t)^{-|\alpha|-\frac{3}{2}}\left(\epsilon_0+\CE^2_{\infty}(t)\right)
   +\chi_{|\alpha|=N-1}\left\|\nabla_x^Nu\right\|_{\nu}^2\nonumber
 \end{eqnarray}
provided that the temporal energy functional $\mathcal{E}_\infty(t)$ is chosen sufficiently small. Here $|\alpha|+|\beta|\leq N$, $|\beta|\geq 1$.

With the above estimates in hand, we now turn to deduce the main result on the lower order energy type estimates on the microscopic component $u_2(t,x,\xi)$ of the solution $u(t,x,\xi)$ to the Cauchy problem of the one-species VPB system \eqref{u}.
\begin{lemma}\label{lower_order_microscopic} Suppose that $u(t,x,\xi)$ is a solution to the Cauchy problem of the one-species VPB system \eqref{u} which is defined on $\prod_T=[0,T]\times\R_x^3\times\R_\xi^3$ and satisfies the a priori assumption \eqref{a priori assumption} for some suitably chosen small positive constant $\delta>0$, then it holds for $|\alpha|+|\beta|\leq N-2$ that
\begin{eqnarray}\label{induc.-Assu.2}
     &&(1+t)^{|\alpha|+\frac{1}{2}+\vartheta}
     \left\|w_{|\beta|-\ell}\partial^{\alpha}_{\beta}u_2\right\|^2\nonumber\\
     &&+\int_0^{t}\left((1+s)^{|\alpha|-\frac{1}{2}}
     \left\|\langle\xi\rangle w_{|\beta|-\ell}\partial^{\alpha}_{\beta}u_2(s)\right\|^2
     +(1+s)^{|\alpha|+\frac{1}{2}+\vartheta}
     \left\|w_{|\beta|-\ell}\partial^{\alpha}_{\beta}u_2(s)\right\|_{\nu}^2\right)ds\\
     &\lesssim&(1+t)^{\vartheta}\left(\epsilon_0+\CE^2_{\infty}(t)\right).\nonumber
  \end{eqnarray}
\end{lemma}
\begin{proof} \eqref{induc.-Assu.2} will be proved by the principle of mathematical induction. To this end, we first notice from \eqref{0-order} that \eqref{induc.-Assu.2} holds for $|\alpha|+|\beta|=0$. Now assume that \eqref{induc.-Assu.2} holds for any $|\alpha|+|\beta|=k\leq N-3$,
  to prove \eqref{induc.-Assu.2}, we only need to verify that \eqref{induc.-Assu.2} also holds for $|\alpha|+|\beta|=k+1\leq N-2$.
  For this purpose, letting $|\alpha|=k+1$ in \eqref{x-u2-cover} and integrating the result with respect to $t$ over $[0,t]$, we obtain
\begin{eqnarray}\label{(k+1)-x}
   &&\left\|w_{-\ell}\nabla_x^{k+1}u_2(t)\right\|^2+\int_0^t\left((1+s)^{-1-\vartheta}\left\|\langle\xi\rangle w_{-\ell}\nabla_x^{k+1}u_2\right\|^2+\left\|w_{-\ell}\nabla_x^{k+1}u_2\right\|^2_{\nu}\right)ds\nonumber\\
   &\lesssim&\CE_{\infty}(t)\sum_{|\alpha_1|<k+1}
   \int_0^t(1+s)^{-(|\alpha-\alpha_1|+2)+(1+\vartheta)\frac{1-\gamma}{2-\gamma}}
   \left(\left\|\langle\xi\rangle w_{-\ell}\partial^{\alpha_1}\nabla_{\xi}u_2\right\|^2
   +\left\|\langle\xi\rangle^{\frac{1}{2}}w_{-\ell}\partial^{\alpha_1}u_2\right\|^2\right)ds\\
   &&+\epsilon_0+\CE^2_{\infty}(t).\nonumber
 \end{eqnarray}
Here we have used the estimate
\begin{equation*}
  \left\|\partial^{\alpha+e_i}u\right\|^2_{\nu}\lesssim(1+t)^{-|\alpha+e_i|-\frac{1}{2}}
  \left(\epsilon_0+\CE^2_{\infty}(t)\right),
\end{equation*}
which follows from Lemma \ref{non-weight} since $|\alpha+e_i|\leq
N-1$.

On the other hand, we multiply \eqref{x-u2-cover} (with
$|\alpha|=k+1$) by $(1+t)^{k+\frac{3}{2}+\vartheta}$ and integrate the resulting inequality with respect to $t$ over $[0,t]$ to yield
\begin{eqnarray}\label{(k+1)-decay1}
   &&(1+t)^{k+\frac{3}{2}+\vartheta}\left\|w_{-\ell}\nabla_x^{k+1}u_2(t)\right\|^2\nonumber\\
   &&+\int_0^t\left((1+s)^{k+\frac{1}{2}}\left\|\langle\xi\rangle w_{-\ell}\nabla_x^{k+1}u_2\right\|^2
   +(1+s)^{k+\vartheta+\frac{3}{2}}\left\|w_{-\ell}\nabla_x^{k+1}u_2\right\|^2_{\nu}\right)ds\\
   &\lesssim&\int_0^t(1+s)^{k+\frac{1}{2}+\vartheta}\left\|w_{-\ell}\nabla_x^{k+1}u_2(s)\right\|^2ds
   +(1+t)^{\vartheta}\left(\epsilon_0+\CE^2_{\infty}(t)\right)\nonumber\\
   &&+\CE_{\infty}(t)\sum_{|\alpha_1|<|\alpha|=k+1}
   \int_0^t(1+s)^{-(|\alpha-\alpha_1|+2)+(1+\vartheta)\frac{1-\gamma}{2-\gamma}
   +|\alpha|+\frac{1}{2}+\vartheta}\left\|\langle\xi\rangle w_{-\ell}\partial^{\alpha_1}\nabla_{\xi}u_2(s)\right\|^2ds.\nonumber
 \end{eqnarray}
By Lemma \ref{in.} with $m=k+\frac 12$, the first term in the right hand side of \eqref{(k+1)-decay1} can be further bounded by
\begin{eqnarray}\label{first-term}
   &&\int_0^t(1+s)^{k+\frac{1}{2}+\vartheta}\left\|w_{-\ell}\nabla_x^{k+1}u_2(s)\right\|^2ds\nonumber\\
   &\lesssim&\eta\int_0^t\left((1+s)^{k+\frac{1}{2}}\left\|\langle\xi\rangle w_{-\ell}\nabla_x^{k+1}u_2\right\|^2
   +(1+s)^{k+\vartheta+\frac{3}{2}}\left\|w_{-\ell}\nabla_x^{k+1}u_2\right\|^2_{\nu}\right)ds\\
   &&+C_{\eta}\int_0^t\left\|w_{-\ell}\nabla_x^{k+1}u_2(t)\right\|^2_{\nu}ds
   +(1+t)^{\vartheta}\left(\epsilon_0+\CE^2_{\infty}(t)\right).\nonumber
 \end{eqnarray}
Then multiplying \eqref{(k+1)-x} by a suitable large constant and adding the
corresponding inequality and \eqref{first-term} to \eqref{(k+1)-decay1} yield
\begin{eqnarray}\label{(k+1)-decay2}
   &&(1+t)^{(k+1)+\frac{1}{2}+\vartheta}\left\|w_{-\ell}\nabla_x^{k+1}u_2(t)\right\|^2\nonumber\\
   &&+\int_0^t\left((1+s)^{(k+1)-\frac{1}{2}}\left\|\langle\xi\rangle w_{-\ell}\nabla_x^{k+1}u_2\right\|^2
   +(1+s)^{(k+1)+\vartheta+\frac{1}{2}}\left\|w_{-\ell}\nabla_x^{k+1}u_2\right\|^2_{\nu}\right)ds\\
   &\lesssim&\CE_{\infty}(t)\int_0^t(1+s)^{-3+(1+\vartheta)\frac{1-\gamma}{2-\gamma}
   +(k+1)+\frac{1}{2}+\vartheta}\left\|\langle\xi\rangle w_{-\ell}\nabla_x^k\nabla_{\xi}u_2(s)\right\|^2ds\nonumber\\
   &&+(1+t)^{\vartheta}\left(\epsilon_0+\CE^2_{\infty}(t)\right),\nonumber
 \end{eqnarray}
where we have used the fact that \eqref{induc.-Assu.2} is assumed to
hold for all $|\alpha|+|\beta|\leq k$ and  that $$
-(|\alpha-\alpha_1|+2)+(1+\vartheta)\frac{1-\gamma}{2-\gamma}+|\alpha|+\frac{1}{2}
+\vartheta\leq|\alpha_1|-\frac{1}{2}
$$
holds for $|\alpha_1|\leq|\alpha|$, $0<\theta\leq\frac{1}{9}$ and
$\gamma\in(-3,0)$.

To control the first term in the right hand side of
\eqref{(k+1)-decay2}, we need to deal with the terms involving the
mixed spatial and velocity derivatives of $u_2(t,x,\xi)$. Letting
$|\alpha|+|\beta|=k+1$ in (\ref{mixed-w-u2}), we integrate it with
respect to $t$ over $[0,t]$ to get
\begin{eqnarray}\label{mixed-1}
   &&\left\|w_{|\beta|-\ell}\partial^{\alpha}_{\beta}u_2(t)\right\|^2+\int_0^t\left((1+s)^{-1-\vartheta}
   \left\|\langle\xi\rangle w_{|\beta|-\ell}\partial^{\alpha}_{\beta}u_2(s)\right\|^2
   +\left\|w_{|\beta|-\ell}\partial^{\alpha}_{\beta}u_2(s)\right\|^2_{\nu}\right)ds\nonumber\\
   &\lesssim&\int_0^t\left(\sum_{|\beta_1|<|\beta|}
   \left\|w_{|\beta_1|-\ell}\partial^{\alpha}_{\beta_1}u_2(s)\right\|^2_{\nu}
   +\sum_{|\beta_1|=1}\left\|w_{|\beta-\beta_1|-\ell}\partial^{\alpha+e_i}_{\beta-\beta_1}u_2(s)\right\|^2\right)ds\\
   &&+\CE_{\infty}(t)\sum_{|\alpha_1|<|\alpha|}\chi_{|\alpha|\geq1}\int_0^t(1+s)^{-(|\alpha-\alpha_1|+2)
   +(1+\vartheta)\frac{1-\gamma}{2-\gamma}}\left(\left\|\langle\xi\rangle w_{|\beta+e_i|-\ell}
   \nabla_{\xi}\partial^{\alpha_1}_{\beta}u_2(s)\right\|^2\right.\nonumber\\
   &&\left.+\left\|\langle\xi\rangle^{\frac{1}{2}}
   w_{|\beta+e_i|-\ell}\partial^{\alpha_1}_{\beta}u_2(s)\right\|^2\right)ds
   +\left(\epsilon_0+\CE^2_{\infty}(t)\right).\nonumber
  \end{eqnarray}
Moreover, we multiply \eqref{mixed-w-u2} with $|\alpha|+|\beta|=k+1$,
$|\beta|\geq1$ by $(1+t)^{|\alpha|+\frac{1}{2}+\theta}$ and integrating the resulting differential inequality with respect to $t$ over $[0,t]$ to obtain
\begin{eqnarray}\label{a}
   &&(1+t)^{|\alpha|+\frac{1}{2}+\vartheta}\left\|w_{|\beta|-\ell}\partial^{\alpha}_{\beta}u_2(t)\right\|^2\nonumber\\
   &&+\int_0^t\left((1+s)^{|\alpha|-\frac{1}{2}}\left\|\langle\xi\rangle w_{|\beta|-\ell}\partial^{\alpha}_{\beta}u_2(s)\right\|^2
   +(1+s)^{|\alpha|+\frac{1}{2}+\vartheta}
   \left\|w_{|\beta|-\ell}\partial^{\alpha}_{\beta}u_2(s)\right\|^2_{\nu}\right)ds\nonumber\\
   &\lesssim&\int_0^t(1+s)^{|\alpha|+\frac{1}{2}+\vartheta}\left(\sum_{|\beta_1|<|\beta|}\left\|w_{|\beta_1|-\ell}
   \partial^{\alpha}_{\beta_1}u_2(s)\right\|^2_{\nu}+\sum_{|\beta_1|=1}\left\|w_{|\beta-\beta_1|-\ell}
   \partial^{\alpha+e_i}_{\beta-\beta_1}u_2(s)\right\|^2\right)ds\\
   &&+\CE_{\infty}(t)\sum_{|\alpha_1|<|\alpha|}\int_0^t(1+s)^{-(|\alpha-\alpha_1|+2)
   +(1+\vartheta)\frac{1-\gamma}{2-\gamma}+|\alpha|+\frac{1}{2}+\vartheta}
   \left(\left\|\langle\xi\rangle w_{|\beta+e_i|-\ell}\nabla_{\xi}\partial^{\alpha_1}_{\beta}u_2(s)\right\|^2\right.\nonumber\\
   &&\left.+\left\|\langle\xi\rangle^{\frac{1}{2}}w_{|\beta|-\ell}\partial^{\alpha_1}_{\beta}u_2\right\|^2\right)ds
   +(1+t)^{\vartheta}\left(\epsilon_0+\CE^2_{\infty}(t)\right)
   +\int_0^t(1+s)^{|\alpha|-\frac{1}{2}+\vartheta}
   \left\|w_{|\beta|-\ell}\partial^{\alpha}_{\beta}u_2\right\|^2ds.\nonumber
  \end{eqnarray}
For the last term in the right hand side of the inequality above, by
Lemma \ref{in.} (with $m=|\alpha|-\frac{1}{2}$), we can bound it
by
\begin{eqnarray}\label{first-term-again}
&&\int_0^t(1+s)^{|\alpha|-\frac{1}{2}+\vartheta}
   \left\|w_{|\beta|-\ell}\partial^{\alpha}_{\beta}u_2\right\|^2ds\nonumber\\
&\lesssim&\eta\int_0^t\left((1+s)^{|\alpha|-\frac{1}{2}}\left\|\langle\xi\rangle w_{|\beta|-\ell}\partial^{\alpha}_{\beta}u_2\right\|^2
   +(1+s)^{|\alpha|+\frac{1}{2}+\vartheta}\left\|w_{|\beta|-\ell}\partial^{\alpha}_{\beta}u_2\right\|^2_{\nu}\right)ds\\
&&+C_{\eta}\int_0^t\left\|w_{|\beta|-\ell}\partial^{\alpha}_{\beta}u_2(s)\right\|_{\nu}^2ds.\nonumber
\end{eqnarray}
Here  $\eta>0$ can be any sufficiently small positive constant.

A suitably linear combination of \eqref{mixed-1}, \eqref{a}, and \eqref{first-term-again} yields
\begin{eqnarray}\label{b}
   &&(1+t)^{|\alpha|+\frac{1}{2}+\vartheta}\left\|w_{|\beta|-\ell}\partial^{\alpha}_{\beta}u_2(t)\right\|^2\nonumber\\
   &&+\int_0^t\left((1+s)^{|\alpha|-\frac{1}{2}}\left\|\langle\xi\rangle w_{|\beta|-\ell}\partial^{\alpha}_{\beta}u_2(s)\right\|^2
   +(1+s)^{|\alpha|+\frac{1}{2}+\vartheta}
   \left\|w_{|\beta|-\ell}\partial^{\alpha}_{\beta}u_2(s)\right\|^2_{\nu}\right)ds\nonumber\\
   &\lesssim&\int_0^t(1+s)^{|\alpha|+\frac{1}{2}+\vartheta}\left(\sum_{|\beta_1|<|\beta|}\left\|w_{|\beta_1|-\ell}
   \partial^{\alpha}_{\beta_1}u_2(s)\right\|^2_{\nu}+\sum_{|\beta_1|=1}\left\|w_{|\beta-\beta_1|-\ell}
   \partial^{\alpha+e_i}_{\beta-\beta_1}u_2(s)\right\|^2\right)ds\nonumber\\
   &&+\CE_{\infty}(t)\sum_{|\alpha_1|<|\alpha|}\int_0^t(1+s)^{-(|\alpha-\alpha_1|+2)
   +(1+\vartheta)\frac{1-\gamma}{2-\gamma}+|\alpha|+\frac{1}{2}+\vartheta}
   \left(\left\|\langle\xi\rangle w_{|\beta+e_i|-\ell}\nabla_{\xi}\partial^{\alpha_1}_{\beta}u_2(s)\right\|^2\right.\\
   &&\left.+\left\|\langle\xi\rangle^{\frac{1}{2}}w_{|\beta|-\ell}\partial^{\alpha_1}_{\beta}u_2\right\|^2\right)ds
   +(1+t)^{\vartheta}\left(\epsilon_0+\CE^2_{\infty}(t)\right)\nonumber\\
   &\lesssim&\int_0^t\left((1+s)^{|\alpha+e_i|-\frac{1}{2}}\left\|\langle\xi\rangle w_{|\beta-e_i|-\ell}
   \partial^{\alpha+e_i}_{\beta-e_i}u_2\right\|^2+(1+s)^{|\alpha+e_i|+\frac{1}{2}+\vartheta}
   \left\|w_{|\beta-e_i|-\ell}\partial^{\alpha+e_i}_{\beta-e_i}u_2\right\|_{\nu}^2\right)ds\nonumber\\
   &&+\CE_{\infty}(t)\sum_{|\alpha_1|=|\alpha|-1}\int_0^t(1+s)^{-|\alpha_1|-\frac{1}{2}}
   \left\|\langle\xi\rangle w_{|\beta+e_i|-\ell}\nabla_{\xi}\partial^{\alpha_1}_{\beta}u_2(s)\right\|^2ds
   +(1+t)^{\vartheta}\left(\epsilon_0+\CE^2_{\infty}(t)\right).\nonumber
  \end{eqnarray}
Here we have used  Lemma \ref{in.} with $m=|\alpha|+\frac{1}{2}$, and the induction assumption that \eqref{induc.-Assu.2} holds for $|\alpha|+|\beta|\leq k$. It is worth to pointing out that the last term vanishes when $|\alpha|=0$.

A further suitable linear combination of the estimate \eqref{b} for the cases
$|\alpha|=0,~|\alpha|=1,...,|\alpha|=k$ together with the estimate \eqref{(k+1)-decay2} for $|\alpha|=k+1$, we can deduce for  $|\alpha|+|\beta|=k+1$ that
\begin{eqnarray}\label{part<N-1}
   &&\sum_{|\alpha|+|\beta|=k+1}\sum_{|\alpha|=0}^{k+1}\widetilde{M}_{|\alpha|}
   \left\{(1+t)^{|\alpha|+\frac{1}{2}+\vartheta}
   \left\|w_{|\beta|-\ell}\partial^{\alpha}_{\beta}u_2(t)\right\|^2\right.\nonumber\\
   &&\left.+\int_0^t\left((1+s)^{|\alpha|-\frac{1}{2}}\left\|\langle\xi\rangle w_{|\beta|-\ell}\partial^{\alpha}_{\beta}u_2(s)\right\|^2
   +(1+s)^{|\alpha|+\frac{1}{2}+\vartheta}
   \left\|w_{|\beta|-\ell}\partial^{\alpha}_{\beta}u_2(s)\right\|^2_{\nu}\right)ds\right\}\\
   &\lesssim&(1+t)^{\vartheta}\left(\epsilon_0+\CE^2_{\infty}(t)\right),\nonumber
  \end{eqnarray}
where we have used the smallness of $\CE_{\infty}(t)$.

\eqref{part<N-1} implies that \eqref{induc.-Assu.2} holds for $|\alpha|+|\beta|=k+1\leq N-2$ and the proof of Lemma \ref{lower_order_microscopic} is complete by the principle of mathematical induction.
\end{proof}

\section{Higher Order Energy Estimates}
This section is devoted to deducing the desired higher order energy type estimates, i.e. the weighted estimates on $\partial^\alpha_\beta u$ for $(N-1\leq |\alpha|+|\beta|\leq N)$,  in terms of the temporal energy functional $\mathcal{E}_\infty(t)$. As mentioned before, when we deal with the higher order energy type estimates on the solution $u(t,x,\xi)$ of the Cauchy problem of the one-species VPB system \eqref{u}, we will encounter the terms like $\left\|\nabla_x^Nu\right\|^2$ whose temporal decay estimates is not obtained by Lemma \ref{non-weight}. Thus we need to treat the cases $|\alpha|+|\beta|=N-1$ and $|\alpha|+|\beta|=N$ case by case individually. For our later use, we now write down the macroscopic equations of the one-species VPB system \eqref{u}$_1$-\eqref{u}$_2$ up to the third-order moments. As in \cite{Duan-Strain-ARMA-2011}, for any $v=v(\xi)$, if we define moment functions $\Theta_{ij}(v)$ and $\Lambda_i(v),~1\leq i,j\leq3,$ by
\begin{equation*}
\Theta_{ij}(v)=\int_{\R^3}(\xi_i\xi_j-1){\bf M}^{\frac{1}{2}}vd\xi,\quad
\Lambda_i(v)=\frac{1}{10}\int_{\R^3}\left(|\xi|^2-5\right)\xi_i{\bf M}^{\frac{1}{2}}vd\xi,
\end{equation*}
then, one can derive from \eqref{u} a fluid-type system of equations
\begin{equation*}\label{fluid-equ.}
\begin{cases}
\partial_ta+\nabla_x\cdot b=0,\\
\partial_tb+\nabla_x\cdot(a+2c)+\nabla_x\cdot\Theta(u_2)-\nabla_x\phi=\nabla_x\phi a,\\
\partial_tc+\frac{1}{3}\nabla_x\cdot b+\frac53\nabla_x\cdot\Lambda(u_2)=\frac{1}{3}\nabla_x\phi\cdot b,\\
\triangle_x\phi=a
\end{cases}
\end{equation*}
and
\begin{equation*}\label{mic-u}
\begin{cases}
\partial_t\Theta_{ij}(u_2)+\partial_{x_i}b_j+\partial_{x_j}b_i
-\frac{2}{3}\delta_{ij}\nabla_x\cdot\Lambda(u_2)
=\Theta_{ij}(r+G)-\frac{2}{3}\delta_{ij}\nabla_x\phi\cdot b,\\
\partial_t\Lambda_{i}(u_2)+\partial_{x_i}c=\Lambda_i(u_2)(r+G)
\end{cases}
\end{equation*}
with
\begin{equation*}
  r=-\xi\cdot\nabla_xu_2-{\bf L}u_2,\quad
  G=\Gamma(u,u)+\frac{1}{2}\xi\cdot\nabla_x\phi u-\nabla_x\phi\cdot\nabla_{\xi}u,
\end{equation*}
where $r$ is a linear term related only to the microscopic component $u_2$ and $G$ is a quadratic nonlinear term.

A similar process as in \cite{Duan-Yang-SIMA-2010, Duan-Yamg-Zhao-JDE-2012, Duan-Yang-Zhao-M3AS-2013} yields
\begin{equation}\label{N-mac}
 \begin{split}
 \frac{d}{dt}\mathfrak{E}_N^{int}(t)+\left\|\nabla_x^N(a,b,c)\right\|^2+\left\|\nabla_x^{N+1}\phi\right\|^2
  \lesssim\left\|\nabla_x^Nu_2\right\|_{\nu}^2+(1+t)^{-(N+\frac{3}{2})}\left(\epsilon_0+\CE_{\infty}^2(t)\right),
 \end{split}
\end{equation}
where the instant energy functional $\mathfrak{E}_N^{int}(t)$ is defined as
\begin{eqnarray*}
  \mathfrak{E}_N^{int}(t)&=&-\sum_{|\alpha|=N-1}\left[\sum_{j,m}\left(\partial^{\alpha+e_j}\Theta_{jm}(u_2)
  +\frac{1}{2}\partial^{\alpha+e_m}\Theta_{mm}(u_2),\partial^{\alpha}b_m\right)\right.\\
  &&\left.+\sum_{j}\left(\partial^{\alpha+e_j}\Lambda_j(u_2),\partial^{\alpha}c\right)
  +\kappa\left(\partial^{\alpha}\nabla_xb,\partial^{\alpha}a\right)\right]
 \end{eqnarray*}
and it is easy to see that
$$
\left|\mathfrak{E}_N^{int}(t)\right|\lesssim\left|\nabla_x^Nu\right|^2+\left|\nabla_x^{N-1}{\bf P}u\right|^2.
$$
On the other hand, for $|\alpha|=N-1$, we can deduce from Remark \ref{u2-(N-1)} that
\begin{eqnarray}\label{x-u2(N-1)}
   &&\frac{d}{dt}\left(M\left\|\nabla_x^{N-1}u_2(t)\right\|^2
   +\left\|w_{-\ell}\nabla_x^{N-1}u_2(t)\right\|^2\right)\nonumber\\
   &&+(1+t)^{-1-\vartheta}\left\|\langle\xi\rangle w_{-\ell}\nabla_x^{N-1}u_2\right\|^2+\left\|w_{-\ell}\nabla_x^{N-1}u_2\right\|^2_{\nu}\\
   &\lesssim&\CE_{\infty}(t)\sum_{|\alpha_1|<N-1}(1+t)^{-(|\alpha-\alpha_1|+2)+(1+\vartheta)\frac{1-\gamma}{2-\gamma}}
   \left(\left\|\langle\xi\rangle w_{-\ell}\partial^{\alpha_1}\nabla_{\xi}u_2\right\|^2
   +\left\|\langle\xi\rangle^{\frac{1}{2}}w_{-\ell}\partial^{\alpha_1}u_2\right\|^2\right)\nonumber\\
   &&+(1+t)^{-N-\frac{3}{2}+\frac{2}{3}}\left(\epsilon_0+\CE^2_{\infty}(t)\right)
   +\left\|\nabla_x^{N}u\right\|_{\nu}^2.\nonumber
 \end{eqnarray}
For $|\alpha|=N$, we apply $\partial^{\alpha}(|\alpha|=N)$ to the equation $\eqref{u}_1$ to get that
\begin{eqnarray}\label{x-u-equ.}
  &&\partial_t\partial^{\alpha}u+\xi\cdot\nabla_x\partial^{\alpha}u
  +\nabla_x\phi\cdot\nabla_{\xi}\partial^{\alpha}u\nonumber\\
  &&-\frac{1}{2}\xi\cdot\nabla_x\phi\partial^{\alpha}u
  +\sum_{|\alpha_1|<|\alpha|}C_{\alpha}^{\alpha_1}\partial^{\alpha-\alpha_1}\nabla_x\phi\cdot
  \nabla_{\xi}\partial^{\alpha_1}u\\
  &&-\frac{1}{2}\sum_{|\alpha_1|<|\alpha|}C_{\alpha}^{\alpha_1}\xi\cdot\partial^{\alpha-\alpha_1}\nabla_x\phi
  \partial^{\alpha_1}u-\partial^{\alpha}\nabla_x\phi\cdot\xi {\bf M}^{\frac{1}{2}}
  -{\bf L}\partial^{\alpha}u\nonumber\\
  &=&\partial^{\alpha}\Gamma(u,u).\nonumber
 \end{eqnarray}
Multiplying the above identity \eqref{x-u-equ.} by $\partial^{\alpha}u$, taking summation over $|\alpha|=N$, and integrating the final resulting identity with respect to $x$ and $\xi$ over $\R^3_x\times\R_\xi^3$, we have by employing similar analysis used to deduce \eqref{x-u2} that
\begin{eqnarray}\label{x-u}
  &&\frac{1}{2}\frac{d}{dt}\sum_{|\alpha|=N}\left\{\left\|\partial^{\alpha}u(t)\right\|^2
  +\left\|\partial^{\alpha}\nabla_x\phi\right\|^2\right\}
  +\eta_0\sum_{|\alpha|=N}\left\|\partial^{\alpha}u_2\right\|^2_{\nu}\nonumber\\
  &\lesssim&(1+t)^{-N-\frac{3}{2}+\frac{1}{6}}\left(\epsilon_0+\CE^2_{\infty}(t)\right)
  +\sum_{|\alpha|=N}\left(\eta\left\|\partial^{\alpha}u\right\|_{\nu}^2
  +\left\|\nabla_x\phi\right\|_{L^{\infty}}^{\frac{2-\gamma}{1-\gamma}}
  \left\|\langle\xi\rangle\partial^{\alpha}u\right\|^2\right)\nonumber\\
  &\lesssim&C_{\eta}(1+t)^{-N-\frac{3}{2}+\frac{1}{6}}\left(\epsilon_0+\CE^2_{\infty}(t)\right)
  +\sum_{|\alpha|=N}\left(\eta\left\|\partial^{\alpha}u\right\|_{\nu}^2
  +(1+t)^{-\frac{5}{4}}\CE^{-\frac{5}{8}}\left\|\langle\xi\rangle\partial^{\alpha}u\right\|^2\right),
 \end{eqnarray}
 since $\frac{2-\gamma}{1-\gamma}\geq\frac{5}{4}$ for $-3<\gamma<0$ and $\|\nabla_x\phi(t)\|_{L^{\infty}}^{5/4}\lesssim(1+t)^{-{5/4}}\CE_\infty^{{5/8}}(t)$ by (\ref{Energy}).

For the weighted estimate on the term involving the pure space derivative of $u$, multiplying \eqref{x-u-equ.} by $w_{-\ell}^2\partial^{\alpha}u$ and integrating the resulting equation over $\R^3_{x}\times\R^3_{\xi}$, one has
\begin{eqnarray}\label{x-w-u1}
   &&\frac{d}{dt}\left\|w_{-\ell}\partial^{\alpha}u(t)\right\|^2+(1+t)^{-1-\vartheta}
   \left\|\langle\xi\rangle w_{-\ell}\partial^{\alpha}u\right\|^2
   +\left\|w_{-\ell}\partial^{\alpha}u_2\right\|^2_{\nu}\nonumber\\
   &\lesssim&-\underbrace{\left(\partial^{\alpha}\left(\nabla_x\phi\cdot\nabla_{\xi}u\right),
   w_{-\ell}^2\partial^{\alpha}u\right)}_{I_1}
   +\underbrace{\frac 12\left(\partial^{\alpha}\left(\xi\cdot\nabla_x\phi u\right),w_{-\ell}^2\partial^{\alpha}u\right)}_{I_2}
   +\underbrace{\left(K\partial^{\alpha}u_2,w_{-\ell}^2\partial^{\alpha}u\right)}_{I_3}\\
   &&+\underbrace{\left(\partial^{\alpha}\nabla_x\phi\cdot\xi {\bf M}^{\frac{1}{2}},w_{-\ell}^2\partial^{\alpha}u\right)}_{I_4}
   +\underbrace{\left(\partial^{\alpha}\Gamma(u,u),w_{-\ell}^2\partial^{\alpha}u\right)}_{I_5}.\nonumber
 \end{eqnarray}
For the third term $I_3$ in the right hand side of \eqref{x-w-u1}, by Lemma \ref{L-K} one has for each $\eta>0$ that
\begin{eqnarray*}
 I_3&\lesssim&\left\{\eta\left\|w_{-\ell}\partial^{\alpha}u_2\right\|_{\nu}
+C_{\eta}\left\|\chi_{|\xi|<2C_{\eta}}\langle\xi\rangle^{-\gamma \ell}\partial^{\alpha}u_2\right\|\right\}
\left\|w_{-\ell}\partial^{\alpha}u\right\|_{\nu}\\
&\lesssim&\eta\left\|w_{-\ell}\partial^{\alpha}u_2\right\|_{\nu}^2
+C_{\eta}\left\{\left\|\partial^{\alpha}u_2\right\|_{\nu}^2
+\left\|\partial^{\alpha}(a,b,c)\right\|^2\right\}.
 \end{eqnarray*}
For the fourth term $I_4$, we have from Cauchy's inequality that
\begin{equation*}
 I_4\lesssim\eta\left\|\partial^{\alpha}u\right\|_{\nu}^2+C_{\eta}\left\|\partial^{\alpha}\nabla_x\phi\right\|^2
 \lesssim\eta\left\|\partial^{\alpha}u_2\right\|_{\nu}^2+C_{\eta}\left\{\left\|\partial^{\alpha}\nabla_x\phi\right\|^2
 +\left\|\partial^{\alpha}(a,b,c)\right\|^2\right\}.
\end{equation*}
As to the last term $I_5$, one can get from Lemma \ref{Gamma} that
\begin{equation*}
I_5\lesssim (1+t)^{-N-\frac{3}{2}+\frac{1}{6}}\left(\epsilon_0+\CE^2_{\infty}(t)\right).
\end{equation*}
Now for the term $I_1$, we can write it as
\begin{equation*}
  I_1=\underbrace{\left(\nabla_x\phi\cdot\nabla_{\xi}\partial^{\alpha}u,w_{-\ell}^2\partial^{\alpha}u\right)}_{I_{1,1}}
  +\underbrace{\sum_{|\alpha_1|<|\alpha|}C_{\alpha}^{\alpha_1}\left(\partial^{\alpha-\alpha_1}\nabla_x\phi\cdot
  \nabla_{\xi}\partial^{\alpha_1}u,w_{-\ell}^2\partial^{\alpha}u\right)}_{I_{1,2}}.
\end{equation*}
Note that $\nabla_{\xi}w_{-\ell}^2(t,\xi)\lesssim\langle\xi\rangle w_{-\ell}^2(t,\xi)$, we
apply integration by parts with respect to $\xi$ and use (\ref{basic1}) to bound $I_{1,1}$ as follows
\begin{eqnarray*}
I_{1,1}&=&\frac 12\left(1,\nabla_x\phi\cdot\nabla_{\xi}\left\{(w_{-\ell}\partial^{\alpha}u)^2\right\}\right)
  -\frac 12\left(\nabla_x\phi\cdot\nabla_{\xi}w_{-\ell}^2(t,\xi),\left(\partial^{\alpha}u\right)^2\right)\\
  &\lesssim&\eta\left(\left\|w_{-\ell}\partial^{\alpha}u_2\right\|_{\nu}^2
  +\left\|\partial^{\alpha}(a,b,c)\right\|^2\right)
  +C_{\eta}\left\|\nabla_x\phi\right\|_{L^{\infty}}^{\frac{2-\gamma}{1-\gamma}}\left\|\langle\xi\rangle w_{-\ell} \partial^{\alpha}u\right\|^2.
 \end{eqnarray*}
And we can further write $I_{1,2}$ as
\begin{eqnarray*}
  I_{1,2}&=&\sum_{|\alpha_1|<|\alpha|}C_{\alpha}^{\alpha_1}\left(\partial^{\alpha-\alpha_1}\nabla_x\phi\cdot
  \nabla_{\xi}\partial^{\alpha_1}u_1,w_{-\ell}^2\partial^{\alpha}u\right)
  +\sum_{|\alpha_1|<|\alpha|}C_{\alpha}^{\alpha_1}\left(\partial^{\alpha-\alpha_1}\nabla_x\phi\cdot
  \nabla_{\xi}\partial^{\alpha_1}u_2,w_{-\ell}^2\partial^{\alpha}u_1\right)\\
  &&+\sum_{|\alpha_1|<|\alpha|}C_{\alpha}^{\alpha_1}\left(\partial^{\alpha-\alpha_1}\nabla_x\phi\cdot
  \nabla_{\xi}\partial^{\alpha_1}u_2,w_{-\ell}^2\partial^{\alpha}u_2\right).
 \end{eqnarray*}
Notice that the Maxwellian in the macroscopic component can absorb the weight function when $0<q\ll1$ and any algebraic power of $\langle\xi\rangle$, we can apply integration by parts with respect to $\xi$ again and use the definition \eqref{Energy} of $\CE_{\infty}(t)$ to control the first two terms in the right hand side of the inequality above as
\begin{equation*}
  \eta\left\|\partial^{\alpha}u\right\|_{\nu}^2+C_{\eta}(1+t)^{-N-\frac{3}{2}}\CE_{\infty}^2(t),
\end{equation*}
since $|\alpha|=N$. For the last term in the right hand side of the above inequality,  according to the range of $\alpha_1$, we estimate it in three cases: For the case of $|\alpha_1|=0$, we have by employing the H$\ddot{o}$lder inequality with $(p,q,r)=(2,\infty,2)$ that
\begin{eqnarray*}
  &&\left|\left(\partial^{\alpha}\nabla_x\phi\cdot
  \nabla_{\xi}u_2,w_{-\ell}^2\partial^{\alpha}u_2\right)\right|\\
  &\lesssim&\left\|\partial^{\alpha}\nabla_x\phi\right\|
  \left\|\langle\xi\rangle w_{1-\ell}\nabla_{\xi}u_2\right\|_{L^{\infty}(L^2_{\xi})}
  \left\|\langle\xi\rangle^{\frac{1}{2}}w_{-\ell}\partial^{\alpha}u_2\right\|\\
  &\lesssim&\left\|\partial^{\alpha}\nabla_x\phi\right\|
  \left\|\langle\xi\rangle w_{1-\ell}\nabla_{\xi}u_2\right\|_{L^{\infty}(L^2_{\xi})}
  \left\|\langle\xi\rangle w_{-\ell}\partial^{\alpha}u_2\right\|^{\frac{1-\gamma}{2-\gamma}}
  \left\|\langle\xi\rangle^{\frac{\gamma}{2}}w_{-\ell}\partial^{\alpha}u_2\right\|^{\frac{1}{2-\gamma}}\\
  &\lesssim&\eta\left((1+t)^{-1-\vartheta}\left\|\langle\xi\rangle w_{-\ell}\partial^{\alpha} u_2\right\|^2
   +\left\|w_{-\ell}\partial^{\alpha}u_2\right\|_{\nu}^2\right)\\
   &&+C_{\eta}(1+t)^{-\left(N-\frac{1}{2}\right)}(1+t)^{(1+\vartheta)\frac{1-\gamma}{2-\gamma}}
   \CE_{\infty}(t)\left\|\langle\xi\rangle w_{1-\ell}\nabla_{\xi}u_2\right\|_{L^{\infty}(L^2_{\xi})}^2;
 \end{eqnarray*}
For the case of $1\leq|\alpha_1|\leq N-2$,  we can use the H$\ddot{o}$lder inequality with $(p,q,r)=(3,6,2)$ to yield
\begin{eqnarray*}
  &&\sum_{1\leq\alpha_1\leq N-2}\left|\left(\partial^{\alpha-\alpha_1}\nabla_x\phi\cdot
  \nabla_{\xi}\partial^{\alpha_1}u_2,w_{-\ell}^2\partial^{\alpha}u_2\right)\right|\\
  &\lesssim&\left\|\partial^{\alpha-\alpha_1}\nabla_x\phi\right\|_{L^3}
  \left\|\langle\xi\rangle w_{1-\ell}\nabla_{\xi}\partial^{\alpha_1}u_2\right\|_{L^6(L^2_{\xi})}
  \left\|\langle\xi\rangle^{\frac{1}{2}}w_{-\ell}\partial^{\alpha}u_2\right\|\\
  &\lesssim&\left\|\partial^{\alpha-\alpha_1}\nabla_x\phi\right\|_{L^3}
  \left\|\langle\xi\rangle w_{1-\ell}\nabla_{\xi}\partial^{\alpha_1}u_2\right\|_{L^6(L^2_{\xi})}
  \left\|\langle\xi\rangle w_{-\ell}\partial^{\alpha}u_2\right\|^{\frac{1-\gamma}{2-\gamma}}
  \left\|\langle\xi\rangle^{\frac{\gamma}{2}}w_{-\ell}\partial^{\alpha}u_2\right\|^{\frac{1}{2-\gamma}}\\
  &\lesssim&\eta\left((1+t)^{-1-\vartheta}\left\|\langle\xi\rangle w_{-\ell}\partial^{\alpha}u_2\right\|^2
   +\left\|w_{-\ell}\partial^{\alpha}u_2\right\|_{\nu}^2\right)\\
  &&+C_{\eta}(1+t)^{-(|\alpha-\alpha_1|+1)}(1+t)^{(1+\vartheta)\frac{1-\gamma}{2-\gamma}}
   \CE_{\infty}(t)\left\|\langle\xi\rangle w_{1-\ell}\nabla_{\xi}\nabla_x\partial^{\alpha_1}u_2\right\|^2.
 \end{eqnarray*}
Finally for the case of $|\alpha_1|=N-1$, one can deduce from the H$\ddot{o}$lder inequality with $(p,q,r)=(\infty,2,2)$ that
\begin{eqnarray*}
  &&\sum_{|\alpha_1|=N-1}\left|\left(\partial^{\alpha-\alpha_1}\nabla_x\phi\cdot
  \nabla_{\xi}\partial^{\alpha_1}u_2,w_{-\ell}^2\partial^{\alpha}u_2\right)\right|\\
  &\lesssim&\left\|\partial^{\alpha-\alpha_1}\nabla_x\phi\right\|_{L^{\infty}}
  \left\|\langle\xi\rangle w_{1-\ell}\nabla_{\xi}\partial^{\alpha_1}u_2\right\|
  \left\|\langle\xi\rangle^{\frac{1}{2}}w_{-\ell}\partial^{\alpha}u_2\right\|\\
  &\lesssim&\eta\left((1+t)^{-1-\vartheta}\left\|\langle\xi\rangle w_{-\ell}\partial^{\alpha}u_2\right\|^2
   +\left\|w_{-\ell}\partial^{\alpha}u_2\right\|_{\nu}^2\right)\\
   &&+C_{\eta}(1+t)^{-(|\alpha-\alpha_1|+2)}(1+t)^{(1+\vartheta)\frac{1-\gamma}{2-\gamma}}
   \CE_{\infty}(t)\left\|\langle\xi\rangle w_{1-\ell}\nabla_{\xi}\partial^{\alpha_1}u_2\right\|^2.
 \end{eqnarray*}
Collecting the estimates above, one has
\begin{eqnarray*}
I_{1,2}&\lesssim&\eta\left((1+t)^{-1-\vartheta}\left\|\langle\xi\rangle w_{-\ell}\partial^{\alpha}u_2\right\|^2
   +\left\|w_{-\ell}\partial^{\alpha}u_2\right\|_{\nu}^2\right)+C_{\eta}(1+t)^{-N-\frac{3}{2}}\CE_{\infty}^2(t)\\
  &&+C_{\eta}(1+t)^{-(|\alpha-\alpha_1|+2)+(1+\vartheta)\frac{1-\gamma}{2-\gamma}}
   \CE_{\infty}(t)\sum_{2\leq|\alpha_1|<|\alpha|}\left\|\langle\xi\rangle w_{1-\ell}\nabla_{\xi}\partial^{\alpha_1}u_2\right\|^2\\
   &&+C_{\eta}(1+t)^{-\left(N-\frac{1}{2}\right)+(1+\vartheta)\frac{1-\gamma}{2-\gamma}}
   \CE_{\infty}(t)\left\|\langle\xi\rangle w_{1-\ell}\nabla_{\xi}u_2\right\|_{L^{\infty}(L^2_{\xi})}^2.
\end{eqnarray*}
Consequently, we can bound $I_1$ by
\begin{eqnarray*}
I_1&\lesssim&\eta\left((1+t)^{-1-\vartheta}\left\|\langle\xi\rangle w_{-\ell}\partial^{\alpha}u_2\right\|^2
   +\left\|w_{-\ell}\partial^{\alpha}u_2\right\|_{\nu}^2+\left\|\partial^{\alpha}(a,b,c)\right\|^2\right)\\
   && +C_{\eta}\left\|\nabla_x\phi\right\|_{L^{\infty}}^{\frac{2-\gamma}{1-\gamma}}\left\|\langle\xi\rangle w_{-\ell} \partial^{\alpha}u\right\|^2+C_{\eta}(1+t)^{-N-\frac{3}{2}}\CE_{\infty}^2(t)\\
  &&+C_{\eta}(1+t)^{-(|\alpha-\alpha_1|+2)+(1+\vartheta)\frac{1-\gamma}{2-\gamma}}
   \CE_{\infty}(t)\sum_{2\leq|\alpha_1|<|\alpha|}\left\|\langle\xi\rangle w_{1-\ell}\nabla_{\xi}\partial^{\alpha_1}u_2\right\|^2\\
   &&+C_{\eta}(1+t)^{-\left(N-\frac{1}{2}\right)+(1+\vartheta)\frac{1-\gamma}{2-\gamma}}
   \CE_{\infty}(t)\left\|\langle\xi\rangle w_{1-\ell}\nabla_{\xi}u_2\right\|_{L^{\infty}(L^2_{\xi})}^2.
\end{eqnarray*}

At last, we deal with the term $I_2$. To this end, we further rewrite it as:
\begin{equation*}
  I_2=\underbrace{\frac 12\left(\xi\cdot\nabla_x\phi\partial^{\alpha}u,w_{-\ell}^2\partial^{\alpha}u\right)}_{I_{2,1}}
  +\underbrace{\frac 12\sum_{|\alpha_1|<|\alpha|}C_{\alpha}^{\alpha_1}\left(\xi\cdot
  \partial^{\alpha-\alpha_1}\nabla_x\phi\partial^{\alpha_1}u,w_{-\ell}^2\partial^{\alpha}u\right)}_{I_{2,2}},
\end{equation*}
Similar to the estimation of $I_{1,1}$ and $I_{1,2}$,  $I_{2,1}$ and $I_{2,2}$ can be estimated as follows:
\begin{equation*}
   I_{2,1}\lesssim \eta\left(\left\|w_{-\ell}\partial^{\alpha}u_2\right\|_{\nu}^2+\left\|\partial^{\alpha}(a,b,c)\right\|^2\right)
  +C_{\eta}\left\|\nabla_x\phi\right\|_{L^{\infty}}^{\frac{2-\gamma}{1-\gamma}}\left\|\langle\xi\rangle w_{-\ell} \partial^{\alpha}u\right\|^2,
\end{equation*}
and
\begin{eqnarray*}
  I_{2,2}&=&\sum_{|\alpha_1|<|\alpha|}C_{\alpha}^{\alpha_1}\left\{
  \left(\xi\cdot\partial^{\alpha-\alpha_1}\nabla_x\phi\partial^{\alpha_1}u_1,w_{-\ell}^2\partial^{\alpha}u\right)
  +\left(\xi\cdot\partial^{\alpha-\alpha_1}\nabla_x\phi\partial^{\alpha_1}u_2,
  w_{-\ell}^2\partial^{\alpha}u_1\right)\right.\\
  &&+\left(\xi\cdot\partial^{\alpha-\alpha_1}\nabla_x\phi\partial^{\alpha_1}u_2,
  w_{-\ell}^2\partial^{\alpha}u_2\right)\Big\}\\
  &\lesssim& \eta\left\|w_{-\ell}\partial^{\alpha}u_2\right\|_{\nu}^2+\left\|\partial^{\alpha}(a,b,c)\right\|^2
  +(1+t)^{-(N+\frac{3}{2})}\CE_{\infty}^2(t)\\
 &&+\eta\left((1+t)^{-1-\vartheta}\left\|\langle\xi\rangle w_{-\ell}\partial^{\alpha}u_2\right\|^2
   +\left\|w_{-\ell}\partial^{\alpha}u_2\right\|_{\nu}^2\right)\\
   &&+C_{\eta}\sum_{|\alpha_1|<|\alpha|}(1+t)^{-(|\alpha-\alpha_1|+2)+(1+\vartheta)\frac{1-\gamma}{2-\gamma}}
   \left\|\langle\xi\rangle^{\frac{1}{2}}w_{-\ell}\partial^{\alpha_1}u_2\right\|^2\CE_{\infty}(t),
 \end{eqnarray*}
Thus, we obtain
\begin{eqnarray*}
I_2&\leq&\eta\left((1+t)^{-1-\vartheta}\left\|\langle\xi\rangle w_{-\ell}\partial^{\alpha}u_2\right\|^2
   +\left\|w_{-\ell}\partial^{\alpha}u_2\right\|_{\nu}^2+\left\|\partial^{\alpha}(a,b,c)\right\|^2\right)\\
   && +C_{\eta}\left\|\nabla_x\phi\right\|_{L^{\infty}}^{\frac{2-\gamma}{1-\gamma}}\left\|\langle\xi\rangle w_{-\ell} \partial^{\alpha}u\right\|^2+C_{\eta}(1+t)^{-N-\frac{3}{2}}\CE_{\infty}^2(t)\\
  &&+C_{\eta}(1+t)^{-(|\alpha-\alpha_1|+2)+(1+\vartheta)\frac{1-\gamma}{2-\gamma}}
   \sum_{2\leq|\alpha_1|<|\alpha|}
   \left\|\langle\xi\rangle^{\frac{1}{2}}w_{-\ell}\partial^{\alpha_1}u_2\right\|^2\CE_{\infty}(t)\\
   &&+C_{\eta}(1+t)^{-(N-\frac{1}{2})+(1+\vartheta)\frac{1-\gamma}{2-\gamma}}
   \left\|\langle\xi\rangle^{\frac{1}{2}}w_{-\ell}u_2\right\|_{L^{\infty}(L^2_{\xi})}^2\CE_{\infty}(t).
\end{eqnarray*}

Substituting the estimates on $I_1$-$I_5$  into \eqref{x-w-u1} and choosing $\eta>0$ small enough, we have
\begin{eqnarray}\label{x-w-u}
   &&\frac{d}{dt}\left\|w_{-\ell}\partial^{\alpha}u(t)\right\|^2+q\vartheta(1+t)^{-1-\vartheta}
   \left\|\langle\xi\rangle w_{-\ell}\partial^{\alpha}u\right\|^2
   +\frac{1}{2}\left\|w_{-\ell}\partial^{\alpha}u_2\right\|^2_{\nu}\nonumber\\
   &\lesssim&\left\|\partial^{\alpha}u\right\|_{\nu}^2+\left\|\partial^{\alpha}\nabla_x\phi\right\|^2
  +(1+t)^{-N-\frac{3}{2}+\frac{1}{6}}\CE^2_{\infty}(t)\\
  &&+(1+t)^{-(N-\frac{1}{2})+(1+\vartheta)\frac{1-\gamma}{2-\gamma}}
   \left(\left\|\langle\xi\rangle w_{1-\ell}\nabla_{\xi}u_2\right\|_{L^{\infty}(L^2_{\xi})}^2
   +\left\|\langle\xi\rangle^{\frac{1}{2}}w_{-\ell}u_2\right\|_{L^{\infty}(L^2_{\xi})}^2\right)\CE_{\infty}(t)\nonumber\\
   &&+\sum_{2\leq|\alpha_1|<|\alpha|}(1+t)^{-(|\alpha-\alpha_1|+2)+(1+\vartheta)\frac{1-\gamma}{2-\gamma}}
   \left(\left\|\langle\xi\rangle w_{1-\ell}\nabla_{\xi}\partial^{\alpha_1}u_2\right\|^2
   +\left\|\langle\xi\rangle^{\frac{1}{2}}w_{-\ell}\partial^{\alpha_1}u_2\right\|^2\right)\CE_{\infty}(t).\nonumber
 \end{eqnarray}
 Choosing $0<M_1\ll M_2$ suitably large and $\delta$ small enough in (\ref{a priori assumption}), and performing $[\eqref{N-mac}+M_1\times\eqref{x-u}]\times M_2+\eqref{x-w-u}+\eqref{x-u2(N-1)}$, we can obtain
\begin{eqnarray}\label{x-u-close}
  &&\frac{d}{dt}\left\{M\left\|\nabla_x^{N-1}u_2\right\|^2+\left\|w_{-\ell}\nabla_x^{N-1}u_2\right\|^2
  +\left\|w_{-\ell}\nabla_x^Nu\right\|^2\right.\nonumber\\
  &&\left.+M_2\left[\mathfrak{E}_N^{int}(t)+M_1\left(\left\|\nabla_x^Nu\right\|^2
  +\left\|\nabla_x^{N+1}\phi\right\|^2\right)\right]\right\}\nonumber\\
  &&+(1+t)^{-1-\vartheta}\left\{\left\|\langle\xi\rangle w_{-\ell}\nabla_x^{N-1}u\right\|^2
  +\left\|\langle\xi\rangle w_{-\ell}\nabla_x^{N}u_2\right\|^2\right\}\nonumber\\
  &&+\left\|w_{-\ell}\nabla_x^{N-1}u_2\right\|^2_{\nu}
  +\left\|w_{-\ell}\nabla_x^Nu\right\|^2_{\nu}+\left\|\nabla_x^{N+1}\phi\right\|^2\\
  &\lesssim&(1+t)^{-(N-\frac{1}{2})+(1+\vartheta)\frac{1-\gamma}{2-\gamma}}
   \left(\left\|\langle\xi\rangle w_{1-\ell}\nabla_{\xi}u_2\right\|_{L^{\infty}(L^2_{\xi})}^2
   +\left\|\langle\xi\rangle^{\frac{1}{2}}w_{-\ell}u_2\right\|_{L^{\infty}(L^2_{\xi})}^2\right)
   \CE_{\infty}(t)\nonumber\\
   &&+\sum_{|\alpha|=N-1}^{N}\sum_{|\alpha_1|<|\alpha|}
   (1+t)^{-(|\alpha-\alpha_1|+2)+(1+\vartheta)\frac{1-\gamma}{2-\gamma}}
   \left(\left\|\langle\xi\rangle w_{1-\ell}\nabla_{\xi}\partial^{\alpha_1}u_2\right\|^2\right.\nonumber\\
   &&\left.+\left\|\langle\xi\rangle^{\frac{1}{2}}w_{-\ell}\partial^{\alpha_1}u_2\right\|^2\right)\CE_{\infty}(t)
   +(1+t)^{-N-\frac{3}{2}+\frac{2}{3}}\CE^2_{\infty}(t).\nonumber
 \end{eqnarray}
Integrating \eqref{x-u-close} with respect to $t$ over $[0,t]$, we have from the definition of the temporal energy functional $\CE_{\infty}(t)$ defined in \eqref{Energy} that
\begin{eqnarray}\label{x-high}
  &&\left\|w_{-\ell}\nabla_x^{N-1}u_2\right\|^2+\left\|w_{-\ell}\nabla_x^Nu\right\|^2
  +\left\|\nabla_x^{N+1}\phi\right\|^2\nonumber\\
  &&+\int_0^t(1+s)^{-1-\vartheta}\left(\left\|\langle\xi\rangle w_{-\ell}\nabla_x^Nu(s)\right\|^2
  +\left\|\langle\xi\rangle w_{-\ell}\nabla_x^{N-1}u_2(s)\right\|^2\right)ds\nonumber\\
  &&+\int_0^t\left(\left\|w_{-\ell}\nabla_x^{N-1}u_2(s)\right\|^2_{\nu}
  +\left\|w_{-\ell}\nabla_x^Nu(s)\right\|^2_{\nu}+\left\|\nabla_x^{N+1}\phi\right\|^2\right)ds\\
  &\lesssim&\epsilon_0+\CE^2_{\infty}(t)
  +\CE_{\infty}(t)\int_0^t(1+s)^{-(N-\frac{1}{2})+(1+\vartheta)\frac{1-\gamma}{2-\gamma}}
   \left(\left\|\langle\xi\rangle w_{1-\ell}\nabla_{\xi}u_2\right\|_{L^{\infty}(L^2_{\xi})}^2
   +\left\|\langle\xi\rangle^{\frac{1}{2}}w_{-\ell}u_2\right\|_{L^{\infty}(L^2_{\xi})}^2\right)\nonumber\\
   &&+\CE_{\infty}(t)\sum_{|\alpha|=N-1}^{N}\sum_{2\leq|\alpha_1|<|\alpha|}
   \int_0^t(1+s)^{-(|\alpha-\alpha_1|+2)+(1+\vartheta)\frac{1-\gamma}{2-\gamma}}
   \left(\left\|\langle\xi\rangle w_{1-\ell}\nabla_{\xi}\partial^{\alpha_1}u_2\right\|^2
   +\left\|\langle\xi\rangle^{\frac{1}{2}}w_{-\ell}\partial^{\alpha_1}u_2\right\|^2\right)ds.\nonumber
 \end{eqnarray}
Moreover, we multiply \eqref{x-u-close} by $(1+t)^{N-\frac{1}{2}+\vartheta}$ and integrate the resulting equality with respect to $t$ over $[0,t]$ to obtain
\begin{eqnarray}\label{decay-high2}
 &&(1+t)^{N-\frac{1}{2}+\vartheta}\left(\left\|w_{-\ell}\nabla_x^{N-1}u_2\right\|^2
 +\left\|w_{-\ell}\nabla_x^Nu\right\|^2+\left\|\nabla_x^{N+1}\phi\right\|^2\right)\nonumber\\
 &&+\int_0^t(1+s)^{N-\frac{3}{2}}\left(\left\|\langle\xi\rangle w_{-\ell}\nabla_x^{N-1}u_2(s)\right\|^2
  +\left\|\langle\xi\rangle w_{-\ell}\nabla_x^{N}u(s)\right\|^2\right)ds\nonumber\\
  &&+\int_0^t(1+s)^{N-\frac{1}{2}+\vartheta}\left(\left\|w_{-\ell}\nabla_x^{N-1}u_2(s)\right\|^2_{\nu}
  +\left\|w_{-\ell}\nabla_x^Nu(s)\right\|^2_{\nu}+\left\|\nabla_x^{N+1}\phi\right\|^2\right)ds\nonumber\\
  &\lesssim&(1+t)^{\vartheta}\left(\epsilon_0+\CE^2_{\infty}(t)\right)
  +\int_0^t(1+s)^{N-\frac{3}{2}+\vartheta}\left\|\nabla_x^{N-1}{\bf P}u(s)\right\|^2ds\\
  &&+\int_0^t(1+s)^{N-\frac{3}{2}+\vartheta}\left(\left\|w_{-\ell}\nabla_x^{N-1}u_2\right\|^2
  +\left\|w_{-\ell}\nabla_x^Nu\right\|^2
  +\left\|\nabla_x^{N+1}\phi\right\|^2\right)ds\nonumber\\
  &&+\CE_{\infty}(t)\int_0^t(1+s)^{\vartheta+(1+\vartheta)\frac{1-\gamma}{2-\gamma}}
  \left(\left\|\langle\xi\rangle w_{1-\ell}\nabla_{\xi}u_2\right\|_{L^{\infty}(L^2_{\xi})}^2
   +\left\|\langle\xi\rangle^{\frac{1}{2}}w_{-\ell}u_2\right\|_{L^{\infty}(L^2_{\xi})}^2\right)ds\nonumber\\
  &&+\CE_{\infty}(t)\sum_{|\alpha|=N-1}^{N}\sum_{|\alpha_1|<|\alpha|}
  \int_0^t(1+s)^{-(|\alpha-\alpha_1|+2)+N-\frac{1}{2}+\vartheta+(1+\vartheta)\frac{1-\gamma}{2-\gamma}}
  \left(\left\|\langle\xi\rangle w_{1-\ell}\nabla_{\xi}\partial^{\alpha_1}u_2\right\|^2\right.\nonumber\\
   &&\left.+\left\|\langle\xi\rangle^{\frac{1}{2}}w_{-\ell}\partial^{\alpha_1}u_2\right\|^2\right)ds.\nonumber
 \end{eqnarray}
For the second term in the right hand side of \eqref{decay-high2}, by Lemma \ref{non-weight}, it can be estimated as follows
\begin{equation}\label{second term of decay high2}
  \int_0^t(1+s)^{N-\frac{3}{2}+\vartheta}\left\|\nabla_x^{N-1}{\bf P}u(s)\right\|^2ds\lesssim(1+t)^{\vartheta}\left(\epsilon_0+\CE^2_{\infty}(t)\right).
\end{equation}
Employing Lemma \ref{in.} with $m=N-\frac{3}{2}$, the third term in the right hand side of \eqref{decay-high2} can be bounded by
\begin{eqnarray}\label{second term of decay high3}
   &&\int_0^t(1+s)^{N-\frac{3}{2}+\vartheta}\left(\left\|w_{-\ell}\nabla_x^{N-1}u_2\right\|^2
   +\left\|w_{-\ell}\nabla_x^Nu(s)\right\|^2+\left\|\nabla_x^{N+1}\phi(s)\right\|^2\right)ds\nonumber\\
   &\lesssim&\eta\int_0^t(1+s)^{N-\frac{3}{2}}\left(\left\|\langle\xi\rangle w_{-\ell}\nabla_x^{N-1}u_2(s)\right\|^2
   +\left\|\langle\xi\rangle w_{-\ell}\nabla_x^Nu(s)\right\|^2\right)ds\\
   &&+\eta\int_0^t(1+s)^{N-\frac{1}{2}+\vartheta}\left(\left\|w_{-\ell}\nabla_x^{N-1}u_2(s)\right\|^2_{\nu}
  +\left\|w_{-\ell}\nabla_x^Nu(s)\right\|^2_{\nu}+\left\|\nabla_x^{N+1}\phi(s)\right\|^2\right)ds\nonumber\\
  &&+C_{\eta}\int_0^t\left(\left\|w_{-\ell}\nabla_x^{N-1}u_2(s)\right\|^2_{\nu}
  +\left\|w_{-\ell}\nabla_x^Nu(s)\right\|^2_{\nu}+\left\|\nabla_x^{N+1}\phi(s)\right\|^2\right)ds.\nonumber
 \end{eqnarray}
Multiplying \eqref{x-high} by a suitably large positive constant and adding it to \eqref{decay-high2}, we can use \eqref{second term of decay high2} and \eqref{second term of decay high3} to obtain
\begin{eqnarray}\label{decay-high3}
 &&(1+t)^{N-\frac{1}{2}+\vartheta}\left(\left\|w_{-\ell}\nabla_x^{N-1}u_2(t)\right\|^2
 +\left\|w_{-\ell}\nabla_x^Nu(t)\right\|^2+\left\|\nabla_x^{N+1}\phi(t)\right\|^2\right)\nonumber\\
 &&+\int_0^t(1+s)^{N-\frac{3}{2}}\left(\left\|\langle\xi\rangle w_{-\ell}\nabla_x^{N-1}u_2(s)\right\|^2
  +\left\|\langle\xi\rangle w_{-\ell}\nabla_x^{N}u(s)\right\|^2\right)ds\nonumber\\
  &&+\int_0^t(1+s)^{N-\frac{1}{2}+\vartheta}\left(\left\|w_{-\ell}\nabla_x^{N-1}u_2(s)\right\|^2_{\nu}
  +\left\|w_{-\ell}\nabla_x^Nu(s)\right\|^2_{\nu}+\left\|\nabla_x^{N+1}\phi(s)\right\|^2\right)ds\\
  &\lesssim&\CE_{\infty}(t)
  \int_0^t(1+s)^{-3+N-\frac{1}{2}+\vartheta+(1+\vartheta)\frac{1-\gamma}{2-\gamma}}
  \left(\left\|\langle\xi\rangle w_{1-\ell}\nabla_{\xi}\nabla_x^{N-1}u_2\right\|^2\right.\nonumber\\
   &&\left.+\left\|\langle\xi\rangle^{\frac{1}{2}}w_{-\ell}\partial^{\alpha_1}u_2\right\|^2\right)ds
   +(1+t)^{\vartheta}\left(\epsilon_0+\CE^2_{\infty}(t)\right)\nonumber\\
   &&+\CE_{\infty}(t)\sum_{k=N-1}^{N}\int_0^t(1+s)^{-3+N-\frac{1}{2}+\vartheta+(1+\vartheta)\frac{1-\gamma}{2-\gamma}}
  \left\|\langle\xi\rangle w_{1-\ell}\nabla_{\xi}\nabla_x^ku_2(s)\right\|^2ds,\nonumber
 \end{eqnarray}
where we have used the following estimate
\begin{eqnarray}\label{error-term}
&&\CE_{\infty}(t)\sum_{|\alpha|=N-1}^{N}\sum_{|\alpha_1|<|\alpha|}
  \int_0^t(1+s)^{-(|\alpha-\alpha_1|+2)+N-\frac{1}{2}+\vartheta+(1+\vartheta)\frac{1-\gamma}{2-\gamma}}
  \left(\left\|\langle\xi\rangle w_{1-\ell}\nabla_{\xi}\partial^{\alpha_1}u_2\right\|^2\right.\nonumber\\
   &&\left.+\left\|\langle\xi\rangle^{\frac{1}{2}}w_{-\ell}\partial^{\alpha_1}u_2\right\|^2\right)ds\nonumber\\
   &\lesssim&\CE_{\infty}(t)\sum^{N-1}_{|\alpha_1|=N-2}
  \int_0^t(1+s)^{-3+N-\frac{1}{2}+\vartheta+(1+\vartheta)\frac{1-\gamma}{2-\gamma}}
  \left\|\langle\xi\rangle w_{1-\ell}\nabla_{\xi}\partial^{\alpha_1}u_2\right\|^2ds\\
   &&+\CE_{\infty}(t)
  \int_0^t(1+s)^{-3+N-\frac{1}{2}+\vartheta+(1+\vartheta)\frac{1-\gamma}{2-\gamma}}
  \left\|\langle\xi\rangle^{\frac{1}{2}}w_{-\ell}\nabla_x^{N-1}u_2\right\|^2ds\nonumber\\
   &&+\CE_{\infty}(t)\sum_{|\alpha_1|\leq N-3}
  \int_0^t(1+s)^{|\alpha_1|-1-\frac{1}{2}+\vartheta+(1+\vartheta)\frac{1-\gamma}{2-\gamma}}
  \left\|\langle\xi\rangle w_{1-\ell}\nabla_{\xi}\partial^{\alpha_1}u_2\right\|^2ds\nonumber\\
   &&+\CE_{\infty}(t)\sum_{|\alpha_1|\leq N-2}
  \int_0^t(1+s)^{|\alpha_1|-1-\frac{1}{2}+\vartheta+(1+\vartheta)\frac{1-\gamma}{2-\gamma}}
  \left\|\langle\xi\rangle^{\frac{1}{2}}w_{-\ell}\partial^{\alpha_1}u_2\right\|^2ds\nonumber
   \end{eqnarray}
and the facts that the second term in the right hand side of \eqref{error-term} can be absorbed by the second term in the left hand side of (\ref{decay-high3}) since $\CE_{\infty}(t)\leq \delta$ can be chosen sufficiently small, and the last two terms in the right hand side of \eqref{error-term} can be controlled by $(1+t)^{\vartheta}\left(\epsilon_0+\CE^2_{\infty}(t)\right)$ due to (\ref{induc.-Assu.2}) and $0<\vartheta\leq\frac{1}{9}$. Therefore, only the first term in the right hand side of \eqref{error-term} is kept in (\ref{decay-high3}).

To deduce a closed energy type estimate based on the estimate \eqref{decay-high3}, we need to deal with the term involving the higher order mixed spatial and velocity derivatives $\partial^\alpha_\beta u(t,x,\xi)$ for the case $(N-1\leq|\alpha|+|\beta|=k\leq N,~|\beta|\geq1)$. For this purpose, for $|\alpha|+|\beta|=k,~|\beta|\geq1$, we multiply $\eqref{mixed-w-u2}$ by $(1+t)^{r_{|\alpha|}+\vartheta}$ with
$$
r_{|\alpha|}=\left\{
\begin{array}{lll}
  |\alpha|+\frac{1}{2}, & {\text {if}}\ (\alpha,\beta)\in U_{\alpha,\beta}^{\text{low}}=&\left\{(\alpha,\beta)\mid|\alpha|+|\beta|\leq N-1\right\}\\
  &&\cup
\left\{(\alpha,\beta)\mid|\alpha|+|\beta|=N,|\alpha|<N-2\right\},\\[2mm]
{\frac{N+2|\alpha|}{3}-\frac 12}, &{\text {if}}\ (\alpha,\beta)\in U_{\alpha,\beta}^{\text{high}}=&\left\{(\alpha,\beta)\mid|\alpha|+|\beta|=N,N-2\leq|\alpha|\leq N-1\right\},
\end{array}
\right.
$$
and integrate the resulting inequality with respect to $t$ over $[0,t]$ to get that
\begin{eqnarray}\label{h.-decay0}
   &&\sum\limits_{|\alpha|+|\beta|=k}(1+t)^{r_{|\alpha|}+\vartheta}
   \left\|w_{|\beta|-\ell}\partial^{\alpha}_{\beta}u_2(t)\right\|^2\nonumber\\
   &&+\sum\limits_{|\alpha|+|\beta|=k}\int_0^t\left((1+s)^{r_{|\alpha|}-1}\left\|\langle\xi\rangle w_{|\beta|-\ell}\partial^{\alpha}_{\beta}u_2(s)\right\|^2
   +(1+t)^{r_{|\alpha|}+\vartheta}
   \left\|w_{|\beta|-\ell}\partial^{\alpha}_{\beta}u_2(s)\right\|^2_{\nu}\right)ds\nonumber\\
   &\lesssim&(1+t)^{\vartheta}\left(\epsilon_0+\CE^2_{\infty}(t)\right)
   +\sum_{|\beta_1|<|\beta|\atop |\alpha|+|\beta|=k}\int_0^t(1+s)^{r_{|\alpha|}+\vartheta}
   \left\|w_{|\beta_1|-\ell}\partial^{\alpha}_{\beta_1}u_2(s)\right\|^2_{\nu}ds\\
   &&+\eta\sum\limits_{|\alpha|+|\beta|=k}\int_0^t\left((1+s)^{r_{|\alpha|}-1}\left\|\langle\xi\rangle w_{|\beta|-\ell}\partial^{\alpha}_{\beta}u_2(s)\right\|^2
   +(1+s)^{r_{|\alpha|}+\vartheta}
   \left\|w_{|\beta|-\ell}\partial^{\alpha}_{\beta}u_2(s)\right\|^2_{\nu}\right)ds\nonumber\\
   &&+C_{\eta}\sum\limits_{|\alpha|+|\beta|=k}
   \int_0^t\left\|w_{|\beta|-\ell}\partial^{\alpha}_{\beta}u_2(s)\right\|_{\nu}^2ds+
   \sum\limits_{|\alpha|+|\beta|=k}\int_0^t(1+s)^{r_{|\alpha|}+\vartheta}
   \left\|w_{|\beta-e_i|-\ell}\partial^{\alpha+e_i}_{\beta-e_i}u_2(s)\right\|^2ds\nonumber\\
   &&+\CE_{\infty}(t)\sum_{|\alpha_1|<|\alpha|\atop |\alpha|+|\beta|=k}\int_0^t(1+s)^{-(|\alpha-\alpha_1|+2)
   +(1+\vartheta)\frac{1-\gamma}{2-\gamma}+r_{|\alpha|}+\vartheta}\left(\left\|\langle\xi\rangle w_{|\beta+e_i|-\ell}
   \nabla_{\xi}\partial^{\alpha_1}_{\beta}u_2(s)\right\|^2\right.\nonumber\\
   &&\left.+\left\|\langle\xi\rangle^{\frac{1}{2}}w_{|\beta|-\ell}\partial^{\alpha_1}_{\beta}u_2(s)\right\|^2\right)ds
   +\chi_{|\alpha|=N-1}\int_0^t(1+s)^{r_{|\alpha|}+\vartheta}\left\|\nabla_x^Nu(s)\right\|_{\nu}^2ds\nonumber
\end{eqnarray}
holds for $k=N-1, N$. Here  we have used Lemma \ref{in.} with $m=r_{|\alpha|}-1$.

Now integrating \eqref{mixed-w-u2} with respect to $t$ over $[0,t]$ and multiplying the resulting inequality by a suitably large positive constant $M$ and adding the final inequality to \eqref{h.-decay0} with $k=N-1$, we can get that
\begin{eqnarray}\label{h.-decay1}
   &&\sum\limits_{|\alpha|+|\beta|=N-1}(1+t)^{r_{|\alpha|}+\vartheta}
   \left\|w_{|\beta|-\ell}\partial^{\alpha}_{\beta}u_2(t)\right\|^2\nonumber\\
   &&+\sum\limits_{|\alpha|+|\beta|=N-1}\int_0^t\left((1+s)^{r_{|\alpha|}-1}\left\|\langle\xi\rangle w_{|\beta|-\ell}\partial^{\alpha}_{\beta}u_2(s)\right\|^2
   +(1+s)^{r_{|\alpha|}+\vartheta}
   \left\|w_{|\beta|-\ell}\partial^{\alpha}_{\beta}u_2(s)\right\|^2_{\nu}\right)ds\nonumber\\
   &\lesssim&(1+t)^{\vartheta}\left(\epsilon_0+\CE^2_{\infty}(t)\right)
   +\sum\limits_{|\alpha|+|\beta|=N-1}\int_0^t(1+s)^{r_{|\alpha|}+\vartheta}
   \left\|w_{|\beta-e_i|-\ell}\partial^{\alpha+e_i}_{\beta-e_i}u_2(s)\right\|^2ds\\
   &&+\CE_{\infty}(t)\sum_{|\alpha_1|=|\alpha|-1\atop |\alpha|+|\beta|=N-1}\int_0^t(1+s)^{-(|\alpha-\alpha_1|+2)
   +(1+\vartheta)\frac{1-\gamma}{2-\gamma}+r_{|\alpha|}+\vartheta}\left(\left\|\langle\xi\rangle w_{|\beta+e_i|-\ell}
   \nabla_{\xi}\partial^{\alpha_1}_{\beta}u_2(s)\right\|^2\right.\nonumber\\
   &&\left.+\left\|\langle\xi\rangle^{\frac{1}{2}}
   w_{|\beta|-\ell}\partial^{\alpha_1}_{\beta}u_2(s)\right\|^2\right)ds.\nonumber
    \end{eqnarray}
Here we have used \eqref{induc.-Assu.2} for $|\alpha|+|\beta|\leq
N-2$:
\begin{equation*}
   \sum_{|\beta_1|<|\beta|\atop |\alpha|+|\beta|=N-1}\int_0^t(1+s)^{r_{|\alpha|}+\vartheta}
    \left\|w_{|\beta_1|-\ell}\partial^{\alpha}_{\beta_1}u_2(s)\right\|^2_{\nu}ds
    \lesssim(1+t)^{\vartheta}\left(\epsilon_0+\CE^2_{\infty}(t)\right)
\end{equation*}
 and
\begin{eqnarray*}
   &&\sum_{|\alpha_1|<|\alpha|-1\atop |\alpha|+|\beta|=N-1}\int_0^t(1+s)^{-(|\alpha-\alpha_1|+2)
   +(1+\vartheta)\frac{1-\gamma}{2-\gamma}+r_{|\alpha|}+\vartheta}\left\|\langle\xi\rangle w_{|\beta+e_i|-\ell}\nabla_{\xi}\partial^{\alpha_1}_{\beta}u_2(s)\right\|^2ds\\
   &&+\sum_{|\alpha_1|\leq|\alpha|-1\atop |\alpha|+|\beta|=N-1}\int_0^t(1+s)^{-(|\alpha-\alpha_1|+2)
   +(1+\vartheta)\frac{1-\gamma}{2-\gamma}+r_{|\alpha|}+\vartheta}\left\|\langle\xi\rangle^{\frac{1}{2}} w_{|\beta|-\ell}\partial^{\alpha_1}_{\beta}u_2(s)\right\|^2ds\\
   &\lesssim&(1+t)^{\vartheta}\left(\epsilon_0+\CE^2_{\infty}(t)\right),
  \end{eqnarray*}
since
\begin{equation}\label{key2}
    -(|\alpha-\alpha_1|+2)+(1+\vartheta)\frac{1-\gamma}{2-\gamma}+r_{|\alpha|}+\vartheta
    \leq r_{|\alpha_1|}-1
\end{equation}
holds for $0<\vartheta\leq\frac{1}{9}$ and  $-3<\gamma<0$.

Note that
\begin{eqnarray}\label{key}
   &&(1+t)^{r_{|\alpha|}+\vartheta}
   \left\|w_{|\beta-e_i|-\ell}\partial^{\alpha+e_i}_{\beta-e_i}u_2(t)\right\|^2\nonumber\\
   &\lesssim& (1+t)^{r_{|\alpha+e_i|}-1}\left\|\langle\xi\rangle w_{|\beta-e_i|-\ell}
   \partial^{\alpha+e_i}_{\beta-e_i}u_2(t)\right\|^2+(1+t)^{r_{|\alpha+e_i|}+\vartheta}
   \left\|w_{|\beta-e_i|-\ell}\partial^{\alpha+e_i}_{\beta-e_i}u_2(t)\right\|_{\nu}^2,
 \end{eqnarray}
 for $|\alpha|+|\beta|=N-1,~|\beta|\geq1$, we can take a suitable linear combination of \eqref{h.-decay1} for each order of $|\alpha|$ to obtain
\begin{eqnarray}\label{h.-decay2}
   &&\sum_{|\alpha|=0}^{N-2}\sum\limits_{|\alpha|+|\beta|=N-1}
   \widetilde{M}_{|\alpha|}(1+t)^{{r_{|\alpha|}-\frac{1}{2}+\vartheta}}
   \left\|w_{|\beta|-\ell}\partial^{\alpha}_{\beta}u_2(t)\right\|^2\nonumber\\
   &&+\sum_{|\alpha|=0}^{N-2}\sum\limits_{|\alpha|+|\beta|=N-1}
   \int_0^t\left((1+s)^{r_{|\alpha|}-\frac{3}{2}}\left\|\langle\xi\rangle w_{|\beta|-\ell}\partial^{\alpha}_{\beta}u_2(s)\right\|^2
   +(1+s)^{r_{|\alpha|}-\frac{1}{2}+\vartheta}
   \left\|w_{|\beta|-\ell}\partial^{\alpha}_{\beta}u_2(s)\right\|^2_{\nu}\right)ds\nonumber\\
   &\lesssim&\sum_{|\alpha|=N-1}\int_0^t(1+s)^{r_{|\alpha|}-1}\left\|\langle\xi\rangle w_{-\ell}
   \partial^{\alpha}u_2(s)\right\|^2+(1+s)^{r_{|\alpha|}+\vartheta}
   \left\|w_{-\ell}\partial^{\alpha}u_2(s)\right\|_{\nu}^2ds\nonumber\\
   &&+\sum_{|\alpha|=N-2\atop |\alpha|+|\beta|=N-1}\int_0^t\left((1+s)^{r_{|\alpha+e_i|}-\frac{3}{2}}
   \left\|\langle\xi\rangle w_{-\ell}\partial^{\alpha+e_i}u_2(s)\right\|^2+(1+s)^{r_{|\alpha+e_i|}-\frac{1}{2}+\vartheta}
   \left\|w_{-\ell}\partial^{\alpha+e_i}u_2(s)\right\|_{\nu}^2\right)ds\\
   &&+\CE_{\infty}(t)\sum_{|\alpha|=1}^{N-2}\sum_{|\alpha_1|=|\alpha|-1\atop |\alpha|+|\beta|=N-1}\int_0^t(1+s)^{-(|\alpha-\alpha_1|+2)
   +(1+\vartheta)\frac{1-\gamma}{2-\gamma}+r_{|\alpha|}+\vartheta}\left\|\langle\xi\rangle w_{|\beta+e_i|-\ell}
   \nabla_{\xi}\partial^{\alpha_1}_{\beta}u_2(s)\right\|^2ds\nonumber\\
   &&+(1+t)^{\vartheta}\left(\epsilon_0+\CE^2_{\infty}(t)\right)\nonumber.
 \end{eqnarray}

Similarly, for $|\alpha|+|\beta|=N,~|\beta|\geq1$, by noticing that \eqref{key} holds also for our specially chosen of $r_{|\alpha|}$, we can deduce that
\begin{eqnarray}\label{NN}
   &&\sum_{|\alpha|=0}^{N-1}\sum_{|\alpha|+|\beta|=N}\widetilde{M}_{|\alpha|}\left\{(1+t)^{r_{|\alpha|}+\vartheta}
   \left\|w_{|\beta|-\ell}\partial^{\alpha}_{\beta}u_2(t)\right\|^2\right.\nonumber\\
   &&\left.+\int_0^t\left((1+s)^{r_{|\alpha|}-1}\left\|\langle\xi\rangle w_{|\beta|-\ell}\partial^{\alpha}_{\beta}u_2(s)\right\|^2+(1+s)^{r_{|\alpha|}+\theta}
   \left\|w_{|\beta|-\ell}\partial^{\alpha}_{\beta}u_2(s)\right\|^2_{\nu}\right)ds\right\}\nonumber\\
   &\lesssim&\sum_{|\alpha|=N-1\atop |\alpha|+|\beta|=N}\int_0^t\left((1+s)^{r_{|\alpha+e_i|}-1}
   \left\|\langle\xi\rangle w_{-\ell}\partial^{\alpha+e_i}u_2(s)\right\|^2+(1+s)^{r_{|\alpha+e_i|}+\vartheta}
   \left\|w_{-\ell}\partial^{\alpha+e_i}u_2(s)\right\|_{\nu}^2\right)ds\nonumber\\
   &&+\CE_{\infty}(t)\sum_{|\alpha|+|\beta|=N}
   \int_0^t(1+s)^{-(|\alpha-\alpha_1|+2)+(1+\vartheta)\frac{1-\gamma}{2-\gamma}+r_{|\alpha|}+\vartheta}
   \left(\sum_{|\alpha|=2}^{N-1}\sum_{|\alpha_1|=|\alpha|-2}^{|\alpha|-1}\left\|\langle\xi\rangle w_{|\beta+e_i|-\ell}
   \nabla_{\xi}\partial^{\alpha_1}_{\beta}u_2(s)\right\|^2\right.\nonumber\\
   &&\left.+\sum_{|\alpha|=1}^{N-1}\sum_{|\alpha_1|=|\alpha|-1}
   \left\|\langle\xi\rangle^{\frac{1}{2}}w_{|\beta|-\ell}\partial^{\alpha_1}_{\beta}u_2(s)\right\|^2\right)ds
   +(1+t)^{\vartheta}\left(\epsilon_0+\CE^2_{\infty}(t)\right)\\
   &&+\sum_{|\alpha|=0}^{N-1}\sum_{|\alpha|+|\beta|=N}\int_0^t(1+s)^{r_{|\alpha|}+\vartheta}
   \left\|w_{|\beta-e_i|-\ell}\partial^{\alpha}_{\beta-e_i}u_2(s)\right\|_{\nu}^2ds\nonumber\\
   &&+\chi_{|\alpha|=N-1}\int_0^t(1+s)^{r_{|\alpha|}+\vartheta}\left\|\nabla_x^Nu(s)\right\|_{\nu}^2ds.\nonumber
 \end{eqnarray}

Finally, by taking a suitable linear combination as $[\eqref{h.-decay2}\times M_1+\eqref{NN}]+\eqref{decay-high3}\times M_2$ for some suitably large positive constants $M_1\ll M_2$, we can get from the smallness of $\CE_{\infty}(t)$ and the fact \eqref{key2} that
\begin{lemma}\label{high-oder-estimates} Under the assumptions listed in Lemma \ref{lower_order_microscopic}, we can deduce that
\begin{eqnarray}\label{N-1=N}
   &&\sum_{k=N-1}^{N}\left\{\sum_{|\alpha|+|\beta|=k \atop|\beta|\geq1}(1+t)^{r_{|\alpha|}+\vartheta}
   \left\|w_{|\beta|-\ell}\partial^{\alpha}_{\beta}u_2(t)\right\|^2\right.\nonumber\\
   &&+(1+t)^{N-\tfrac{1}{2}+\vartheta}\left[\left\|w_{-\ell}\nabla_x^{N-1}u_2\right\|^2
   +\left\|w_{-\ell}\nabla_x^Nu\right\|^2+\left\|\nabla^{N+1}_x\phi\right\|^2\right]\\
   &&\left.+\sum_{|\alpha|+|\beta|=k}\int_0^t\left((1+s)^{r_{|\alpha|}-1}\left\|\langle\xi\rangle w_{|\beta|-\ell}\partial^{\alpha}_{\beta}u_2(s)\right\|^2
   +(1+s)^{r_{|\alpha|}+\vartheta}
   \left\|w_{|\beta|-\ell}\partial^{\alpha}_{\beta}u_2(s)\right\|^2_{\nu}\right)ds\right\}\nonumber\\
   &\lesssim&(1+t)^{\vartheta}\left(\epsilon_0+\CE^2_{\infty}(t)\right)\nonumber
  \end{eqnarray}
holds for ant $0\leq t\leq T$.
\end{lemma}

\begin{remark}\label{DECAY-enough} We want to emphasize here that the very reason why we assign the special temporal decay rates $r_{|\alpha|}$ to certain $L^2-$norm of $\partial^\alpha_\beta u$ as in \eqref{decay_rates} is to guarantee that the estimate \eqref{key} holds for each pair of multiindex $(\alpha,\beta)$ satisfying $|\alpha|+|\beta|\leq N$, especially for the case when $(\alpha,\beta)\in U_{\alpha,\beta}^{\text{high}}$, so that the linear combinations performed in \eqref{h.-decay2} and \eqref{NN} work well.

To show that the estimate \eqref{key} holds for each pair of multiindex $(\alpha,\beta)$ satisfying $|\alpha|+|\beta|\leq N$, we only need to consider the following three subcases:
\begin{itemize}
\item [{\bf Case I:}] For the case when all the temporal decay rates on the involving terms $\partial^\alpha_\beta u$ are given by $r_{|\alpha|}=|\alpha|+\frac{1}{2}$, it
is easy to see that \eqref{key} is true by Lemma \ref{in.}.
\item [{\bf Case II:}] If the corresponding rates on the involving terms $\partial^\alpha_\beta u$ are given by
$r_{|\alpha|}=\frac{N+2|\alpha|}{3}-\frac{1}{2}$, noticing that $\frac{N+2|\alpha|}{3}-\frac{1}{2}+\vartheta=\left(\frac{N+2|\alpha+e_i|}{3}-\frac{3}{2}\right)
+\left(\frac{1}{3}+\vartheta\right)$,
we have from Lemma \ref{in.} that
\begin{eqnarray*}
   &&(1+t)^{\frac{N+2|\alpha|}{3}-\frac{1}{2}+\vartheta}
   \left\|w_{|\beta-e_i|-\ell}\partial^{\alpha+e_i}_{\beta-e_i}u_2(t)\right\|^2\\
   &\lesssim& (1+t)^{\frac{N+2|\alpha+e_i|}{3}-\frac{3}{2}}\left\|\langle\xi\rangle w_{|\beta-e_i|-\ell}
   \partial^{\alpha+e_i}_{\beta-e_i}u_2(t)\right\|^2+(1+t)^{\frac{N+2|\alpha|}{3}-\frac{5}{6}
   +\frac{2-\gamma}{2}(\frac{1}{3}+\vartheta)}
   \left\|w_{|\beta-e_i|-\ell}\partial^{\alpha+e_i}_{\beta-e_i}u_2(t)\right\|_{\nu}^2\\
   &\lesssim& (1+t)^{\frac{N+2|\alpha+e_i|}{3}-\frac{3}{2}}\left\|\langle\xi\rangle w_{|\beta-e_i|-\ell}
   \partial^{\alpha+e_i}_{\beta-e_i}u_2(t)\right\|^2+(1+t)^{\frac{N+2|\alpha+e_i|}{3}-\frac{1}{2}+\vartheta}
   \left\|w_{|\beta-e_i|-\ell}\partial^{\alpha+e_i}_{\beta-e_i}u_2(t)\right\|_{\nu}^2,
 \end{eqnarray*}
where we have used the facts that for $-3<\gamma<0$
\begin{equation*}
  \frac{N+2|\alpha|}{3}-\frac{5}{6}+\frac{2-\gamma}{2}\left(\frac{1}{3}+\vartheta\right)
  \leq\frac{N+2|\alpha+e_i|}{3}-\frac{1}{2}+\vartheta,\quad 0<\vartheta\leq \frac 19,\quad \frac{4+\gamma}{-3\gamma}\geq\frac{1}{9}.
\end{equation*}
\item[{\bf Case III:}] The last case we need to deal with is $|\alpha|+|\beta|=N,~|\alpha|=N-3$. For such a case, noticing our definition of $r_{|\alpha|}$ given by \eqref{decay_rates}, one can deduce from Lemma \ref{in.} that
\begin{eqnarray*}
   &&(1+t)^{\left(N-3+\frac{1}{2}\right)+\vartheta}\left\|w_{|\beta-e_i|-\ell}\partial^{\alpha+e_i}_{\beta-e_i}u_2(t)\right\|^2\\
   &\lesssim& (1+t)^{(N-2+\frac{1}{2}-\frac{1}{3})-1}\left\|\langle\xi\rangle w_{|\beta-e_i|-\ell}
   \partial^{\alpha+e_i}_{\beta-e_i}u_2(t)\right\|^2+(1+t)^{N-2-\frac{5}{6}+\frac{2-\gamma}{2}\left(\frac{1}{3}+\vartheta\right)}
   \left\|w_{|\beta-e_i|-\ell}\partial^{\alpha+e_i}_{\beta-e_i}u_2(t)\right\|_{\nu}^2\\
   &\lesssim& (1+t)^{r_{N-2}-1}\left\|\langle\xi\rangle w_{|\beta-e_i|-\ell}
   \partial^{\alpha+e_i}_{\beta-e_i}u_2(t)\right\|^2+(1+t)^{r_{N-2}+\vartheta}
   \left\|w_{|\beta-e_i|-\ell}\partial^{\alpha+e_i}_{\beta-e_i}u_2(t)\right\|_{\nu}^2,
 \end{eqnarray*}
since by writing $\left(N-3+\frac{1}{2}\right)+\vartheta=(r_{N-2}-1)+\left(\frac{1}{3}+\vartheta\right)$, we can verify that
\begin{equation*}
  r_{N-2}-\frac{3}{2}+\frac{2-\gamma}{2}\left(\frac{1}{3}+\vartheta\right)
  \leq r_{N-2}-\frac{1}{2}+\vartheta,
\end{equation*}
when $\vartheta\leq\frac{4+\gamma}{-3\gamma}$. Since $\frac{4+\gamma}{-3\gamma}\geq \frac{1}{9}$ for $-3<\gamma<0$, it is sufficient to choose $0<\vartheta\leq \frac{1}{9}$.
\end{itemize}
Putting the above three subcases together yield the estimate \eqref{key}.
\end{remark}

\section{The Proof of Theorem \ref{Thm}}
This section is devoted to proving our main result Theorem \ref{Thm}. For this purpose, as pointed out in the very beginning of Section 3, the local solvability of the Cauchy problem of the one-species VPB system \eqref{u} for the whole range of cutoff intermolecular interactions can be obtained by employing the argument used in \cite{Guo-CPAM-2002} for the one-species VPB system \eqref{u} in a periodic box for hard sphere model and in \cite{Guo-ARMA-2003} for the Boltzmann equation in a periodic box for cutoff soft potentials. Now assume that such a local solution $u(t,x,\xi)$ has been extended to the time step $t=T$ for some $T>0$, that is, $u(t,x,\xi)$ is a solution of the Cauchy problem of the one-species VPB system \eqref{u} defined on the strip $\prod_T=[0,T]\times{\mathbb{R}_x^3}\times{\mathbb{R}_\xi^3}$, then the energy type estimates performed in Sections 3 and 4 tells us that if $u(t,x,\xi)$ satisfies the a priori assumption \eqref{a priori assumption}
for some sufficiently small positive constant $\delta>0$ and all $0\leq t\leq T$, then the estimates obtained in Lemma \ref{non-weight}, Lemma \ref{lower_order_microscopic}, Lemma \ref{high-oder-estimates}, \eqref{A-u2-opt1.}, and \eqref{A-u2-opt2.} tell us that
\begin{equation}\label{5.2}
  \CE_{\infty}(t)\lesssim\epsilon_0+\CE^2_{\infty}(t),\quad 0\leq t\leq T.
\end{equation}
The estimate \eqref{5.2} together with the assumption that $\epsilon_0>0$ is chosen sufficiently small tell us that
\begin{equation}\label{5.3}
  \CE_{\infty}(t)\lesssim\epsilon_0
\end{equation}
holds for all $0\leq t\leq T$. It is worth to emphasis that, on the one hand, the estimate \eqref{5.3} yields a time-independent estimate on the solution $u(t,x,\xi)$ for all $(t,x,\xi)\in\prod_T$ and, on the other hand, can be used to close the a priori assumption \eqref{a priori assumption}.

Having obtained \eqref{5.3}, the global existence and uniqueness of solution $u(t,x,\xi)$ to the Cauchy problem of the one-species VPB system \eqref{u} follows by a standard continuation argument as in \cite{Guo-CPAM-2002,Guo-JAMS-2012}, so we omit the details for brevity. This completes the proof of Theorem \ref{Thm}.

\appendix
\section{Almost optimal temporal decay for microscopic terms}
The main purpose of this section is to deduce the almost optimal temporal decay estimates on the $L^2-$norm of the lower order spatial derivatives of microscopic parts of $u$, that is
\begin{lemma}
  Take $|\alpha|\leq N-2$, $\frac{3}{4}<p<1$. Let $u(t,x,\xi)$ be a solution of the Cauchy problem of the one-species VPB system \eqref{u} and  $u_2$ be its  microscopic part.
  Then, under the a priori assumption (\ref{a priori assumption}), we have the following almost optimal temporal decay estimates
   \begin{equation}\label{A-u2-opt1.}
     \left\|\partial^{\alpha}u_2(t)\right\|^2\lesssim (1+t)^{-\left(|\alpha|+\frac{3}{2}\right)+1-p}\left(\epsilon_0+\CE^2_{\infty}(t)\right).
   \end{equation}
   Similarly, for $|\alpha|=N-2,~|\beta|=1$, we have the following estimate as well
   \begin{equation}\label{A-u2-opt2.}
     \left\|\nabla^{N-2}_x\nabla_{\xi}u_2(t)\right\|^2\lesssim (1+t)^{-\left(N-\frac{1}{2}\right)+2(1-p)}\left(\epsilon_0+\CE^2_{\infty}(t)\right).
   \end{equation}
   Moreover, for $l_2>N+\frac{1}{2}$ and $\ell\geq (1+\frac{1}{4p-3})l_2$, we have
  \begin{equation}\label{A-u2-wei2.}
     \left\|\langle\xi\rangle^{-\frac{\gamma}{2}l_2}\nabla^{N-2}_x\nabla_{\xi}u_2(t)\right\|^2\lesssim (1+t)^{-N+1}\left(\epsilon_0+\CE^2_{\infty}(t)\right).
   \end{equation}
\end{lemma}
\begin{proof}
  By \eqref{0-u2} and \eqref{x-u2}, we can deduce for $|\alpha|\leq N-2$ that
  \begin{eqnarray}\label{A-u2-opt3.}
     \frac{d}{dt}\left\|\partial^{\alpha}u_2(t)\right\|^2+\left\|\partial^{\alpha}u_2\right\|^2_{\nu}
     &\lesssim & (1+t)^{-|\alpha|-\frac{3}{2}}\CE^2_{\infty}(t)+\left\|\partial^{\alpha}\nabla_xu\right\|_\nu^2
    +\left\|\nabla_x\phi(t)\right\|_{L^{\infty}}^{\frac{2-\gamma}{1-\gamma}}
    \left\|\langle\xi\rangle\partial^{\alpha}u_2\right\|^2\nonumber\\
    &\lesssim&(1+t)^{-|\alpha|-\frac{3}{2}}\left(\epsilon_0+\CE^2_{\infty}(t)\right)
    +(1+t)^{-\frac{5}{4}}\CE^{\frac{5}{8}}_{\infty}(t)
    (1+t)^{-|\alpha|-\frac{1}{2}}\CE^{\frac{1}{2}}_{\infty}(t)\\
    &\lesssim &(1+t)^{-|\alpha|-\frac{3}{2}}\left(\epsilon_0+\CE^2_{\infty}(t)\right).\nonumber
\end{eqnarray}
Here we have used the fact that $\left\|\partial^{\alpha}\nabla_xu\right\|_\nu^2\leq \left\|\partial^{\alpha}\nabla_xu\right\|^2$ and Lemma \ref{non-weight}.

As in \cite{Duan-Liu-CMP-2013, Strain-KRM-2012, Strain-Guo-ARMA-2008}, we split the velocity space $\R^3_{\xi}$ into two parts
\begin{equation*}
  E(t)=\{\langle\xi\rangle^{-\gamma}\leq t^{1-p}\},\quad
  E^c(t)=\{\langle\xi\rangle^{-\gamma}> t^{1-p}\},
\end{equation*}
where $\frac{3}{4}<p<1$.
Then it follows that
\begin{equation*}
     \frac{d}{dt}\left\|\partial^{\alpha}u_2(t)\right\|^2+\lambda t^{-1+p}\left\|\partial^{\alpha}u_2\right\|^2
    \lesssim (1+t)^{-|\alpha|-\frac{3}{2}}\left(\epsilon_0+\CE^2_{\infty}(t)\right)
    +t^{-1+p}\left\|\chi_{E^c(t)}\partial^{\alpha}u_2(t)\right\|^2,
\end{equation*}
which implies
\begin{eqnarray}\label{A-u2}
   \left\|\partial^{\alpha}u_2(t)\right\|^2&\lesssim& e^{-t^p}\left\|\partial^{\alpha}u_2(0)\right\|^2
   +\left(\epsilon_0+\CE^2_{\infty}(t)\right)\int_0^te^{-(t^p-\tau^p)}(1+\tau)^{-|\alpha|-\frac{3}{2}}d\tau\nonumber\\
   &&+\int_0^te^{-(t^p-\tau^p)}\tau^{-1+p}\left\|\chi_{E^c(\tau)}\partial^{\alpha}u_2(\tau)\right\|^2d\tau.
 \end{eqnarray}
Next we deal with the last two terms in the right hand side of the
above estimate. To do so, we first use (\ref{A-u2-opt3.}) to see
that
\begin{equation*}
  \left\|\partial^{\alpha}u_2(t)\right\|^2\lesssim\epsilon_0+\CE^2_{\infty}(t),
\end{equation*}
on the other hand, noticing that
\begin{equation*}
  t^{r_{|\alpha|}}=\left(t^{1-p}\right)^{\frac{|\alpha|+3/2}{1-p}}\lesssim
  \langle\xi\rangle^{-\frac{2(|\alpha|+3/2)}{1-p}\frac{\gamma}{2}},
\end{equation*}
on $E^c(t)$, we can get by choosing $l_1=\frac{2(|\alpha|+3/2)}{1-p}$ that
\begin{equation*}
  \left\|\chi_{E^c(t)}\partial^{\alpha}u_2(t)\right\|^2\lesssim t^{-|\alpha|-\frac{3}{2}}\|w_{-l_1}\partial^{\alpha}u_2(t)\|^2\lesssim t^{-|\alpha|-\frac{3}{2}}\epsilon_0.
\end{equation*}
By the combination of above two estimates, one has
\begin{equation}\label{A-u2-high}
  \left\|\chi_{E^c(t)}\partial^{\alpha}u_2(t)\right\|^2\lesssim (1+t)^{-|\alpha|-\frac{3}{2}}\left(\epsilon_0+\CE^2_{\infty}(t)\right).
\end{equation}
Substituting \eqref{A-u2-high} into \eqref{A-u2}, we finally get our desired estimates \eqref{A-u2-opt1.}.

For the proof of \eqref{A-u2-opt2.}, we can obtain by multiplying \eqref{mixed-equ.}
(with $|\alpha|=N-2,~|\beta|=1$) by
$\nabla^{N-2}_x\nabla_{\xi}u_2(t)$ and then integrating the resulting identity with respect to $x$ and $\xi$ over ${\mathbb{R}}^3\times{\mathbb{R}}^3$ that
\begin{eqnarray*}
  \frac{d}{dt}\left\|\nabla^{N-2}_x\nabla_{\xi}u_2(t)\right\|^2
  +\left\|\nabla^{N-2}_x\nabla_{\xi}u_2(t)\right\|^2_{\nu}
  &\lesssim&\left\|\nabla^{N-2}u_2\right\|^2_{\nu}+\left\|\nabla^{N-1}_x u\right\|^2+(1+t)^{-\left(N-\frac{1}{2}\right)}\left(\epsilon_0+\CE^2_{\infty}(t)\right)\\
  &\lesssim&(1+t)^{-\left(N-\frac{1}{2}\right)+1-p}\left(\epsilon_0+\CE^2_{\infty}(t)\right).
 \end{eqnarray*}
Here we have used
%the fact for the control of of the transport term
%\begin{equation*}
%  \left\|\langle\xi\rangle^{-\frac{\gamma}{2}}\nabla_x^{N-1}u_2\right\|^2\lesssim\left\|w_{-\ell}\nabla_x^{N-1}u_2\right\|^2
%  \lesssim (1+t)^{-(N-\frac{1}{2})}\left(\epsilon_0+\CE^2_{\infty}(t)\right),
%\end{equation*}
 \eqref{A-u2-opt1.} and Lemma \ref{non-weight}. Then the same argument used to deduce \eqref{A-u2-opt1.} yields \eqref{A-u2-opt2.}.

 Finally, we show that the estimate (\ref{A-u2-wei2.}) is true. By the definition of $\CE_{\infty}(t)$ and \eqref{A-u2-opt2.}, we use the H\"{o}lder inequality to obtain
\begin{eqnarray*}
\left\|\langle\xi\rangle^{-\frac{\gamma}{2}l_2}\nabla^{N-2}_x\nabla_{\xi}u_2\right\|^2
  &\lesssim&\left\|w_{1-\ell}\nabla^{N-2}\nabla_{\xi}u_2\right\|^{2\frac{l_2}{\ell-1}}
  \left\|\nabla^{N-2}\nabla_{\xi}u_2\right\|^{2(1-\frac{l_2}{\ell-1})}\\
  &\lesssim&(1+t)^{-\left(N-\frac{3}{2}\right)\frac{l_2}{\ell-1}-\left[\left(N-\frac{1}{2}\right)-2(1-p)\right]
  \left(1-\frac{l_2}{\ell-1}\right)}
  \left(\epsilon_0+\CE^2_{\infty}(t)\right)\\
   &\lesssim&(1+t)^{-N+1}\left(\epsilon_0+\CE^2_{\infty}(t)\right),
 \end{eqnarray*}
since $\frac{3}{4}<p<1$, $l_2>N+\frac{1}{2}$, and $\ell\geq \frac{4p-2}{4p-3}l_2+1$.
\end{proof}
We now state the temporal decay estimates on the solution operator of the linearized system related to \eqref{u} whose proof can be found in \cite{Xiao-Xiong-Zhao-JDE-2013}, cf. Lemma 2.5 of \cite{Xiao-Xiong-Zhao-JDE-2013}.
\begin{lemma}[\cite{Xiao-Xiong-Zhao-JDE-2013}]\label{lem.-decay} Let $-3<\gamma<0,~|\alpha|\geq0,~l_0\geq0,~l_2>|\alpha|+\frac{3}{2}$. Then, the evolution operator $e^{tB}$ satisfies
\begin{eqnarray*}
\left\|\langle\xi\rangle^{-\frac{\gamma}{2}l_0}\partial^{\alpha}e^{tB}u_0\right\|
+\left\|\partial^{\alpha}\nabla_x\triangle_x^{-1}{\bf P}_0e^{tB}u_0\right\|
\lesssim(1+t)^{-\frac{1}{2}\left(\frac{1}{2}+|\alpha|\right)}
\left(\left\|\langle\xi\rangle^{-\frac{\gamma}{2}(l_0+l_2)}\partial^{\alpha}u_0\right\|
+\left\|\langle\xi\rangle^{-\frac{\gamma}{2}(l_0+l_2)}u_0\right\|_{Z_1}\right)
\end{eqnarray*}
for any $t\geq0$. Moreover, it holds for some sufficiently small positive constant $\lambda>0$ that
\begin{equation*}
\partial_tE_{l_0}\left(\widehat{u}(t,k )\right)+\frac{\lambda |k|^2}{1+|k|^2}E_{l_0-1}\left(\widehat{u}(t,k )\right)\leq0,
\end{equation*}
where $E_{l_0}\left(\widehat{u}(t,k,\xi)\right)\sim \left|\langle\xi\rangle^{-\frac{\gamma}{2}l_0}
\widehat{u}(t,k)\right|^2_{L^2_{\xi}}+|k|^{-2}|\hat{a}|^2$ is a time-frequency functional and $\widehat{u}(t,k,\xi)=\mathcal {F}[u](t,x,\xi)$ denotes the Fourier transform of $f$ with respect to the variable
$x$ and $k$ denotes the corresponding frequency variable. Furthermore, if we use the notation $e^{(t-\tau)B}G(\tau,x,\xi)$ to denote the unique solution of the following Cauchy problem
\begin{equation*}
  \begin{cases}
  u_t-Bu=0,\\
  \Delta_x\phi={\displaystyle\int_{{\mathbb{ R}}^3}}{\bf M}^{1/2}ud\xi,\\
  u(\tau,x,\xi)=G(\tau,x,\xi), \quad 0\leq \tau<t,
  \end{cases}
\end{equation*}
then it holds for any $t>\tau>0$ and any $l_0\geq0$ that
\begin{equation}\label{G-decay}
\left|\langle\xi\rangle^{-\frac{\gamma}{2}l_0}e^{(t-\tau)B}G(\tau,k,\xi)\right|_{L_{\xi}^2}^2
=\left|\langle\xi\rangle^{-\frac{\gamma}{2}l_0}\widehat{u}(t,k,\xi)\right|^2_{L^2_{\xi}}
\lesssim\left\{1+\frac{|k|^2}{1+|k|^2}(t-\tau)\right\}^{-l_2}
\left|\langle\xi\rangle^{-\frac{\gamma}{2}(l_0+l_2)}\widehat{G}(\tau,k,\xi)\right|^2_{L^2_{\xi}}.
\end{equation}
Here $\widehat{G}(t,k,\xi)=\mathcal{F}\{G(t,x,\xi)\}$ denotes the Fourier transform of $G(t,x,\xi)$ with respect to the space variable $x$.
\end{lemma}

\bigbreak
\noindent
{\bf Acknowledgement.} The authors would like to acknowledge valuable discussions with Dr. Renjun Duan of the Chinese University of Hong Kong. This work was supported by the Fundamental Research Funds for the Central Universities and three grants from the National Natural Science Foundation of China under contracts 10925103, 11271160, and 11261160485, respectively.

%  \bibliography{bibfile}

\end{document}